\documentclass{amsart}
\usepackage[all]{xy}
\usepackage{amsmath,amsfonts,amssymb,amsthm,epsfig,amscd,comment,latexsym,psfrag}
\usepackage{nicefrac,xspace,tikz}
\usetikzlibrary{arrows}

\newtheorem{prop}{Proposition}
\newtheorem{theorem}{Theorem}
\newtheorem{cor}{Corollary}
\newtheorem{conj}{Conjecture}
\newtheorem{lemma}{Lemma}
\newtheorem{Remark}{Remark}

\theoremstyle{definition}

\def\sP{\mathcal{P}}
\def\supp{\mbox{supp}}
\def\Hilb{\mbox{Hilb}}
\def\tr{\mbox{tr}}

\def\fP{\mathsf{P}}
\def\fQ{\mathsf{Q}}
\def\O{\mathcal{O}}
\def\catC{\mathcal{C}}
\def\la{\langle}
\def\ra{\rangle}

\def\k{\Bbbk}

\def\gl{\mathfrak{gl}}
\def\sl{\mathfrak{sl}}
\def\H{\mathcal{H}}
\def\h{\mathfrak{h}}
\def\P{\mathbb P}
\def\R{\mathbb R}

\def\Z{\mathbb Z}
\def\N{\mathbb N} 
\def\C{\mathbb C}
\def\g{\mathfrak{g}}
\def\G{\Gamma}

\def\Ext{\mathrm{Ext}}
\def\End{\mathrm{End}}
\def\Sym{\mathrm{Sym}}
\def\adj{\mathrm{adj}}
\def\id{\mathrm{id}}

\def\dmod{{\mathrm{-mod}}} 
\def\dgmod{{\mathrm{-gmod}}} 
\newcommand{\Hom}{{\rm Hom}}

\newcommand{\Ind}{{\rm{Ind}}}
\newcommand{\Res}{{\rm{Res}}}

\psfrag{Qpm}{$Q_{+-}$} \psfrag{Qmp}{$Q_{-+}$}
\psfrag{Ione}{$\quad\mathbf{1}$}
\psfrag{Xiota}{$\iota$}\psfrag{Xiotap}{$\iota'$}
\psfrag{Xgi}{$g_i$}\psfrag{Xgim1}{$g_i^{-1}$}
\psfrag{Xsumi}{\large{$\sum_{i\in I}$}}
\psfrag{Xbeta}{$\overline{\beta}$}
\psfrag{Xgj}{$g_j$} \psfrag{ifinotj}{\large{if $i\not= j$}}
\psfrag{XSnLm}{$S^n_-\otimes \Lambda^m_+$} 
\psfrag{XLmSn}{$\Lambda^m_+\otimes S^n_-$}
\psfrag{XLm1Sn1}{$\Lambda^{m-1}_+\otimes S^{n-1}_-$}
\psfrag{XSumkb}{$-\sum_{b=0}^{k-2} (k-b-1)$}

\title{Heisenberg categorification and Hilbert Schemes}

\begin{document} 
\setcounter{tocdepth}{1}

\author{Sabin Cautis}
\email{scautis@math.columbia.edu}
\address{Department of Mathematics\\ Columbia University \\ New York, NY}

\author{Anthony Licata}
\email{amlicata@math.stanford.edu}
\address{Department of Mathematics\\ Stanford University \\ Stanford, CA}

\begin{abstract} 
Given a finite subgroup $\G \subset SL_2(\C)$ we define an additive 2-category $\H_\G$ whose Grothendieck group is isomorphic to an integral form $\h_\G$ of the Heisenberg algebra. We construct an action of $\H_\G$ on derived categories of coherent sheaves on Hilbert schemes of points on the minimal resolutions $\widehat{\C^2/\G}$. 
\end{abstract}

\maketitle
\tableofcontents

\section{Introduction} 

The well-documented relationship between the Heisenberg algebra and Hilbert schemes of points on a surface $S$ involves an action of the Heisenberg algebra on the cohomology $\oplus_n H^*(\Hilb^n(S))$, \cite{N1,G}. Inspired by other constructions from geometric representation theory, one expects this Heisenberg action on cohomology to lift to algebraic $K$-theory, though in general lifting this action is not completely straightforward (see \cite{ScV} and \cite{FT} for $K$-theoretic constructions when the surface is $\C^2$).  Of particular interest in representation theory is the case when the surface $S= \widehat{\C^2/\G}$ is the minimal resolution of the quotient of $\C^2$ by a non-trivial finite subgroup $\G \subset SL_2(\C)$; in this case one expects a Heisenberg action on $K$-theory to be the basic building block for a $K$-theoretic construction of the basic representation of the corresponding affine Lie algebra.

After $K$-theory, however, one can investigate yet another level of structure by considering the derived categories of coherent sheaves $\oplus_n DCoh(\Hilb^n(S))$ on these Hilbert schemes. Then one expects a Heisenberg algebra action on $K$-theory to be the shadow a richer categorical Heisenberg action on these triangulated categories. 
 
More precisely, the Hilbert schemes $\{\Hilb^n(S)\}_{n \in \N}$ naturally form a 2-category $\Hilb(S)$:
\begin{enumerate}
\item The objects of $\Hilb(S)$ are the natural numbers $\N$;
\item The 1-morphisms from $m$ to $n$ are compactly supported objects in $DCoh(\Hilb^m(S) \times \Hilb^n(S))$.  Composition of 1-morphisms is given by convolution.
\item For $\sP_1, \sP_2 \in DCoh(\Hilb^m(S) \times \Hilb^n(S))$, the space of 2-morphisms from $\sP_1$ to $\sP_2$ is the space $\Ext^*(\sP_1, \sP_2)$.  The composition of 2-morphisms is the natural composition of $\Ext$s.
\end{enumerate}
So to give a categorical Heisenberg action one should define both a Heisenberg 2-category $\H_S$ and a 2-functor $\H_S \rightarrow \Hilb(S)$. The Grothendieck group of $\H_S$ should be isomorphic to the Heisenberg algebra so that passing to $K$-theory one gets a Heisenberg action on $\oplus_n K_0(\Hilb^n(S))$. In this paper we construct such a 2-category $\H_S = \H_\G$ and a 2-functor $\H_\G \rightarrow \Hilb(S)$ when $S = \widehat{\C^2/\G}$ is the minimal resolution of the quotient of $\C^2$ by a non-trivial finite subgroup $\G \subset SL_2(\C)$ (Theorems \ref{thm:main1} and \ref{thm:main2}). In particular, we obtain a (quantum) Heisenberg algebra action on $\oplus_n K_0(\Hilb^n(S))$.  (We also note here that in the special case when $S=\C^2$ a natural $\C^\times$ action together with localization techniques have already been used to define Heisenberg algebra actions on localized equivariant K-theory, \cite{FT,ScV}). 

\subsection{Analogy with Kac-Moody Lie algebra actions} 

There are clear parallels between our story for Heisenberg algebras $\h_\G$ and the categorification of Kac-Moody algebras $U(\g)$. In \cite{N2} and \cite{N3} Nakajima constructed actions of the enveloping algebra $U(\g)$ on the cohomology and K-theory of quiver varieties, generalizing earlier work of Ginzburg for $U(\sl_n)$ \cite{CG}. On the other hand, Khovanov-Lauda \cite{KL1,KL2,KL3} define a 2-category $\mathcal{U}(\g)$ whose Grothendieck group is isomorphic to (an integral version of) the quantized enveloping algebra. There is also independent and very similar work of Rouquier \cite{R} in the same direction. 

In \cite{CKL2} the authors, together with Joel Kamnitzer, construct a ``geometric categorical $\g$ action'' on the derived category of coherent sheaves on quiver varieties. This action conjecturally induces an action of the 2-categories $\mathcal{U}(\g)$ on the derived categories of quiver varieties. The 2-categories $\mathcal{U}(\g)$ and these actions on the derived categories of quiver varieties are analogous to our Heisenberg 2-categories $\H_\G$ and their action on the derived categories of Hilbert schemes. 

\subsection{The 2-category $\H_\G$}

By the McKay correspondence, isomorphism classes of finite subgroups $\G \subset SL_2(\C)$ are parametrized by simply-laced affine Dynkin diagrams (i.e. diagrams of type $\widehat{A}_n$, $\widehat{D}_n$, $\widehat{E}_6$, $\widehat{E}_7$ and $\widehat{E}_8$).  Each of these affine Dynkin diagrams, in turn, gives rise to a presentation for a Heisenberg algebra $\h_\G$, which we describe in Section \ref{sec:h_Gdef}, and a presentation of a 2-category $\H_\G$, which we define in Section \ref{sec:catH}.  The 2-morphisms in $\H_\G$ are defined using a planar graphical calculus reminiscent of the graphical calculus used by Khovanov-Lauda \cite{KL2} in their categorification of quantum groups. Our first theorem (\ref{thm:main1}) says that the Grothendieck group $K_0(\H_\G)$ is isomorphic to the Heisenberg algebra $\h_\G$.  In particular, the indecomposable 1-morphisms of $\H_\G$ descend to a ``canonical basis" of the Heisenberg algebra.

\subsection{Actions of $\H_\G$} 

A direct relationship between the finite group $\G \subset SL_2(\C)$ and the Heisenberg algebra 
$\h_\G$ can be formulated in two closely related ways (one algebraic and the other geometric).  In the algebraic setting one can consider the wreath products $\G^n \rtimes S_n$ of $\G$ with the symmetric group $S_n$.  Then, following \cite{FJW1}), one constructs the basic representation of $\h_\G$ on the Grothendieck groups of the module categories $\C[\G^n \rtimes S_n] \dmod$. Geometrically, following \cite{N1,G}, one considers the Hilbert schemes $\Hilb^n(\widehat{\C^2/\G})$ and constructs a representation of $\h_\G$ on their cohomology. 

We will work in a setting that shares features with both the algebraic and geometric constructions above by considering directly the bounded derived categories $D(A_n^\G \dgmod)$ of finitely generated graded $A_n^\G$-modules. Here $A_n^\G$ is the algebra 
$$A_n^\G := [(\Sym^*(V^\vee) \rtimes \G) \otimes (\Sym^*(V^\vee) \rtimes \G) \otimes \dots \otimes (\Sym^*(V^\vee) \rtimes \G)] \rtimes S_n$$
which inherits the natural grading from $\Sym^*(V^\vee)$ (here $V = \C^2$).   

We recall why $D(A_n^\G \dgmod)$ is equivalent to derived categories of Hilbert schemes in Section \ref{sec:geometry} and explain the relationship to $\C[\G^n \rtimes S_n] \dmod$ in Section \ref{sec:wreath}. Our second main theorem (Theorem \ref{thm:main2}) constructs a natural action of the Heisenberg category $\H_\G$ on $\oplus_n D(A_n^\G \dgmod)$.

The Grothendieck group $K_0(A_n^\G \dmod)$ is isomorphic to $K_0(\C[\G]^n  \rtimes S_n \dmod)$. On the other hand, $K_0(A_n^\G \dmod) \cong K_0(\mbox{Hilb}^n(\widehat{\C^2/\G}))$ is isomorphic to the cohomology of $\mbox{Hilb}^n(\widehat{\C^2/\G})$. Thus the category $\oplus_n D(A_n^\G \dmod)$ gives a common categorification of the spaces used by Nakajima, Grojnowski and Frenkel-Jing-Wang.

\subsection{Organization}

\begin{itemize}
\item Section \ref{sec:hei} defines the (quantum) Heisenberg algebra $\h_\G$. Our choice of generators for $\h_\G$ differs somewhat from the choices made elsewhere in the literature. 
\item Section \ref{sec:catH} defines the 2-category $\H_\G$ which categorifies $\h_\G$. We also define a map from $\h_\G$ to the Grothendieck group of $\H_\G$. This map is shown to be an isomorphism in Section \ref{sec:K-theory} (Theorem \ref{thm:main1}).
\item Sections \ref{sec:action} and \ref{sec:action2} are concerned with defining an action of $\H_\G$ on $\oplus_n D(A_n^\G \dmod)$. 
\item In Section \ref{sec:graphical2} we study a slightly different 2-category $\H^\G$ using an alternative simplified graphical calculus. $\H^\G$ is ``Morita equivalent'' to $\H_\G$ in the sense that the spaces of 2-morphisms in $\H^\G$ and $\H_\G$ are Morita equivalent (subsequently the 2-representation theories of $\H^\G$ and $\H_\G$ are equivalent). We construct a natural functor $\eta: \H^\G \rightarrow \H_\G$ which induces an isomorphism at the level of Grothendieck groups (so both $\H^\G$ and $\H_\G$ categorify $\h_\G$). The 2-category $\H^\G$ is of independent interest but another motivation for introducing it is to facilitate  the graphical computations needed to prove Theorem \ref{thm:main1}. 
\item In section \ref{sec:koszul} we sketch a second, simpler (non-derived) action of $\H_\G$, related to the action of Section \ref{sec:action} by Koszul duality. 
\item Section \ref{sec:open} contains a discussion of currently unanswered questions and futher directions.
\item The appendix: most of the paper assumes the Dynkin diagram associated to $\G \subset SL_2(\C)$ is simply-laced, which simplifies notation but fails to cover the case $\G=\Z_2$. In Section \ref{sec:z2} we collect the required definitions for $\G=\Z_2$ (with these modifications the theorems of the paper hold for any nontrivial $\G$). The definition of the Heisenberg category for the trivial subgroup of $SL_2(\C)$ is of interest in its own right, but is not discussed directly in this paper.
\end{itemize}

\subsection{Acknowledgements}
The authors would like to thank Mikhail Khovanov for several useful conversations and for giving us an early version of his paper \cite{K}. Much of this paper is inspired by his work. We would also like to thank Igor Frenkel for sharing with us his unpublished notes with Khovanov and Malkin on categorification and vertex operator algebras \cite{FKM}, and Alistair Savage for useful comments regarding an earlier version of this paper. This project began when both authors were at the program ``Homology Theory of Knots and Links'' at MSRI in 2010. We thank MSRI for its hospitality and great working environment. S.C. was supported by National Science Foundation Grant 0801939/0964439.

\section{Quantum Heisenberg algebras}\label{sec:hei}

\subsection{The McKay correspondence}

Fix an algebraically closed field $\k$ of characteristic zero. Let $\G \subset SL_2(\k)$ denote a finite subgroup (in our discussion this can include $\G = \k^\times$ the Cartan subgroup of diagonal matrices). For notational convenience, the body of the paper assumes that $\G \neq \Z_2$. We deal with this extra case in Appendix \ref{sec:z2}.

Denote by $V$ the standard two dimensional representation of $\G$. Under the McKay correspondence the finite subgroup $\G$ corresponds to an affine Dynkin diagram with vertex set $I_\G$ and edge set $E_\G$.By definition each vertex $i \in I_\G$ is indexed by an irreducible representation $V_i$ of $\G$ and two vertices $i, j \in I_\G$ are joined by $k$ edges where $k$ is the number of times $V_j$ appears as a direct summand of $V_i \otimes V$. Notice that since $V$ is self dual this relation is symmetric in $i$ and $j$. 

For instance, when $G = \Z/n\Z$ then we get the affine Dynkin diagram $\widehat{A}_{n-1}$, which is an $n$ cycle. When $\G = \k^\times$ the associated diagram is the Dynkin diagram $\widehat{A}_\infty$ of the Lie algebra $\sl_\infty$.

Define the \emph{quantum Cartan matrix} $C_\G$ to be the matrix with entries $\la i,j \ra$ indexed by $i,j \in I_\G$ given by
$$ \la i,j \ra =
\begin{cases}
        t + t^{-1} & \text{ if } i=j \\
        -1 & \text{ if } i \ne j \text{ are connected by an edge } \\
	0 & \text{ if } i \ne j \text{ are not connected by an edge }
\end{cases} $$
Note that at $t=1$ the matrix $C_\G$ becomes the extended Cartan matrix of type ADE (this is the McKay correspondence between nontrivial finite subgroups $\G \subset SL_2(\C)$ and simply-laced affine Dynkin diagrams).

\subsection{The quantum Heisenberg algebra associated to $\G$}\label{sec:h_Gdef}

We define the Heisenberg algebra $\h_\G$ associated to $\G$ to be the unital $\k[t,t^{-1}]$ algebra with generators $p_i^{(n)}, q_i^{(n)}$ for $i \in I_\G$ and integers $n \ge 0$ and relations
\begin{equation}\label{rel 1}
p_i^{(n)} p_j^{(m)} = p_j^{(m)} p_i^{(n)} \text{ for all } i,j \in I_\G,
\end{equation}
\begin{equation}\label{rel 2}
q_i^{(n)}q_j^{(m)} = q_j^{(m)} q_i^{(n)} \text{ for all } i,j \in I_\G,
\end{equation}
\begin{equation}\label{rel 3}
q_i^{(n)} p_i^{(m)} = \sum_{k \ge 0} [k+1] p_i^{(m-k)} q_i^{(n-k)} \text{ for all } i \in I_\G,
\end{equation}
\begin{equation}\label{rel 4}
q_i^{(n)} p_j^{(m)} = p_j^{(m)} q_i^{(n)} + p_j^{(m-1)} q_i^{(n-1)} 
\text{ for all } i \neq j \in I_\G \text{ with } \langle i,j \rangle = -1,
\end{equation}
\begin{equation}\label{rel 5}
q_i^{(n)} p_j^{(m)} = p_j^{(m)} q_i^{(n)}
\text{ for all } i\neq j \in I_\G \text{ with } \langle i,j\rangle = 0.
\end{equation}
Here $[k+1] = t^{-k} +t^{-k+2} + \dots + t^{k-2} + t^k$ denotes the quantum integer, and in the above relations we have set $p_i^{(0)} = q_j^{(0)} = 1$ and $p_i^{(k)} = q_i^{(k)} = 0$ when $k <0$ (thus the summations in the above relations are all finite.)

\subsubsection{The usual Heisenberg presentation}
Most mathematical literature about the Heisenberg algebra uses a slightly different presentation than the one above.  In level one, which is the case of interest for this paper, the standard presentation of the Heisenberg algebra is as a unital $\k[t,t^{-1}]$ algebra generated by $a_i(n)$, for $i \in I_\G$ and $n \in \Z-0$.  The relations are
\begin{equation}\label{rel:as}
[a_i(m), a_j(n)] = \delta_{m,-n} [n \la i,j \ra] \frac{[n]}{n}.
\end{equation}
Since the Dynkin diagram with vertex set $I_\G$ is of affine type this algebra is sometimes called the quantum toroidal Heisenberg algebra (since, in this case, the Heisenberg algebra is a subalgebra of the corresponding level one quantum toroidal algebra).  By omitting the affine node from consideration, one recovers a smaller Heisenberg algebra which is a subalgebra of the corresponding level one quantum affine algebra.

\begin{Remark}
Sometimes, in the literature, relation (\ref{rel:as}) above appears with a minus sign on the right hand side. However, this change presentation does not alter the isomorphism type of the algebra, since replacing $a_i(m)$ by $-a_i(m)$ for $m > 0$ takes one presentation to the other. 
\end{Remark}

As will be shown in Lemma \ref{lem:new} below, the relationship between the generators $a_i(m)$ and $p_i^{(m)}, q_i^{(m)}$ is given by the following generating functions
\begin{equation}\label{eq:gens1}
\exp \left( \sum_{m \geq 1} \frac{a_i(-m)}{[m]} z^m \right) = \sum_{n \geq 0} p_i^{(n)} z^n
\text{ and }
\exp \left( \sum_{m \geq 1} \frac{a_i(m)}{[m]} z^m \right) = \sum_{n \geq 0} q_i^{(n)} z^n.
\end{equation}
For example, the first few terms in the above expansion:
$$p_i^{(0)} = 1 = q_i^{(0)},$$
$$p_i^{(1)} = a_i(-1), \quad  q_i^{(1)} = a_i(1),$$
$$p_i^{(2)} = \frac{1}{[2]} a_i(-2) + \frac{1}{2} a_i(-1)^2, \quad 
q_i^{(2)} = \frac{1}{[2]} a_i(2) + \frac{1}{2} a_i(1)^2, \quad \text{etc.}$$

The expressions in Equation \ref{eq:gens1} are known as ``halves of vertex operators'' and they appear naturally in vertex operator constructions of representation of quantum affine and toroidal algebras.  For our categorical considerations, we find the generators $p_i^{(n)}, q_i^{(n)}$ more natural than the generators $a_i(m)$.

\begin{lemma}\label{lem:new}
Relation (\ref{rel:as}) corresponds to relations (\ref{rel 1})-(\ref{rel 5}) under the identification given by the vertex generating functions (\ref{eq:gens1}). 
\end{lemma}
\begin{proof}
We prove relations (\ref{rel 3}) and (\ref{rel 4}) above, the others being completely straightforward to check. 

To prove relation (\ref{rel 3}) let us denote 
$$A(z) := \sum_{\ell \geq 1} \frac{a_i(-\ell)}{[\ell]} z^\ell \text{ and }
B(w) := \sum_{\ell \geq 1} \frac{a_i(\ell)}{[\ell]} w^\ell.$$
Then $q_i^{(n)} p_i^{(m)} = [z^m w^n] \exp(B) \exp(A)$, where $[z^mw^n]$ means taking the coefficient of tthe polynomial $z^mw^n$ in the subsequent power series. 
Now, using relation (\ref{rel:as}), it follows that 
\begin{eqnarray*}
[A,B] 
&=& \sum_{\ell \ge 1} \frac{[2\ell]}{\ell [\ell]} w^\ell z^\ell \\
&=& \sum_{\ell \ge 1} \frac{(t^\ell + t^{-\ell})}{\ell} w^\ell z^\ell \\
&=& \log \left( (1-twz)(1-t^{-1}wz) \right).
\end{eqnarray*}
It follows that $[A,B]$ commutes with $A$ and $B$ and hence that
$$\exp(B)\exp(A) = \exp(-[A,B])\exp(A)\exp(B)$$
(see for instance Lemma 9.43 of \cite{N4}). Thus
\begin{eqnarray*}
[z^m w^n] \exp(B) \exp(A) 
&=& [z^m w^n] \frac{1}{(1-twz)(1-t^{-1}wz)} \exp(A) \exp(B) \\
&=& [z^m w^n] \sum_{\ell \ge 0} [\ell+1] (wz)^\ell \exp(A) \exp(B) \\
&=& \sum_{k \ge 0} [k+1] p_i^{(m-k)} q_i^{(n-k)}
\end{eqnarray*}
which is what we needed to prove. 

In the case of relation (\ref{rel 4}) let us denote
$$A(z) := \sum_{\ell \geq 1} \frac{a_j(-\ell)}{[\ell]} z^\ell \text{ and }
B(w) := \sum_{\ell \geq 1} \frac{a_i(\ell)}{[\ell]} w^\ell.$$
Then $q_i^{(n)} p_j^{(m)} = [z^m w^n] \exp(B) \exp(A)$. However, this time 
$$[A,B] = \sum_{\ell \ge 1} \frac{[-1]}{\ell} w^\ell z^\ell = - \log (1-wz)$$
and so 
\begin{eqnarray*}
[z^m w^n] \exp(B) \exp(A) 
&=& [z^m w^n] (1-wz) \exp(A) \exp(B) \\ 
&=& p_j^{(m)} q_i^{(n)} + p_j^{(m-1)} q_i^{(n-1)}
\end{eqnarray*}
which is what we needed to prove. 
\end{proof}

\subsubsection{The transposed generators}
There is an alternative generating set of $\h_\G$ given by elements $p_i^{(1^n)}$ and $a_i^{(1^n)}$. These are defined by halves of vertex operators which are similar to those of (\ref{eq:gens1}):\begin{equation}\label{eq:gens2}
\exp \left( - \sum_{m \geq 1} \frac{a_i(-m)}{[m]} z^m \right) = \sum_{n \geq 0} (-1)^n p_i^{(1^n)} z^n
\text{ and }
\exp \left( -\sum_{m \geq 1} \frac{a_i(m)}{[m]} z^m \right) = \sum_{n \geq 0} (-1)^n q_i^{(1^n)} z^n.
\end{equation}
Thus, for example,
$$p_i^{(1^0)} = q_i^{(1^0)} = 1,$$
$$p_i^{(1^1)} = a_i(-1), \quad q_i^{(1^1)} = a_i(1),$$
$$p_i^{(1^2)} = - \frac{1}{[2]} a_i(-2) + \frac{1}{2} a_i(-1)^2, \quad q_i^{(1^2)} = - \frac{1}{[2]} a_i(2) + \frac{1}{2} a_i(1)^2, \quad \text{etc.}$$
The commutation relations among the $p_i^{(1^n)}$ and $q_i^{(1^n)}$ are the same as those between the $p_i^{(n)}$ and $q_i^{(n)}$ (just replace $(n)$ by $(1^n)$ everywhere).  These relations can be proven directly by arguments similar to those of Lemma \ref{lem:new}, though they also follows from the the existence of an automorphism $\psi$ defined in Section \ref{sec:psi}.

Perhaps more interesting are the following relations
\begin{eqnarray*}
p_i^{(m)} p_j^{(1^n)} &=& p_j^{(1^n)} p_i^{(m)} \text{ and } q_i^{(m)} q_j^{(1^n)} = q_j^{(1^n)} q_i^{(m)} \\
q_i^{(1^m)} p_j^{(n)} &=& 
\begin{cases}
p_i^{(n)} q_i^{(1^m)} + [2] p_i^{(n-1)} q_i^{(1^{m-1})} + p_i^{(n-2)} q_i^{(1^{m-2})} \text{ if } i=j \\
\sum_{k \ge 0} p_j^{(n-k)} q_i^{(1^{m-k})} \text{ if } \la i,j \ra = -1 \\
p_j^{(n)} q_i^{(1^m)} \text{ if } \la i,j \ra = 0.
\end{cases} 
\end{eqnarray*}
These relations can be checked directly in the same way as in the proof of lemma \ref{lem:new}. 

\subsubsection{The automorphism $\psi$}\label{sec:psi}
The Heisenberg algebra $\h_\G$ has an involutive automorphism
$$
	\psi: \h_\G\longrightarrow \h_\G, \ \ a_i(k) \mapsto (-1)^{k+1} a_i(k).
$$
In terms of the alternative generators $p_i^{(m)},q_i^{(m)},p_j^{(1^m)},q_i^{(1^m)}$, this has the effect of exchanging $(n)$ and $(1^n)$:
$$
	\psi: p_i^{(n)} \leftrightarrow p_i^{(1^n)}, \ \ q_i^{(n)} \leftrightarrow q_i^{(1^n)}.
$$

\subsubsection{Idempotent modification}

The version of the Heisenberg algebra we use is actually an idempotent modification of $\h_\G$ (this is similar to the appearance of Lusztig's idempotent version of $U_q(\mathfrak{g})$ in the categorification of quantum groups \cite{KL1,KL2,KL3,R}). In the idempotent modification, the unit $1$ is replaced by a collection of orthogonal idempotents $\{1_m\}_{m \in \Z}$, with
$$ 1_{k+m} p_i^{(m)} = 1_{k+m} p_i^{(m)} 1_k = p_i^{(m)} 1_k $$
and
$$ 1_{k-m} q_i^{(m)} = 1_{k-m} q_i^{(m)} 1_k = q_i^{(m)} 1_k.$$
The remaining defining relations in the unital algebra $h_\G$ give, for each $k\in\Z$, a defining relation of the idempotent modification. Namely, take the original relation and add the idempotent $1_k$ at the end of the left and right hand side, for example
\begin{equation*}
q_i^{(n)} p_j^{(m)} 1_k = p_j^{(m)} q_i^{(n)} 1_k \text{ for all } i \neq j \in I_\G \text{ with } \la i,j \ra = 0 \text{ and all } k \in \Z.
\end{equation*}

Note that these relations do not depend essentially on the particular idempotent $1_k$.  As a result, we abuse notation slightly and also denote both the unital algebra and its idempotent modification by $\h_\G$.  

One feature of the idempotent modified $\h_\G$ is that it is a $\k[t,t^{-1}]$-linear category (just as any $\k$-algebra with a collection of idempotents is a $\k$-linear category). The objects of this category are then the integers, while the space of morphisms from $n$ to $m$ is the $\k[t,t^{-1}]$ module $1_n \h_\G 1_m$. Since $\h_\G$ is already a category, its categorification $\H_\G$ defined in Section \ref{sec:catH} will be a 2-category.

\subsection{The Fock space}

Let $\h_\G^- \subset \h_\G$ denote the subalgebra generated by the $q_i^{(n)}1_k$ for all $i \in I_\G$, $k\leq 0$ and $n \geq 0$.  Let $\mbox{triv}_0$ denote the trivial representation of $\h_\G^-$, where $1_0$ acts be the identity and $1_k$ acts by $0$ for $k<0$. The $\h_\G$ module
$ \mathcal{F}_\G = \Ind_{\h_\G^-}^{\h_\G}(\mbox{triv}_0) $
given by inducing the trivial representation from $\h_\G^-$ to $\h_\G$ is called the Fock space representation of $\h_\G$.  

The Fock space $\mathcal{F}_\G$ is naturally isomorphic to the space of polynomials in the commuting variables $\{p_j^{(m)}\}_{j\in I_\G, m\geq 0}$. If we grade the Fock space by declaring $\deg(p_j^{(m)}) = m$ then the idempotent $1_l\in \h_\G$ acts by projecting onto the degree $l$ subspace inside $\mathcal{F}_\G \cong \k[\{p_j^{(m)}\}_{j,m}]$ (in particular, when $l<0$ the idempotent $1_l$ acts by 0.) We will construct categorifications of this representation in sections \ref{sec:action} and \ref{sec:koszul}.

\section{The 2-category $\H_\G$}\label{sec:catH}

\subsection{The algebra $B^\G$}

To define the 2-category $\H_\G$ which categorifies $\h_\G$ we first need to define the algebra $B^\G$ and fix notation involving idempotents in this algebra. 

Since $\G$ acts on $V$ it also acts on the exterior algebra $\Lambda^*(V)$. Let $B^\G := \Lambda^* (V) \rtimes \G$ be their semi-direct product, which contains both $\Lambda^*(V)$ and $\k[\G]$ as subalgebras. This semi-direct product is also called the smash product of $\k[\G]$ and $\Lambda^*(V)$, and is sometimes denoted $\Lambda^*(V)\#\k[\G]$ in Hopf algebra literature, though we prefer the term semi-direct product and avoid the $\#$ notation.  

Explicitly, an element in $B^\G$ is a linear combination of terms $(v,\gamma)$ where $v \in \Lambda^*(V)$ and $\gamma \in \G$. The multiplication in $B^\G$ is given by 
$$(v, \gamma) \cdot (v', \gamma') = (v \wedge (\gamma \cdot v'), \gamma \gamma').$$  
The natural $\Z$ grading on $\Lambda^*(V)$ extends to a grading of $B^\G$ by putting $\k[\G] \subset B^\G$ in degree zero. This makes $B^\G$ into a superalgebra.  We denote the degree of a homogeneous element $b \in B^\G$ by $|b|$.

If we fix a basis $\{v_1, v_2\}$ of $V$ and let $\omega := v_1 \wedge v_2 \in \wedge^2(V)$ then $B^\G$ has a homogeneous basis over $\k$ given by $\{(1,\gamma), (v_1,\gamma), (v_2,\gamma), (\omega,\gamma) \}_{\gamma \in \G}$.

Define a $\k$-linear trace $\tr : B^\G \longrightarrow \k$ by setting
$$ \tr((\omega,\gamma)) =  \delta_{\gamma,1} 1 \text{ and } \tr((1,\gamma)) = \tr((v_1,\gamma)) = \tr((v_2,\gamma)) = 0.$$
The trace $\tr$ is supersymmetric (for any  $a,b \in B^\G$ we have $\tr(ab) = (-1)^{|a||b|}\tr(ba)$) and non-degenerate. This also induces a trace on $\k[\G]$ via $\tr(\gamma) := \tr((\omega,\gamma))$. This corresponds to the usual trace on $\k[\G]$ divided by $|\G|$ (in this way $\tr(1)=1$). 

For a fixed $\k$-basis $\mathcal{B}$ of $B^\G$, let $\mathcal{B}^\vee$ denote the basis of $B^\G$ dual to $\mathcal{B}$ with respect to the associated non-degenerate bilinear form $\langle a,b \rangle := \tr(ab)$. We denote the dual vector of $b \in \mathcal{B}$ by $b^\vee\in \mathcal{B}^\vee$.

\begin{Remark}
We may think of elements of $B^\G$ as $\k[\G]$-module homomorphisms, 
$$B^\G \cong \Hom_{\k[\G]}(\k[\G], \Lambda^*(V) \otimes \k[\G]).$$
The above isomorphism is given by the following composition
\begin{eqnarray*}
B^\G &\cong& \Hom_{B^\G}(B^\G,B^\G) \cong \Hom_{B^\G}(\Ind_{\k[\G]}^{B^\G} \k[\G],B^\G) \cong \Hom_{\k[\G]}(\k[\G],\Res_{B^\G}^{\k[\G]}B^\G) \\ 
&\cong& \Hom_{\k[\G]}(\k[\G], \Lambda^*(V) \otimes \k[\G]).
\end{eqnarray*}
\end{Remark}

\subsubsection{Idempotents}\label{sec:idemp}

Let $V_1, \dots, V_m$ denote the distinct irreducible representations of $\G$. By Maschke's theorem, the group algebra $\k[\G]$ decomposes as a direct product of matrix algebras
$$ \k[\G] \cong M_{n_1}(\k) \times \dots \times M_{n_m}(\k)$$
where the distinct irreducible representations $V_1,\dots, V_m$ are each realized as an irreducible representation of one of the matrix algebras $M_{n_i}(\k)$.  Let $f_1,\dots,f_m$ denote the distinct pairwise orthogonal central idempotents of $\k[\G]$; each $f_i$ is the identity matrix in the matrix algebra $M_{n_i}(\k)$, and
$$ 1 = \sum_i f_i \in \k[\G].  $$
The above is not, in general, a minimal decomposition of $1$ as a sum of orthogonal idempotents, since the idempotents $f_i$ themselves are not minimal if $\dim(V_i) > 1$. For each $i$ and $s=1, \dots, n_i$, let $e_{i,s}$ denote the matrix unit of $M_{n_i}(\k)$ whose $(s,s)$ entry is equal to 1 and whose other entries are 0.  Then the $e_{i,s}$ are minimal orthogonal idempotents in $\k[\G]$, with $$ f_i = \sum_{s=1}^{n_i} e_{i,s}. $$

An important role will be played by the (super)algebra 
$$B_n^\G := (B^\G \otimes \dots \otimes B^\G) \rtimes S_n.$$ 
The above tensor product is understood as the tensor product in the category of superalgebras, and the action of $S_n$ is by superpermutations: if $s_k \in S_n$ is the simple transposition $(k,k+1)$, then
$$ s_k \cdot (b_1 \otimes \dots \otimes b_k \otimes b_{k+1} \otimes \dots \otimes b_n) = (-1)^{|b_k||b_{k+1}|} b_1 \otimes \dots \otimes b_{k+1} \otimes b_k \otimes \dots \otimes b_n.
$$
The degree zero subalgebra $\k[\G^n \rtimes S_n] \subset B_n^\G$ contains all the idempotents of $B_n^\G$. This subalgebra is isomorphic to a direct product of matrix algebras
$$ \k[\G^n \rtimes S_n] = M_{l_1}(\k) \times \dots \times M_{l_s}(\k)$$
with one matrix algebra for each isomorphism class of irreducible $\k[\G^n \rtimes S_n]$-module. Just as the isomorphism classes of irreducible $\k[S_n]$-modules are parametrized by partitions of $n$ the isomorphism classes of irreducible $\k[\G^n \rtimes S_n]$-modules are parametrized by partition-valued functions of $I_\G$ (i.e. $|I_\G|$-tuples of partitions of $n$). 

\subsection{The category $\H'_\G$}\label{subsec:catH'}

We define an additive, $\k$-linear, $\Z$ graded 2-category $\H'_\G$ as follows. The objects of $\H'_\G$ are indexed by the integers $\Z$. The 1-morphisms are generated by $P$ and $Q$, where for each $n$, $P$ denotes a 1-morphism from $n$ to $n+1$ and $Q$ is a 1-morphism from $n+1$ to $n$. Thus a 1-morphism of $\H'_\G$ is a finite composition (sequence) of $P$'s and $Q$'s. The identity 1-morphism of $n$ is denoted $\mathbf{1}$ (the empty sequence).  

\begin{Remark} Technically we should write $P(n)$ and $(n)Q$ to identify the domain and range of $P$ and $Q$. However, the properties of $P$ and $Q$ do not depend on $n$ so we usually omit this extra parameter in order to simplify notation. We could have chosen to define $\H'_\G$ as a monoidal 1-category instead of as a 2-category, but we prefer to consider $\H'_\G$ as a 2-category because this better parallels the story for Kac-Moody Lie algebras in \cite{KL1,KL2,KL3,R}. 
\end{Remark}

The space of 2-morphisms between two 1-morphisms is a $\Z$ graded $\k$-algebra generated by suitable planar diagrams modulo local relations. The diagrams consist of oriented compact one-manifolds immersed into the plane strip $\R \times [0,1]$ modulo isotopies fixing the boundary which preserve the relative height of dots.  (Isotopic diagrams which change the relative height of dots are equal to one another up to a sign, as will be indicated below.)  The source of a given 2-morphism lies on the boundary $y=0$, while the target lies on $y=1$.  Thus 2-morphisms are read bottom to top.

A single upward oriented strand denotes the identity 2-morphism $\id: P \rightarrow P$ while a downward oriented strand denotes the identity 2-morphism $\id: Q \rightarrow Q$.

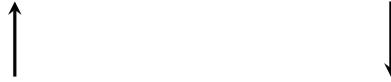
$$\begin{tikzpicture}
[>=stealth] \draw [->](0,0) -- (0,1) [very thick];
[>=stealth] \draw [<-](5,0) -- (5,1) [very thick];
\end{tikzpicture}$$

An upward-oriented line is allowed to carry dots labeled by elements $b \in B^\G$. For example, 
$$
\begin{tikzpicture}[>=stealth]
\draw [->](0,0) -- (0,1) [very thick];
\filldraw [blue](0,.4) circle (2pt);
\draw (0,.4) node [anchor=west] [black] {$b'$};
\filldraw [blue](0,.65) circle (2pt);
\draw (0,.65) node [anchor=west] [black] {$b$};
\end{tikzpicture}
$$
is an element of $\Hom_{\H'_\G}(P,P)$, while
$$
\begin{tikzpicture}[>=stealth]
\draw [->](0,0) -- (1,1) [very thick];
\draw [<-](1,0) -- (0,1) [very thick];
\filldraw [blue](.25,.25) circle (2pt);
\draw (.25,.25) node [anchor=west] [black] {$b$};
\end{tikzpicture}
$$
is an element of $\Hom_{\H'_\G}(PQ,QP)$. Note that the domain of a 2-morphism is specified at the bottom of the diagram and the codomain is specified at the top, and compositions of 2-morphisms are read from bottom to top.  

The local relations we impose are the following. First we have relations involving the movement of dots along the carrier strand. We allow dots to move freely along strands and through intersections:
$$
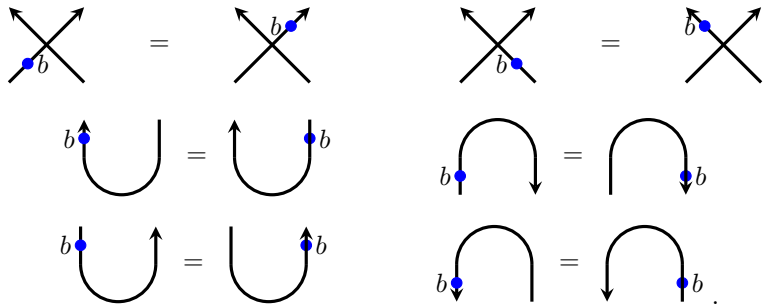
\begin{tikzpicture}[>=stealth]
\draw [->](0,0) -- (1,1) [very thick];
\draw [->](1,0) -- (0,1) [very thick];
\filldraw [blue](.25,.25) circle (2pt);
\draw (.25,.25) node [anchor=west] [black] {$b$};
\draw (2,.5) node{$= $};
\draw [shift={+(1,0)}][->](2,0) -- (3,1) [very thick];
\draw [shift={+(1,0)}][->](3,0) -- (2,1) [very thick];
\filldraw [shift={+(1,0)}][blue](2.75,0.75) circle (2pt);
\draw [shift={+(1,0)}](2.75,.75) node [anchor=east] [black] {$b$};
\draw  [shift={+(1,0)}][shift={+(5,0)}][->](0,0) -- (1,1) [very thick];
\draw  [shift={+(1,0)}][shift={+(5,0)}][->](1,0) -- (0,1) [very thick];
\filldraw  [shift={+(1,0)}][shift={+(5,0)}][blue](.75,.25) circle (2pt);
\draw [shift={+(1,0)}][shift={+(5,0)}] (.75,.25) node [anchor=east] [black] {$b$};
\draw  [shift={+(1,0)}][shift={+(5,0)}](2,.5) node{$= $};
\draw  [shift={+(2,0)}][shift={+(5,0)}][->](2,0) -- (3,1) [very thick];
\draw  [shift={+(2,0)}][shift={+(5,0)}][->](3,0) -- (2,1) [very thick];
\filldraw  [shift={+(2,0)}][shift={+(5,0)}][blue](2.25,0.75) circle (2pt);
\draw [shift={+(2,0)}][shift={+(5,0)}] (2.25,.75) node [anchor=east] [black] {$b$};
\end{tikzpicture}
$$
$$
\begin{tikzpicture}[>=stealth]
\draw (0,0) -- (0,.5) [->][very thick];
\filldraw [blue] (0,.25) circle (2pt);
\draw (0,.25) node [anchor=east] [black] {$b$};
\draw (1,0) -- (1,.5) [very thick];
\draw (0,0) arc (180:360:.5) [very thick];
\draw (1.5,0) node {=};
\draw (2,0) -- (2,.5)[->] [very thick];
\filldraw [blue] (3,.25) circle (2pt);
\draw (3,.25) node [anchor=west] [black] {$b$};
\draw (3,0) -- (3,.5) [very thick];
\draw (2,0) arc (180:360:.5) [very thick];

\draw (5,0) -- (5,-.5) [very thick];
\filldraw [blue] (5,-.25) circle (2pt);
\draw (5,-.25) node [anchor=east] [black] {$b$};
\draw (6,0) -- (6,-.5)[->] [very thick];
\draw (5,0) arc (180:0:.5) [very thick];
\draw (6.5,0) node {=};
\draw (7,0) -- (7,-.5) [very thick];
\filldraw [blue] (8,-.25) circle (2pt);
\draw (8,-.25) node [anchor=west] [black] {$b$};
\draw (8,0) -- (8,-.5) [->][very thick];
\draw (7,0) arc (180:0:.5) [very thick];
\end{tikzpicture}
$$
$$
\begin{tikzpicture}[>=stealth]
\draw (0,0) -- (0,.5) [very thick];
\filldraw [blue] (0,.25) circle (2pt);
\draw (0,.25) node [anchor=east] [black] {$b$};
\draw (1,0) -- (1,.5) [->][very thick];
\draw (0,0) arc (180:360:.5) [very thick];
\draw (1.5,0) node {=};
\draw (2,0) -- (2,.5)[very thick];
\filldraw [blue] (3,.25) circle (2pt);
\draw (3,.25) node [anchor=west] [black] {$b$};
\draw (3,0) -- (3,.5) [->][very thick];
\draw (2,0) arc (180:360:.5) [very thick];

\draw (5,0) -- (5,-.5)[->] [very thick];
\filldraw [blue] (5,-.25) circle (2pt);
\draw (5,-.25) node [anchor=east] [black] {$b$};
\draw (6,0) -- (6,-.5) [very thick];
\draw (5,0) arc (180:0:.5) [very thick];
\draw (6.5,0) node {=};
\draw (7,0) -- (7,-.5)[->] [very thick];
\filldraw [blue] (8,-.25) circle (2pt);
\draw (8,-.25) node [anchor=west] [black] {$b$};
\draw (8,0) -- (8,-.5) [very thick];
\draw (7,0) arc (180:0:.5) [very thick];
\end{tikzpicture}.
$$
The U-turn 2-morphisms are adjunctions making $P$ and $Q$ biadjoint up to a grading shift that will be defined later in this section. Collision of dots is controlled by the algebra structure on $B^\G$:
$$
\begin{tikzpicture}[>=stealth]
\draw (0,0) -- (0,2)[->][very thick];
\filldraw [blue] (0,1) circle (2pt);
\draw (0,1) node [anchor=east] [black] {$b'b$};
\draw (.5,1) node {=};
\draw (.5,1) node {=};
\draw (1,0) -- (1,2)[->][very thick];
\filldraw [blue] (1,.66) circle (2pt);
\draw (1,.66) node [anchor=west] [black] {$b'$};
\filldraw [blue] (1,1.33) circle (2pt);
\draw (1,1.33) node [anchor=west] [black] {$b$};

\draw (5,2) -- (5,0)[->][very thick];
\filldraw [blue] (5,1) circle (2pt);
\draw (5,1) node [anchor=east] [black] {$bb'$};
\draw (5.5,1) node {=};
\draw (5.5,1) node {=};
\draw (6,2) -- (6,0)[->][very thick];
\filldraw [blue] (6,.66) circle (2pt);
\draw (6,.66) node [anchor=west] [black] {$b'$};
\filldraw [blue] (6,1.33) circle (2pt);
\draw (6,1.33) node [anchor=west] [black] {$b$};

\draw [shift = {+(10,0)}] (-1,0) -- (-1,2)[->][very thick];
\filldraw [shift = {+(10,0)}] [blue] (-1,1) circle (2pt);
\draw [shift = {+(10,0)}] (-1,1) node [anchor=east] [black] {$b$};
\draw [shift = {+(10,0)}] (-.5,1) node {+};
\draw [shift = {+(10,0)}] (0,0) -- (0,2)[->][very thick];
\filldraw [shift = {+(10,0)}] [blue] (0,1) circle (2pt);
\draw [shift = {+(10,0)}] (0,1) node [anchor=west] [black] {$b'$};
\draw [shift = {+(10,0)}] (.5,1) node {=};
\draw [shift = {+(10,0)}] (.5,1) node {=};
\draw [shift = {+(10,0)}] (1,0) -- (1,2)[->][very thick];
\draw[shift = {+(10,0)}]  (1,1) node [anchor=west] [black] {$b+b'$};
\filldraw [shift = {+(10,0)}]  [blue] (1,1) circle (2pt);
\end{tikzpicture}
$$ 
Note that dots collide, so to speak, in the direction of the arrow. 

Dots on distinct strands supercommute when they move past one another:
$$
\begin{tikzpicture}[>=stealth]
\draw (0,0) -- (0,2)[->][very thick] ;
\filldraw [blue] (0,.66) circle (2pt);
\draw (0,.66) node [anchor=east] [black] {$b$};
\draw (.5,.5) node {$\dots$};
\draw (1,0)--(1,2)[->] [very thick];
\filldraw [blue] (1,1.33) circle (2pt);
\draw (1,1.33) node [anchor=west] [black] {$b'$};
\draw (3.5,1) node {$= \ (-1)^{|b||b'|}$};
\draw [shift={+(3,0)}](2,0) --(2,2)[->][very thick] ;
\filldraw [shift={+(3,0)}][blue] (2,1.33) circle (2pt);
\draw [shift={+(3,0)}](2,1.33) node [anchor=east] [black] {$b$};
\draw [shift={+(3,0)}](2.5,.5) node {$\dots$};
\filldraw [shift={+(3,0)}][blue] (3,.66) circle (2pt);
\draw [shift={+(3,0)}](3,.66) node [anchor=west] [black] {$b'$};
\draw [shift={+(3,0)}](3,0) -- (3,2)[->][very thick] ;
\end{tikzpicture}
.
$$

In addition to specifying how dots collide and slide we impose the following local relations in the $\G$ graphical calculus:

\begin{equation}\label{eq:rel1}
\begin{tikzpicture}[>=stealth]
\draw [shift={+(7,0)}](0,0) .. controls (1,1) .. (0,2)[->][very thick] ;
\draw [shift={+(7,0)}](1,0) .. controls (0,1) .. (1,2)[->] [very thick];
\draw [shift={+(7,0)}](1.5,1) node {=};
\draw [shift={+(7,0)}](2,0) --(2,2)[->][very thick] ;
\draw [shift={+(7,0)}](3,0) -- (3,2)[->][very thick] ;

\draw (0,0) -- (2,2)[->][very thick];
\draw (2,0) -- (0,2)[->][very thick];
\draw (1,0) .. controls (0,1) .. (1,2)[->][very thick];
\draw (2.5,1) node {=};
\draw (3,0) -- (5,2)[->][very thick];
\draw (5,0) -- (3,2)[->][very thick];
\draw (4,0) .. controls (5,1) .. (4,2)[->][very thick];
\end{tikzpicture}
\end{equation}

\begin{equation}\label{eq:rel2}
\begin{tikzpicture}[>=stealth]
\draw (0,0) .. controls (1,1) .. (0,2)[<-][very thick];
\draw (1,0) .. controls (0,1) .. (1,2)[->] [very thick];
\draw (1.5,1) node {=};
\draw (2,0) --(2,2)[<-][very thick];
\draw (3,0) -- (3,2)[->][very thick];

\draw (3.8,1) node{$-\sum_{b \in \mathcal{B}}$};

\draw (4,1.75) arc (180:360:.5) [very thick];
\draw (4,2) -- (4,1.75) [very thick];
\draw (5,2) -- (5,1.75) [very thick][<-];
\draw (5,.25) arc (0:180:.5) [very thick];
\filldraw [blue] (4.5,1.25) circle (2pt);
\draw (4.5,1.25) node [anchor=south] {$b$};
\filldraw [blue] (4.5,0.75) circle (2pt);
\draw (4.5,.75) node [anchor=north] {$b^\vee$};
\draw (5,0) -- (5,.25) [very thick];
\draw (4,0) -- (4,.25) [very thick][<-];

\draw [shift={+(7,0)}](0,0) .. controls (1,1) .. (0,2)[->][very thick];
\draw [shift={+(7,0)}](1,0) .. controls (0,1) .. (1,2)[<-] [very thick];
\draw [shift={+(7,0)}](1.5,1) node {=};
\draw [shift={+(7,0)}](2,0) --(2,2)[->][very thick];
\draw [shift={+(7,0)}](3,0) -- (3,2)[<-][very thick];
\end{tikzpicture}
\end{equation}

\begin{equation}\label{eq:rel3}
\begin{tikzpicture}[>=stealth]
\draw [shift={+(0,0)}](0,0) arc (180:360:0.5cm) [very thick];
\draw [shift={+(0,0)}][->](1,0) arc (0:180:0.5cm) [very thick];
\filldraw [shift={+(1,0)}][blue](0,0) circle (2pt);
\draw [shift={+(0,0)}](1,0) node [anchor=east] {$b$};
\draw [shift={+(0,0)}](1.75,0) node{$= \tr(b).$};

\draw  [shift={+(5,0)}](0,0) .. controls (0,.5) and (.7,.5) .. (.9,0) [very thick];
\draw  [shift={+(5,0)}](0,0) .. controls (0,-.5) and (.7,-.5) .. (.9,0) [very thick];
\draw  [shift={+(5,0)}](1,-1) .. controls (1,-.5) .. (.9,0) [very thick];
\draw  [shift={+(5,0)}](.9,0) .. controls (1,.5) .. (1,1) [->] [very thick];
\draw  [shift={+(5,0)}](1.5,0) node {$=$};
\draw  [shift={+(5,0)}](2,0) node {$0.$};

\end{tikzpicture}
\end{equation}

In the first equation on line \ref{eq:rel2} above, the summation is taken over a basis $\mathcal{B}$ of $B^\G$-- this morphism is easily seen to be independent of the choice of basis.
We assign a $\Z$ grading to the space of planar diagrams by defining

$$
\begin{tikzpicture}[>=stealth]
\draw  (-.5,.5) node {$deg$};
\draw [->](0,0) -- (1,1) [very thick];
\draw [->](1,0) -- (0,1) [very thick];
\draw (1.5,.5) node{$ = 0$};
\end{tikzpicture}
$$
$$
\begin{tikzpicture}[>=stealth]
\draw  (-.5,-.25) node {$deg$};
\draw (0,0) arc (180:360:.5)[->] [very thick];
\draw (1.75,-.25) node{$ = deg$};
\draw (3.5,-.5) arc (0:180:.5) [->][very thick];
\draw (4.5,-.25) node{$ =-1$};
\end{tikzpicture}
$$
$$
\begin{tikzpicture}[>=stealth]
\draw  (-.5,-.25) node {$deg$};
\draw (0,0) arc (180:360:.5)[<-] [very thick];
\draw (1.75,-.25) node{$ = deg$};
\draw (3.5,-.5) arc (0:180:.5) [<-][very thick];
\draw (4.5,-.25) node{$ = 1$};
\end{tikzpicture}
$$
and by defining the degree of a dot labeled by $b$ to be the degree of $b$ in the graded algebra $B^\G$.  When equipped with these assignments all of the graphical relations are graded.  Thus $\Hom_{\H'_\G}(\cdot, \cdot)$ is a $\Z$ graded vector space and composition of morphisms is compatible with the grading.  

Because of the relation dictating how dots slide through crossings, there is a natural map from the semidirect product 
$$B_n^\G := (B^\G \otimes \dots \otimes B^\G) \rtimes S_n$$
to the endomorphism algebra $\Hom_{\H'_\G}(P^n,P^n)$ whose image is the subalgebra spanned by braid-like diagrams (i.e. diagrams with no local maxima or local minima).  Notice that, if we denote by $T_k$ upward crossings of the $k$ and $(k+1)$st adjacent strands and by $X_k(b)$ dots labeled $b \in B^\G$ on the $k$th strand then this subalgebra is generated by $T_1, \dots, T_{n-1}$ and $X_1(b), \dots, X_n(b)$ for $b \in B^\G$.  The $T$s and $X$s are subject to the following relations:
$$ T_k^2 = 1 \text{ and } T_k T_l = T_l T_k \text{ for } |k-l| > 1,$$
$$ T_kT_{k+1}T_k = T_{k+1}T_kT_{k+1} \text{ for all } k=1, \dots, n-2 $$
$$ T_k X_k(b) = X_{k+1}(b)T_k \text{ for  all } k=1, \dots, n-1. $$
$$ X_k(b)X_k(b') = X_k(bb') \text{ and } X_k(b)X_{l}(b') = (-1)^{|b||b'|}X_{l}(b')X_k(b) \text{ if }  k \ne l.$$
We will write $T$ and $X(b)$ instead of $T_k$ and $X_k(b)$ when the subscript is understood.
In fact, it will follow from the construction of Section \ref{sec:koszul} that the map $B_n^\G \rightarrow \End_{\H'_\G}(P^n)$ is injective, so that the subalgebra of endomorphisms spanned by braid-like diagrams is isomorphic $B_n^\G$ (see Remark \ref{rem:injective}).

\subsection{The category $\H_\G$}

We define $\H_\G$ to be the Karoubi envelope (also known as the idempotent completion) of $\H'_\G$.  By definition, the objects of $\H_\G$ are also indexed by $\Z$ while a 1-morphism consists of a pair $(R,e)$ where $R$ is a 1-morphism of $\H'_\G$ and $e: R \rightarrow R$ is an idempotent endomorphism $e^2 = e$. Morphisms from $(R,e)$ to $(R',e')$ are morphisms $g: R \rightarrow R'$ such that $ge = g$ and $e'g = g$.  The idempotent $e$ defines the identity morphism of $(R,e)$.

Since $\H_\G$ is a graded 2-category, the (split) Grothendieck group $K_0(\H_\G)$ of $\H_\G$ is a $\k[q,q^{-1}]$-linear category where multiplication by $q$ corresponds to the shift $\la 1 \ra$ (we will assume Grothendieck groups are tensored with the base field $\k$). The objects of $K_0(\H_\G)$ are the same as the objects of $\H_\G$, namely the integers. The space of morphisms $\Hom_{K_0(\H_\G)}(n,m)$ is the split Grothendieck group of the additive category $\Hom_{\H_\G}(n,m)$.  Composition of 1-morphisms in $\H_\G$ gives the Grothendieck group $K_0(\H_\G)$ the structure of a $\k$-algebra.

\subsection{Theorem \#1}\label{sec:mainthm1}

The first main theorem of this paper is that $\H_\G$ categorifies $\h_\G$:
\begin{theorem}\label{thm:main1}
There is a canonical isomorphism of algebras 
$$ \pi: \h_\G \xrightarrow{\sim} K_0(\H_\G). $$
\end{theorem}

This theorem will be proven in Section \ref{sec:K-theory}. In the rest of this section we want to sketch how the morphism $\pi$ comes about. 

To do this consider various idempotent 2-morphisms in the category $\H_\G'$. Since $M_{n_i}(\k) \subset \k[\G] \subset B^\G \subset \End_{\H'_\G}(P)$, any idempotent $e \in M_{n_i}(\k)$ gives rise (for any $n \in \Z$) to a 1-morphism $(P,e)$ in $\Hom_{\H_\G}(n,n+1)$. Let $\tilde{P}_i = (P,f_i) \in \H_\G$, where the $f_i$ are the pairwise orthogonal central idempotents of $\C[\G]$ described in Section \ref{sec:idemp}.

\begin{lemma}\label{lem:P}
There is an isomorphism of 1-morphisms in the category $\H_\G$ 
$$ P \cong \bigoplus_{i=1}^m \tilde{P}_i $$ 
\end{lemma}
\begin{proof}
The idempotents $f_i$ define a 2-morphism from
$\bigoplus_{i=1}^m \tilde{P}_i$ to $P$ in $\H_\G$.  Since $1 = \sum_{i=1}^m f_i$ and $f_if_j = \delta_{i,j} f_i$, this 2-morphism is an isomorphism.
\end{proof}

The 1-morphisms $\tilde{P}_i$ can be further decomposed in $\H_\G$ whenever the idempotents $f_i$ are not minimal. Recall the minimal idempotents $e_{i,1}, \dots e_{i,n_i}$ from Section \ref{sec:idemp}. Note that for fixed $i$ the $e_{i,j}$, which are distinct idempotents in $M_{n_i}(\k)$, define isomorphic representations $\C[\G] e_{i,1} \cong \C[\G] e_{i,l}$ of $\C[\G]$.  In particular, the objects $(P,e_{i,j})$ for distinct $j$ are all isomorphic in $\H_\G$.  Let $P_i$ denote the 1-morphism $(P,e_{i,1})\in \H_\G$. Then we have the following

\begin{lemma}
There is an ismorphism in $\H_\G$
$$ \tilde{P_i} \cong \bigoplus_{j=1}^{n_i} P_i^{\oplus n_i} $$
where $n_i$ is the dimension (as a $\k$ vector space) of the $\k[\G]$-module $V_i$. 
\end{lemma}
\begin{proof}
The proof is just as in Lemma \ref{lem:P} using the fact that $f_i = \sum_{j=1}^{n_i} e_{i,j}$. 
\end{proof}

We will prove later in Section \ref{sec:K-theory} that the objects $P_i \in \H_\G$ are indecomposable. 

The canonical adjunctions in $\H'_\G$ involving $P$ and $Q$ (the cups and caps) also allow us to decompose $Q$, since any idempotent $e \in \End_{\H'_\G}(P)$ gives rise to an idempotent $ e \in \End_{\H'_\G}(Q)$. We define 1-morphisms $Q_i \in \H_\G$ by $Q_i= (Q, e_{i,1}).$  It follows immediately by adjunction that $Q_i$ is indecomposable too.

A basic consequence of the relations between 2-morphisms in $\H_\G$ is the following result proved in Section \ref{sec:graphical2}.

\begin{prop}\label{prop:basicrel}
For $i,j\in I_\G$, we have
$$ Q_jP_i \cong 
\begin{cases} 
P_iQ_j \oplus \id \la -1 \ra \oplus \id \la 1 \ra & \text{ if } i=j, \\
P_iQ_j \oplus \id & \text{ if } \la i, j \ra = -1 
\end{cases} 
$$
where $\la \cdot \ra$ denotes the grading shift in $\H_\G$.  
\end{prop}

More generally, we define 1-morphisms $P_i^{(n)}$ and $Q_i^{(n)}$ by
$$P_i^{(n)} = (P^n,e_{triv,i,1}) \text{ and } Q_i^{(n)} = (Q^n, e_{triv,i,1}) $$
where the idempotent $e_{triv,i,1}$ is a minimal idempotent in $M_{n_i}(\k)^n \rtimes S_n \subset \k[\G^n \rtimes S_n]$ corresponding to the trivial representation of $S_n$. In Section \ref{sec:graphical2} we prove the following proposition.
\begin{prop}\label{prop:basicrel2}
For $i,j\in I_\G$ we have isomorphisms in $\H_\G$:
$$P_i^{(n)} P_j^{(m)} \cong P_j^{(m)} P_i^{(n)} \text{ and } Q_i^{(n)}Q_j^{(m)} \cong Q_j^{(m)} Q_i^{(n)} \text{ for all } i,j \in I_\G, $$
$$ Q_i^{(n)} P_j^{(m)} \cong
\begin{cases}
\bigoplus_{k \ge 0} P_i^{(m-k)}Q_i^{(n-k)} \otimes H^\star(\P^k) \text{ if } i=j \in I_\G, \\
P_j^{(m)} Q_i^{(n)} \oplus P_j^{(m-1)} Q_i^{(n-1)} \text{ if } \la i,j \ra = -1, \\
P_j^{(m)} Q_i^{(n)} \text{ if } \la i,j \ra = 0
\end{cases}$$
where $H^\star(\P^k)$ denotes the graded cohomology of projective spaces $\P^k$ shifted so that it lies in degrees $-k, -k+2 , \dots, k-2, k$. 
\end{prop}
Notice that in the third equation above the graded dimension of $H^\star(\P^k)$ is the quantum integer $[k+1]$. By convention $H^\star(\P^k) = 0$ if $k < 0$. Also, $P_i^{(l)} = Q_i^{(l)} = 0$ when $l < 0$ so the summand corresponding to a $k > \mbox{min}(n,m)$ is zero (so the direct sum on the right is finite).  

It follows immediately from Proposition \ref{prop:basicrel2} that there is a well-defined morphism
$$ \pi: \h_\G \longrightarrow K_0(\H_\G) $$
given by sending $p_i^{(n)}$ to the class $[P_i^{(n)}]$ and sending $q_j^{(m)}$ to the class
$[Q_j^{(m)}]$.  That the map $\pi$ is an isomorphism is a rigorous formulation of the informal
statement ``$\H_\G$ categorifies the Heisenberg algebra."

\subsection{Remarks on relations in $\H'_\G$}\label{sec:relremarks}

We end this section with some remarks on the relations from lines (\ref{eq:rel1}) - (\ref{eq:rel3}). 
The second relation in (\ref{eq:rel3}) is natural for degree reasons; in fact imposing this relation is equivalent to declaring that the object $P$ does not have negative degree endomorphisms.

Moreover, if we believe that the identity should not have negative degree endomorphisms then the left relation in (\ref{eq:rel3}) is immediate unless $a$ has degree 2, that is, unless $a$ is a linear combination of elements of the form $(v_1 \wedge v_2, \gamma)$. Since 
$$(v_1 \wedge v_2, \gamma) = (v_1, \gamma) \cdot (\gamma^{-1}v_2, 1),$$ 
from moving $\gamma^{-1}v_2$ around the circle we get 
\begin{equation*}
\begin{tikzpicture}[>=stealth]
\draw [->](0,0) arc (180:360:0.5cm) [very thick];
\draw (1,0) arc (0:180:0.5cm) [very thick];
\filldraw [blue](0,0) circle (2pt);
\draw (0,0) node [anchor=east] {$(v_1 \wedge v_2, \gamma)$};
\draw (1.75,0) node{$=$};
\draw [->](5,0) arc (180:360:0.5cm) [very thick];
\draw (6,0) arc (0:180:0.5cm) [very thick];
\filldraw [blue](5,0) circle (2pt);
\draw (5,0) node [anchor=east] {$(- \gamma^{-1} v_2 \wedge v_1, \gamma)$};
\end{tikzpicture}
\end{equation*}
where the minus sign is a consequence of passing the dot labeled $(v_1, \gamma)$ by the dot labeled $(\gamma^{-1} v_2, 1)$. Since $\gamma$ acts nontrivially on $V$, it follows that the couterclockwise circle with a solid dot labeled by $(v_1 \wedge v_2,\gamma)$ is zero unless $\gamma = 1$.

To understand the relation imposed when $\gamma=1$, we add a cup at the bottom of the first relation in (\ref{eq:rel2}). Simplifying, we obtain  
\begin{equation*}
\begin{tikzpicture}[>=stealth]
\draw (0,0) -- (0,2)[->][very thick];
\draw (1,1) node{$=$};
\draw (5,0) -- (5,2)[->][very thick];
\draw [->](3,1) arc (180:360:0.5cm) [very thick];
\draw (4,1) arc (0:180:0.5cm) [very thick];
\filldraw [blue](3,1) circle (2pt);
\draw (3,1) node [anchor=east] {$v_1 \wedge v_2$};
\end{tikzpicture}
\end{equation*}
It follows that if the closed diagram which appears above is a scalar then it must be equal to one. Alternatively, if we multiply the upward strand by $v_1 \wedge v_2$ and close off we get that the counterclockwise circle with a solid dot labeled by $v_1\wedge v_2$ is an idempotent. So, if this idempotent were not equal to one, then the identity map would break up into a direct sum of multiple 1-morphisms in the idempotent completion.  This would result in extra unwanted 1-morphisms in our Heisenberg category $\H_\G$. 

\subsubsection{Further relations}

Although left-twist curls on an upward pointing strand are zero right-twist curls (which have degree $2$) and are not necessarily zero. As shorthand, we will draw right twist curls as hollow dots:
$$
\begin{tikzpicture}
\draw  [shift={+(3,0)}](1,0) .. controls (1,.5) and (.3,.5) .. (.1,0) [very thick];
\draw  [shift={+(3,0)}](1,0) .. controls (1,-.5) and (.3,-.5) .. (.1,0) [very thick];
\draw  [shift={+(3,0)}](0,-1) .. controls (0,-.5) .. (.1,0) [very thick];
\draw  [shift={+(3,0)}](.1,0) .. controls (0,.5) .. (0,1) [->] [very thick];
\draw  (2.5,0) node {$:=$};
\draw  (2,-1)--(2,1)[very thick] [->];
\draw [red] (2,0) circle (4pt);
\end{tikzpicture}
.
$$
To pass a solid dot past a hollow dot involves sliding the solid dot along the right-twist curl. Thus solid dots commute with hollow dots 

$$\begin{tikzpicture}[>=stealth]\draw (0,0) -- (0,2)[->][very thick];
\draw [red] (0,1.33) circle (4pt);
\draw (0,.66) node [anchor=east] [black] {$b$};
\filldraw [blue] (0,.66) circle (2pt);
\draw (.5,1) node {=};
\draw (1,0) -- (1,2)[->][very thick];
\draw [red] (1,.66) circle (4pt);
\filldraw [blue] (1,1.33) circle (2pt);
\draw (1,1.33) node [anchor=west] [black] {$b$};
\end{tikzpicture}
.
$$
Hollow dots have an interesting ``affine Hecke" type relation with crossings which involves the creation of labeled solid dots,
$$
\begin{tikzpicture}[>=stealth]
\draw [->](0,0) -- (1,1) [very thick];
\draw [->](1,0) -- (0,1) [very thick];
\draw [red](.25,.25) circle (4pt);
\draw (1.25,.5) node{$=$};
\draw [->](1.5,0) -- (2.5,1) [very thick];
\draw [->](2.5,0) -- (1.5,1) [very thick];
\draw [red](2.25,.75) circle (4pt);

\draw (3.2,.5) node{$+\sum_{b \in \mathcal{B}}$};

\draw [shift={+(4.0,0)}][->](0,0) -- (0,1) [very thick];
\draw [shift={+(4.0,0)}][->](.75,0) -- (.75,1) [very thick];
\filldraw [shift={+(4.0,0)}][blue](0,.33) circle (2pt);
\draw [shift={+(4.0,0)}](0,.33) node [anchor=west]{$b$};
\filldraw [shift={+(4.0,0)}][blue](.75,.66) circle (2pt);
\draw [shift={+(4.0,0)}] (.75,.66) node [anchor=west]{$b^\vee$};
\end{tikzpicture}
$$
$$
\begin{tikzpicture}[>=stealth]
\draw [->](0,0) -- (1,1) [very thick];
\draw [->](1,0) -- (0,1) [very thick];
\draw [red](.25,.75) circle (4pt);
\draw (1.25,.5) node{$=$};
\draw [->](1.5,0) -- (2.5,1) [very thick];
\draw [->](2.5,0) -- (1.5,1) [very thick];
\draw [red](2.25,.25) circle (4pt);

\draw (3.2,.5) node{$+\sum_{b \in \mathcal{B}}$};

\draw [shift={+(4.0,0)}][->](0,0) -- (0,1) [very thick];
\draw [shift={+(4.0,0)}][->](.75,0) -- (.75,1) [very thick];
\filldraw [shift={+(4.0,0)}][blue](0,.33) circle (2pt);
\draw [shift={+(4.0,0)}](0,.33) node [anchor=west]{$b$};
\filldraw [shift={+(4.0,0)}][blue](.75,.66) circle (2pt);
\draw [shift={+(4.0,0)}] (.75,.66) node [anchor=west]{$b^\vee$};
\end{tikzpicture}
$$
where again the summations are over the basis $\mathcal{B}$.  These relations are reminiscent of relations in the degenerate affine Hecke algebra associated to the symmetric group and its wreath products, see for instance \cite{RS,W,WW}.  The proof of these relations is an easy  calculation using the defining relations for 2-morphisms in $\H'_\G$.

Finally, it also follows directly from the graphical relations that the triple point move,
$$
\begin{tikzpicture}[>=stealth]
\draw (0,0) -- (2,2)[very thick];
\draw (2,0) -- (0,2)[very thick];
\draw (1,0) .. controls (0,1) .. (1,2)[very thick];
\draw (2.5,1) node {=};
\draw (3,0) -- (5,2)[very thick];
\draw (5,0) -- (3,2)[very thick];
\draw (4,0) .. controls (5,1) .. (4,2)[very thick];
\end{tikzpicture}
$$
which was defined to hold when all three strands are oriented up, in fact holds in all possible orientations.  

\section{A 2-representation of $\H_\G$}\label{sec:action}

In this section we construct an action of $\H_\G$ on the derived categories of finite dimensional graded modules over certain algebras $A^\G_n$. This induces an action on derived categories of coherent sheaves on Hilbert schemes and allows us to show in the next section that our categories $\H_\G$ are non-degenerate (subsequently proving Theorem \ref{thm:main1}). 

\subsection{Categories}\label{sec:cats}

Since $\G \subset SL_2(\k)$ acts on $V \cong \k^2$ it acts naturally on $\Sym^*(V^\vee)$ and we denote by $\Sym^*(V^\vee) \rtimes \G$ their semi-direct product (also known as the smash product $\Sym^*(V^\vee)\#\k[\G]$). For instance, if $\G = \Z/n\Z$, identified as $n$th roots of unity, then we can identify $\Sym^*(V^\vee)$ with $\k[x,y]$ and the action is $\zeta \cdot (x,y) = (\zeta x, \zeta^{-1} y)$. $\Sym^*(V^\vee) \rtimes \G$ is a $\k$-algebra,  and an element of $\Sym^*(V^\vee) \rtimes \G$ is a linear combination of elements of the form $(f, \gamma)$ where $f \in \Sym^*(V^\vee)$ and $\gamma \in \G$ with multiplication given by 
$$(f, \gamma) \cdot (f', \gamma') = (f \gamma \cdot f', \gamma \gamma').$$

For each non-negative integer $n$ define 
$$A_n^\G := [(\Sym^*(V^\vee) \rtimes \G) \otimes (\Sym^*(V^\vee) \rtimes \G) \otimes \dots \otimes (\Sym^*(V^\vee) \rtimes \G)] \rtimes S_n$$
where $S_n$ is the symmetric group acting by permuting the $n$ terms in the product. Notice that here we use the ordinary tensor product and not the super tensor product used to define $B_n^\G$. These algebras inherit the natural grading from $\Sym^*(V^\vee)$. We denote by $D(A_n^\G \dgmod)$ the bounded derived category of finite dimensional, graded (left) $A_n^\G$-modules. 

We have maps 
$$\k[\G] \xrightarrow{i} A_1^\G \xrightarrow{p} \k[\G]$$ 
where the first denotes the natural inclusion of the group algebra $\k[\G]$ and the second is the projection which takes $\Sym^{>0}(V)$ to zero. Thus any $\k[\G]$-module is also an $A_1^\G$-module and vice versa. 

\subsection{Functors $\fP$ and $\fQ$}

Consider the natural inclusion $A_n^\G \otimes A_1^\G \rightarrow A_{n+1}^\G$ making use of the embedding $S_n = S_n \times S_1 \hookrightarrow S_{n+1}$. This gives $A_{n+1}^\G$ and $A_n^\G \otimes \k[\G]$ the structure of a module over $A_n^\G \otimes A_1^\G$. 

Define the $(A_{n+1}^\G,A_n^\G)$ bimodule
$$P^\G(n) := A_{n+1}^\G \otimes_{A_n^\G \otimes A_1^\G} (A_n^\G \otimes \k[\G])$$
and the $(A_n^\G,A_{n+1}^\G)$ bimodule
$$(n)Q^\G := (A_n^\G \otimes \k[\G]) \otimes_{A_n^\G \otimes A_1^\G} A_{n+1}^\G.$$

Define the functor $\fP(n): D(A_n^\G \dgmod) \rightarrow D(A_{n+1}^\G \dgmod)$ by
$$\fP(n)(\cdot) := P^\G(n) \otimes_{A_n^\G} (\cdot).$$
Similarly, we define $(n)\fQ: D(A_{n+1}^\G \dgmod) \rightarrow D(A_n^\G \dgmod)$ by
$$(n)\fQ(\cdot) := (n)Q^\G \otimes_{A_{n+1}^\G} (\cdot) [-1]\{1\}$$
where $[\cdot]$ denotes the cohomological shift while $\{\cdot\}$ the grading shift. The relationship between these shifts and the shift $\la \cdot \ra$ in $\H_\G$ is that $\la 1 \ra = [1]\{-1\}$, as we will see. We will also usually omit the $(n)$ from the notation for functors and just write $\fP$ or $\fQ$. 

\subsection{Natural transformations}

In order to define an $\H_\G$ action on $\oplus_n D(A_n^\G \dgmod)$ we also need to define the following natural transformations:
\begin{enumerate}
\item $X(v): \fP \rightarrow \fP [1]\{-1\}$ and $X(v): \fQ \rightarrow \fQ [1]\{-1\}$ for any $v \in V$,
\item $X(\gamma): \fP \rightarrow \fP$ and $X(\gamma): \fQ \rightarrow \fQ$ for any $\gamma \in \G$,
\item $T: \fP \fP \rightarrow \fP \fP$, $T: \fQ \fQ \rightarrow \fQ \fQ$, $T: \fQ \fP \rightarrow \fP \fQ$ and $T: \fP \fQ \rightarrow \fQ \fP$,  
\item $\adj: \fQ \fP \rightarrow \id [-1]\{1\}$ and $\adj: \fP \fQ \rightarrow \id [1]\{-1\}$, and
\item $\adj: \id \rightarrow \fQ \fP [-1]\{1\}$ and $\adj: \id \rightarrow \fP \fQ [1]\{-1\}$. 
\end{enumerate}

These natural transformations define the action of the following 2-morphisms in $\H'_\G$ (as usual we read the diagrams from the bottom up):
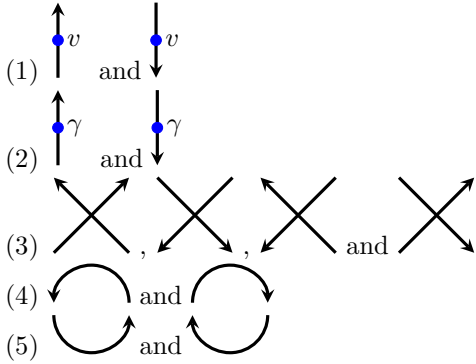
\begin{enumerate}
\item 
\begin{tikzpicture}[>=stealth]
\draw [->](0,0) -- (0,1) [very thick];
\filldraw [blue](0,.5) circle (2pt);
\draw (0,.5) node [anchor=west] [black] {$v$};
\end{tikzpicture}
and   
\begin{tikzpicture}[>=stealth]
\draw [<-](0,0) -- (0,1) [very thick];
\filldraw [blue](0,.5) circle (2pt);
\draw (0,.5) node [anchor=west] [black] {$v$};
\end{tikzpicture}

\item
\begin{tikzpicture}[>=stealth]
\draw [->](0,0) -- (0,1) [very thick];
\filldraw [blue](0,.5) circle (2pt);
\draw (0,.5) node [anchor=west] [black] {$\gamma$};
\end{tikzpicture}
and  
\begin{tikzpicture}[>=stealth]
\draw [<-](0,0) -- (0,1) [very thick];
\filldraw [blue](0,.5) circle (2pt);
\draw (0,.5) node [anchor=west] [black] {$\gamma$};
\end{tikzpicture}
\item
\begin{tikzpicture}[>=stealth]
\draw [->](0,0) -- (1,1) [very thick];
\draw [->](1,0) -- (0,1) [very thick];
\end{tikzpicture}
,
\begin{tikzpicture}[>=stealth]
\draw [<-](0,0) -- (1,1) [very thick];
\draw [<-](1,0) -- (0,1) [very thick];
\end{tikzpicture}
,
\begin{tikzpicture}[>=stealth]
\draw [<-](0,0) -- (1,1) [very thick];
\draw [->](1,0) -- (0,1) [very thick];
\end{tikzpicture}
and
\begin{tikzpicture}[>=stealth]
\draw [->](0,0) -- (1,1) [very thick];
\draw [<-](1,0) -- (0,1) [very thick];
\end{tikzpicture}

\item
\begin{tikzpicture}[>=stealth]
\draw (0,0) arc (180:0:.5) [<-,very thick];
\end{tikzpicture}
and
\begin{tikzpicture}[>=stealth]
\draw (2,0) arc (180:0:.5) [->,very thick];
\end{tikzpicture}

\item
\begin{tikzpicture}[>=stealth]
\draw (0,0) arc (180:360:.5) [->,very thick];
\end{tikzpicture}
and
\begin{tikzpicture}[>=stealth]
\draw (2,0) arc (180:360:.5) [<-,very thick];
\end{tikzpicture}
\end{enumerate}

\subsubsection{Preliminaries}

Since we will deal with tensor products of complexes we briefly review some standard conventions. Given two complexes 
$$A_\bullet = \dots \rightarrow A_i \xrightarrow{d} A_{i+1} \rightarrow \dots \text{ and }
B_\bullet = \dots \rightarrow B_i \xrightarrow{d} B_{i+1} \rightarrow \dots$$
the tensor product $A_\bullet \otimes B_\bullet$ is the complex with terms $\oplus_{i+j=k} A_i \otimes B_j$ and differential
$$A_i \otimes B_j \xrightarrow{d \otimes 1 + (-1)^i 1 \otimes d} A_{i+1} \otimes B_j \oplus A_i \otimes B_{j+1}.$$
Given a map $f: A_\bullet \rightarrow A_\bullet [k]$ of complexes induced by $f_i: A_i \rightarrow A_{i+k}$ we get a map $f \otimes 1: A_\bullet \otimes B_\bullet \rightarrow A_\bullet \otimes B_\bullet [k]$ induced by 
$$(f \otimes 1)_i = f_i \otimes 1: A_i \otimes B_j \rightarrow A_{i+k} \otimes B_j.$$ 
Similarly, given a map $g: B_\bullet \rightarrow B_\bullet [k]$ induced by $g_i: B_i \rightarrow B_{i+k}$ we get a map $1 \otimes g$ induced by 
$$(1 \otimes g)_i = (-1)^{ik} g_i \otimes 1: A_i \otimes B_j \rightarrow A_i \otimes B_{j+k}.$$
Notice that if $k=1$ then $(f \otimes 1)$ and $(1 \otimes g)$ anti-commute. 

\subsubsection{Definition of $X \in \End(P)$}

We define a map $X(\gamma): \k[\G] \rightarrow \k[\G]$ of left $A_1^\G$-modules via multiplication on the right $\gamma' \mapsto \gamma' \cdot \gamma$. This induces a map $X(\gamma): P^\G(n) \rightarrow P^\G(n)$ of $(A_{n+1}^\G, A_n^\G)$-bimodules.

To define $X(v)$ we have to work harder. Consider $A_1^\G \otimes V^\vee$ where the tensor product is over $\k$ (so $A_1^\G$ only acts on the left factor by multiplication on the left). Then the multiplication map $\Sym^*(V^\vee) \otimes V^\vee \rightarrow \Sym^*(V^\vee)$ induces a map $d: A_1^\G \otimes V^\vee \rightarrow A_1^\G$ of left $A_1^\G$-modules via 
$$(f, \gamma) \otimes w \mapsto (f \gamma \cdot w, \gamma)$$ 
where $f \in \Sym^*(V)$, $\gamma \in \G$ and $w \in V^\vee$. Similarly, we have a map $d: A_1^\G \otimes \wedge^2 V^\vee \rightarrow A_1^\G \otimes V^\vee$ given by 
$$(f, \gamma) \otimes (w_1 \wedge w_2) \mapsto (f (\gamma w_1), \gamma) \otimes w_2 - (f (\gamma w_2), \gamma) \otimes w_1.$$ 

Using these maps we can define the following free resolution of the left $A_1^\G$-module $\k[\G]$:
\begin{equation}\label{eq:freeres}
0 \rightarrow A_1^\G \otimes \wedge^2 V^\vee \xrightarrow{d} A_1^\G \otimes V^\vee \xrightarrow{d} A_1^\G \rightarrow \k[\G].
\end{equation}

Suppose $v \in V$. We need to define a map $X(v): \k[\G] \rightarrow \k[\G][1]\{-1\}$ inside $D(A_1^\G \dgmod)$. To do this define
$$\phi_v: A_1^\G \otimes V^\vee \rightarrow A_1^\G \{-1\}$$
to be the map induced by $(f, \gamma) \otimes w \mapsto \la w, v \ra (f, \gamma)$, where $\la \cdot, \cdot \ra : V^\vee \otimes V \rightarrow \k$ is the natural pairing. It is easy to check this map is a map of graded left $A_1^\G$-modules.

Similarly, define 
$$\phi'_v: A_1^\G \otimes \wedge^2 V^\vee \rightarrow A_1^\G \otimes V^\vee \{-1\}$$
to be the map induced by 
$$(f, \gamma) \otimes (w_1 \wedge w_2) \mapsto \la w_1, v \ra (f, \gamma) \otimes w_2 - \la  w_2, v \ra (f, \gamma) \otimes w_1.$$

Using these maps we can write down the following commutative diagram 
\begin{equation}\label{eq:X1def}
\xymatrix{
0 \ar[r] \ar[d] & A_1^\G \otimes \wedge^2 V^\vee \ar[r]^{d} \ar[d]^{\phi'_v} & A_1^\G \otimes V^\vee \ar[r]^{d} \ar[d]^{\phi_v} & A_1^\G \ar[d] \\
A_1^\G \otimes \wedge^2 V^\vee \{-1\} \ar[r]^{-d} & A_1^\G \otimes V^\vee \{-1\} \ar[r]^{-d} & A_1^\G \{-1\} \ar[r] & 0 
}
\end{equation}
The commutativity of the middle square is an easy exercise (note that the differential is $-d$ in the second row since shifting the complex by one negates the differential). 

The map from (\ref{eq:X1def}) induces a map $X(v): \k[\G] \rightarrow \k[\G][1]\{-1\}$ and subsequently a map $X(v): P^\G(n) \rightarrow P^\G(n)[1]\{-1\}$ (in the derived category) of graded $(A_{n+1}^\G, A_n^\G)$-bimodules.  

\subsubsection{Definition of $X \in \End(Q)$} 

If $\gamma \in \G$ then multiplication on the left induces a map $X(\gamma): \k[\G] \rightarrow \k[\G]$ of right $A_1^\G$-modules via $\gamma' \mapsto \gamma \cdot \gamma'$ and subsequently a map $X(\gamma): (n)Q^\G \rightarrow (n)Q^\G$ of $(A_n^\G, A_{n+1}^\G)$-bimodules. 

The map $X(v): (n)Q^\G \rightarrow (n)Q^\G [1]\{-1\}$ is defined by using a resolution of $Q^\G$ as above. We consider $V^\vee \otimes A_1^\G$ where $A_1^\G$ only acts on the second factor from the right. We have the maps 
$$d: V^\vee \otimes A_1^\G \rightarrow A_1^\G \text{ and } d: \wedge^2 V^\vee \otimes A_1^\G \rightarrow V^\vee \otimes A_1^\G$$ 
of right $A_1^\G$-modules via 
$$w \otimes (f, \gamma) \mapsto (wf, \gamma) \text{ and }(w_1 \wedge w_2) \otimes (f, \gamma) \mapsto w_1 \otimes (w_2f, \gamma) - w_2 \otimes (w_1f, \gamma).$$

Using these maps we have a free resolution of the right $A_1^\G$-module $\k[\G]$ as in (\ref{eq:freeres})
\begin{equation}\label{eq:freeres2}
0 \rightarrow \wedge^2 V^\vee \otimes A_1^\G \xrightarrow{d} V^\vee \otimes A_1^\G \xrightarrow{d} A_1^\G \rightarrow \k[\G].
\end{equation}

Then we can write a map $X(v): \k[\G] \rightarrow \k[\G][1]\{-1\}$ of graded right $A_1^\G$-modules as in Equation (\ref{eq:X1def}) except that we use the maps 
$$\psi_v: V^\vee \otimes A_1^\G \rightarrow A_1^\G \{-1\} \text{ and }
\psi'_v: \wedge^2 V^\vee \otimes A_1^\G \rightarrow V^\vee \otimes A_1^\G \{-1\}$$
given by 
$$w \otimes (f, \gamma) \mapsto - \la v, w \ra (f, \gamma) \text{ and } (w_1 \wedge w_2) \otimes (f, \gamma) \mapsto \la v, w_1 \ra w_2 \otimes (f, \gamma) - \la v, w_2 \ra w_1 \otimes (f, \gamma).$$

Using these maps we can again write down the following commutative diagram
\begin{equation}\label{eq:X2def}
\xymatrix{
0 \ar[r] \ar[d] & \wedge^2 V^\vee \otimes A_1^\G \ar[r]^d \ar[d]^{\psi'_v} & V^\vee \otimes A_1^\G \ar[r]^d \ar[d]^{\psi_v} & A_1^\G \ar[d] \\
\wedge^2 V^\vee \otimes A_1^\G \{-1\} \ar[r]^{-d} & V^\vee \otimes A_1^\G \{-1\} \ar[r]^{-d} & A_1^\G \{-1\} \ar[r] & 0.
}
\end{equation}
This induces a map $X(v): \k[\G] \rightarrow \k[\G][1]\{-1\}$ and subsequently a map $X(v): (n)Q^\G \rightarrow (n)Q^\G[1]\{-1\}$ (in the derived category) of graded $(A_n^\G, A_{n+1}^\G)$-bimodules. 

\subsubsection{Definition of $T \in \End(P^2)$ and $T \in \End(Q^2)$}

To define $T: \fP \fP \rightarrow \fP \fP$ we need a map 
$$P^\G(n+1) \otimes_{A_{n+1}^\G} P^\G(n) \rightarrow P^\G(n+1) \otimes_{A_{n+1}^\G} P^\G(n).$$
Now 
\begin{eqnarray*}
P^\G(n+1) \otimes_{A_{n+1}^\G} P^\G(n) 
&\cong& A_{n+2}^\G \otimes_{A_{n+1}^\G \otimes A_1^\G} (A_{n+1}^\G \otimes \k[\G]) \otimes_{A_n^\G \otimes A_1^\G} (A_n^\G \otimes \k[\G]) \\
&\cong& A_{n+2}^\G \otimes_{A_n^\G \otimes A_1^\G \otimes A_1^\G} (A_n^\G \otimes \k[\G] \otimes \k[\G]) \\
&\cong& A_{n+2}^\G \otimes_{A_n^\G \otimes A_2^\G} A_n^\G \otimes \left(A_2^\G \otimes_{A_1^\G \otimes A_1^\G} \k[\G] \otimes \k[\G] \right). 
\end{eqnarray*}
Now define a map 
$T: A_2^\G \otimes_{A_1^\G \otimes A_1^\G} (\k[\G] \otimes \k[\G]) \rightarrow A_2^\G \otimes_{A_1^\G \otimes A_1^\G} (\k[\G] \otimes \k[\G])$
by
\begin{equation}\label{eq:Tdef}
a \otimes (\gamma \otimes \gamma') \mapsto a s_1 \otimes (\gamma' \otimes \gamma)
\end{equation}
where $a \in A_2^\G$ and $s_1 \in S_2$ is the transposition $(1,2)$. 

Let us check that this induces a well defined map. Suppose $a = (a',a'')$ where $a', a'' \in A_1^\G$. Then 
\begin{equation}\label{eq:1}
T(a \otimes (\gamma \otimes \gamma')) = T(1 \otimes (a' \cdot \gamma \otimes a'' \cdot \gamma')) = s_1 \otimes (a'' \gamma' \otimes a' \gamma).
\end{equation}
On the other hand
$$(a', a'') s_1 \otimes (\gamma' \otimes \gamma) = s_1 (a'',a') \otimes (\gamma' \otimes \gamma) = s_1 \otimes (a'' \gamma', a' \gamma)$$
which agrees with (\ref{eq:1}). This shows that $T$ is well defined. 

Notice that the induced map 
$$T: A_{n+2}^\G \otimes_{A_n^\G \otimes A_1^\G \otimes A_1^\G} (A_n^\G \otimes \k[\G] \otimes \k[\G]) \rightarrow A_{n+2}^\G \otimes_{A_n^\G \otimes A_1^\G \otimes A_1^\G} (A_n^\G \otimes \k[\G] \otimes \k[\G])$$
is given by 
$$1 \otimes (1 \otimes \gamma \otimes \gamma') \mapsto s_{n+1} \otimes (1 \otimes \gamma' \otimes \gamma)$$
where $s_{n+1} = (n+1, n+2)$. 

The map $T: \fQ \fQ \rightarrow \fQ \fQ$ is defined in exactly the same way. 

\subsubsection{Definition of $T: QP \rightarrow PQ$}
To define $T: \fQ \fP \rightarrow \fP \fQ$ we need a map 
$$(A_n^\G \otimes \k[\G]) \otimes_{A_n^\G \otimes A_1^\G} A_{n+1}^\G \otimes_{A_n^\G \otimes A_1^\G} (A_n^\G \otimes \k[\G]) \rightarrow A_n^\G \otimes_{A_{n-1}^\G \otimes A_1^\G} (A_{n-1}^\G \otimes \k[\G] \otimes \k[\G]) \otimes_{A_{n-1}^\G \otimes A_1^\G} A_n^\G$$
of $(A_n^\G, A_n^\G)$-bimodules. To define this map it suffices to say where to take $(1 \otimes \gamma) \otimes a \otimes  (1 \otimes \gamma')$ where $a \in A_{n+1}^\G$ equals $1$ or $s_n = (n, n+1)$. If $a=1$ then we map it to zero while
$$(1 \otimes \gamma) \otimes s_n \otimes  (1 \otimes \gamma') \mapsto 1 \otimes (1 \otimes \gamma' \otimes \gamma) \otimes 1.$$

\subsubsection{Definition of $T: PQ \rightarrow QP$}
To define $T: \fP \fQ \rightarrow \fQ \fP$ we need a map 
$$A_n^\G \otimes_{A_{n-1}^\G \otimes A_1^\G} (A_{n-1}^\G \otimes \k[\G] \otimes \k[\G]) \otimes_{A_{n-1}^\G \otimes A_1^\G} A_n^\G \rightarrow (A_n^\G \otimes \k[\G]) \otimes_{A_n^\G \otimes A_1^\G} A_{n+1}^\G \otimes_{A_n^\G \otimes A_1^\G} (A_n^\G \otimes \k[\G])$$
of $(A_n^\G, A_n^\G)$-bimodules. This map is uniquely defined by 
$$1 \otimes (1 \otimes 1 \otimes 1) \otimes 1 \mapsto (1 \otimes 1) \otimes s_n \otimes (1 \otimes 1).$$

\subsubsection{Definition of $adj: QP \rightarrow id [-1]\{1\}$}

We need a map of $(A_n^\G, A_n^\G)$-bimodules
$$adj: (n)Q^\G \otimes_{A_{n+1}^\G} P^\G(n) \rightarrow A_n^\G.$$
Now 
$$(n)Q^\G \otimes_{A_{n+1}^\G} P^\G(n) = (A_n^\G \otimes \k[\G]) \otimes_{A_n^\G \otimes A_1^\G} A_{n+1}^\G \otimes_{A_n^\G \otimes A_1^\G} (A_n^\G \otimes \k[\G]).$$
Notice that $\k[S_{n+1}]$ as a $\k[S_n]$-bimodule is free and generated by $1$ and $s_n = (n, n+1)$. So the map $\k[S_{n+1}] \rightarrow \k[S_n]$ of $\k[S_n]$-bimodules given by $1 \mapsto 1$ and $s_n \mapsto 0$ induces a map $A_{n+1}^\G \rightarrow A_n^\G \otimes A_1^\G$ of $A_n^\G \otimes A_1^\G$-bimodules. Subsequently we obtain a map 
\begin{eqnarray}\label{eq:2}
(n)Q^\G \otimes_{A_{n+1}^\G} P^\G(n) 
&\rightarrow& (A_n^\G \otimes \k[\G]) \otimes_{A_n^\G \otimes A_1^\G} (A_n^\G \otimes A_1^\G) \otimes_{A_n^\G \otimes A_1^\G} (A_n^\G \otimes \k[\G]) \\
&\cong& (A_n^\G \otimes \k[\G]) \otimes_{A_n^\G \otimes A_1^\G} (A_n^\G \otimes \k[\G]).
\end{eqnarray}
Using the composition $\k[\G] \otimes_{A_1^\G} \k[\G] \rightarrow \k[\G] \xrightarrow{\tr} \k$ where the first map is multiplication and the second is the trace map (normalized so that $\tr(1) = 1$) we get a morphism 
$$(A_n^\G \otimes \k[\G]) \otimes_{A_n^\G \otimes A_1^\G} (A_n^\G \otimes \k[\G]) \rightarrow A_n^\G \otimes_{A_n^\G} A_n^\G = A_n^\G.$$
Composing with (\ref{eq:2}) defines $adj: (n)Q^\G \otimes_{A_{n+1}^\G} P^\G(n) \rightarrow A_n^\G$. 

\subsubsection{Definition of $adj: PQ \rightarrow id[1]\{-1\}$} 

We need a map of $(A_{n+1}^\G, A_{n+1}^\G)$-bimodules 
$$P^\G(n) \otimes_{A_n^\G} (n)Q^\G \rightarrow A_{n+1}^\G[2]\{-2\}.$$
Now
\begin{eqnarray*}
P^\G(n) \otimes_{A_n^\G} (n)Q^\G = A_{n+1}^\G \otimes_{A_n^\G \otimes A_1^\G} (A_n^\G \otimes \k[\G] \otimes \k[\G]) \otimes_{A_n^\G \otimes A_1^\G} A_{n+1}^\G
\end{eqnarray*}
so we basically need a map $h: \k[\G] \otimes \k[\G] \rightarrow A_1^\G [2]\{-2\}$ of graded $(A_1^\G, A_1^\G)$-bimodules. Then we define $adj$ as the composition 
$$adj: A_{n+1}^\G \otimes_{A_n^\G \otimes A_1^\G} (A_n^\G \otimes \k[\G] \otimes \k[\G]) \otimes_{A_n^\G \otimes A_1^\G} A_{n+1}^\G \xrightarrow{h} A_{n+1}^\G \otimes_{A_n^\G \otimes A_1^\G} A_{n+1}^\G [2]\{-2\} \rightarrow A_{n+1}^\G [2]\{-2\}$$
where the second map is multiplication. 

To define $h$ we use the resolutions (\ref{eq:freeres}) and (\ref{eq:freeres2}) of $\k[\G]$. Tensoring the two resolutions we see that $h$ is defined by a map
$$(A_1^\G \otimes \wedge^2 V^\vee) \otimes A_1^\G \oplus (A_1^\G \otimes V^\vee) \otimes (V^\vee \otimes A_1^\G) \oplus A_1^\G \otimes (\wedge^2 V^\vee \otimes A_1^\G) \rightarrow A_1^\G \{-2\}.$$ 
We define the map from each summand as
\begin{enumerate}
\item $((f, \gamma) \otimes (w \wedge w')) \otimes (f', \gamma') \mapsto (f, \gamma) (f', \gamma') \omega(w \wedge w')$
\item $((f, \gamma) \otimes w) \otimes (w' \otimes (f', \gamma')) \mapsto (f, \gamma) (f', \gamma') \omega(w \wedge w')$
\item $(f, \gamma) \otimes ((w \wedge w') \otimes (f', \gamma')) \mapsto - (f, \gamma) (f', \gamma') \omega(w 
\wedge w')$
\end{enumerate}
where $\omega: \wedge^2 V^\vee \rightarrow \k$ is our fixed isomorphism. 

In order for this map to be well defined ({\it{i.e.}} a map of complexes) we need the composition 
\begin{equation*}
\xymatrix{
(A_1^\G \otimes \wedge^2 V^\vee) \otimes A_1^\G \oplus (A_1^\G \otimes V^\vee) \otimes (V^\vee \otimes A_1^\G) \oplus A_1^\G \otimes (\wedge^2 V^\vee \otimes A_1^\G) \ar[r] & A_1^\G \{-2\} \\
(A_1^\G \otimes \wedge^2 V^\vee) \otimes (V^\vee \otimes A_1^\G) \oplus (A_1^\G \otimes V^\vee) \otimes (\wedge^2 V^\vee \otimes A_1^\G) \ar[u] &  
}
\end{equation*}
to be zero. 

Let us check that the composition $(A_1^\G \otimes \wedge^2 V^\vee) \otimes (V^\vee \otimes A_1^\G) \rightarrow A_1^\G \{-2\}$ is zero. The first map is 
\begin{eqnarray*}
(f, \gamma) \otimes (w \wedge w') \otimes (w'' \otimes (f', \gamma')) 
&\mapsto& ((f, \gamma) \otimes (w \wedge w')) \otimes (f' w'', \gamma') + \\
& & ((f \gamma \cdot w, \gamma) \otimes w') \otimes (w'' \otimes (f', \gamma')) - \\
& & ((f \gamma \cdot w', \gamma) \otimes w) \otimes (w'' \otimes (f', \gamma')).
\end{eqnarray*}
Composing with the second map we get
$$(f (\gamma f') (\gamma w''), \gamma \gamma') \omega(w \wedge w') + (f (\gamma f') (\gamma w), \gamma \gamma') \omega(w' \wedge w'') - (f (\gamma f') (\gamma w'), \gamma \gamma') \omega(w \wedge w'').$$
Since the maps are linear in the $w$'s it suffices to check this is zero when $w = w'$, when $w = w''$ and when $w' = w''$. In all cases one of the three terms is immediately zero and the other two cancel out. 

Similarly, the other composition maps $((f,\gamma) \otimes w) \otimes ((w' \wedge w'') \otimes (f', \gamma'))$ to 
$$-(f(\gamma w)(\gamma f'), \gamma \gamma') \omega(w' \wedge w'') - (f (\gamma w'') (\gamma f'), \gamma \gamma') \omega(w \wedge w') + (f (\gamma w') (\gamma f'), \gamma \gamma') \omega(w \wedge w'')$$
which also equals zero. 

\subsubsection{Definition of $adj: \id \rightarrow QP [-1]\{1\}$}

We need a map of $(A_n^\G,A_n^\G)$-bimodules
$$A_n^\G \rightarrow (n)Q^\G \otimes_{A_{n+1}^\G} P^\G(n) [-2]\{2\}$$
which translates to a map 
\begin{equation}\label{eq:3}
A_n^\G [2]\{-2\} \rightarrow (A_n^\G \otimes \k[\G]) \otimes_{A_n^\G \otimes A_1^\G} A_{n+1}^\G \otimes_{A_1^\G \otimes A_n^\G} (\k[\G] \otimes A_n^\G) .
\end{equation}

We first replace each $\k[\G]$ with its resolution using (\ref{eq:freeres}) and (\ref{eq:freeres2}). Then to define the map in (\ref{eq:3}) we just need to define a map 
$$A_n^\G \{-2\} \rightarrow \bigoplus_{k+l=2} (A_n^\G \otimes (\wedge^k V^\vee \otimes A_1^\G)) \otimes_{A_n^\G \otimes A_1^\G} A_{n+1}^\G \otimes_{A_1^\G \otimes A_n^\G} ((A_1^\G \otimes \wedge^l V^\vee) \otimes A_n^\G)$$ 
of graded $(A_n^\G, A_n^\G)$-bimodules. We send 
\begin{eqnarray*}
1 &\mapsto& (1 \otimes 1) \otimes 1 \otimes ((1 \otimes w_1 \wedge w_2) \otimes 1) \\
& & - (1 \otimes (w_1 \otimes 1)) \otimes 1 \otimes ((1 \otimes w_2) \otimes 1) + (1 \otimes (w_2 \otimes 1)) \otimes 1 \otimes ((1 \otimes w_1) \otimes 1) \\
& & - (1 \otimes (w_1 \wedge w_2 \otimes 1) \otimes 1 \otimes (1 \otimes 1)
\end{eqnarray*}
where $w_1, w_2$ is our chosen basis of $V^\vee$. This uniquely determines the map. 

To see this is well defined ({\it{i.e.}} a map of complexes) we need to check that the composition 
\begin{equation*}
\xymatrix{
 & \bigoplus_{k+l=1} (A_n^\G \otimes (\wedge^k V^\vee \otimes A_1^\G)) \otimes_{A_n^\G \otimes A_1^\G} A_{n+1}^\G \otimes_{A_1^\G \otimes A_n^\G} ((A_1^\G \otimes \wedge^l V^\vee) \otimes A_n^\G) \\
A_n^\G \{-2\} \ar[r] & \bigoplus_{k+l=2} (A_n^\G \otimes (\wedge^k V^\vee \otimes A_1^\G)) \otimes_{A_n^\G \otimes A_1^\G} A_{n+1}^\G \otimes_{A_1^\G \otimes A_n^\G} ((A_1^\G \otimes \wedge^l V^\vee) \otimes A_n^\G) \ar[u] 
}
\end{equation*}
is zero. The check of this fact is a straight-forward exercise.

\subsubsection{Definition of $adj: id \rightarrow PQ [1]\{-1\}$}

We need a map of graded $(A_{n+1}^\G, A_{n+1}^\G)$-bimodules 
$A_{n+1}^\G \rightarrow P^\G(n) \otimes_{A_n^\G} (n)Q^\G$
or equivalently a map 
$$A_{n+1}^\G \rightarrow A_{n+1}^\G \otimes_{A_n^\G \otimes A_1^\G} (A_n^\G \otimes \k[\G] \otimes \k[\G]) \otimes_{A_n^\G \otimes A_1^\G} A_{n+1}^\G.$$
We send
\begin{equation}\label{eq:adj4}
1 \mapsto \sum_{i=0}^n \sum_{\gamma \in \G} s_i \dots s_n \otimes (1 \otimes \gamma \otimes \gamma^{-1}) \otimes s_n \dots s_i
\end{equation}
where $s_i = (i,i+1) \in S_{n+1}$. Here $s_i \dots s_n = 1$ if $i=0$ by convention.

To check this bimodule map is well defined we need to show that 
$$\sum_{i=0}^n \sum_{\gamma \in \G} a s_i \dots s_n \otimes (1 \otimes \gamma \otimes \gamma^{-1}) \otimes s_n \dots s_i = \sum_{i=0}^n \sum_{\gamma \in \G} s_i \dots s_n \otimes (1 \otimes \gamma \otimes \gamma^{-1}) \otimes s_n \dots s_i a$$
for any $a \in A_{n+1}^\G$. If $a \in (A_1^\G)^{\otimes n+1} \subset A_{n+1}^\G$ then 
\begin{eqnarray*}
\sum_{i=0}^n \sum_{\gamma \in \G} a s_i \dots s_n \otimes (1 \otimes \gamma \otimes \gamma^{-1}) \otimes s_n \dots s_i  
&=& \sum_{i=0}^n \sum_{\gamma \in \G} s_i \dots s_n b_i \otimes (1 \otimes \gamma \otimes \gamma^{-1}) \otimes s_n \dots s_i \\
&=& \sum_{i=0}^n \sum_{\gamma \in \G} s_i \dots s_n \otimes (1 \otimes \gamma \otimes \gamma^{-1}) \otimes b_i s_n \dots s_i \\
&=& \sum_{i=0}^n \sum_{\gamma \in \G} s_i \dots s_n \otimes (1 \otimes \gamma \otimes \gamma^{-1}) \otimes s_n \dots s_i a
\end{eqnarray*}
where $b_i = (s_n \dots s_i) \cdot a$. The second equality follows since $\sum_{\gamma \in \G} \gamma \otimes \gamma^{-1}$ lies in the centre of $\k[\G] \otimes_\k \k[\G]$. 

Since $A_{n+1}^\G$ is generated by $\k[S_{n+1}]$ as an $((A_1^\G)^{\otimes n+1}, (A_1^\G)^{\otimes n+1})$-bimodule it remains to show that 
$$\sum_{i=0}^n \sum_{\gamma \in \G} s_k s_i \dots s_n \otimes (1 \otimes \gamma \otimes \gamma^{-1}) \otimes s_n \dots s_i  = \sum_{i=0}^n \sum_{\gamma \in \G} s_i \dots s_n \otimes (1 \otimes \gamma \otimes \gamma^{-1}) \otimes s_n \dots s_i s_k$$
for $k=1, \dots, n$. The left hand side is the sum over $\gamma \in \G$ of 
\begin{eqnarray*}
& & \sum_{i=0}^{k-1} s_i \dots s_n s_{k-1} \otimes (1 \otimes \gamma \otimes \gamma^{-1}) \otimes s_n \dots s_i + s_ks_{k+1} \dots s_n \otimes (1 \otimes \gamma \otimes \gamma^{-1}) \otimes s_n \dots s_{k+1} \\
 &+& s_{k+1} \dots s_n \otimes (1 \otimes \gamma \otimes \gamma^{-1}) \otimes s_n \dots s_k + \sum_{i=k+2}^n s_i \dots s_n s_k \otimes (1 \otimes \gamma \otimes \gamma^{-1}) \otimes s_n \dots s_i
\end{eqnarray*}
because $s_k s_i \dots s_n = s_i \dots s_n s_{k-1}$ if $i \le k-1$ and $s_k s_i \dots s_n = s_i \dots s_n s_k$ if $i \ge k+2$. A similar calculation of the right side yields the same expression so $adj$ is well defined. 

\subsection{Theorem \#2} 

The second main theorem of this paper is the following: 

\begin{theorem}\label{thm:main2}
The natural transformations $X$, $T$ and $\adj$ satisfy the Heisenberg 2-relations and give a categorical Heisenberg action of $\H_\G$ on $\oplus_{n \ge 0} D(A_n^\G \dgmod)$. 
\end{theorem}

We shall check all the Heisenberg relations required for the proof of Theorem \ref{thm:main2} in the next section. 

\begin{Remark}
In type $A$ (when $\G$ is a cyclic group) one can define everything $\k^\times$-equivariantly where the $\k^\times \subset SL(V)$ is the torus which commutes with $\G$. More precisely, the action of $\k^\times$ on $V$ induces an action on $\Sym^*(V^\vee)$ and instead of $A_n^\G$ one uses
$$\hat{A}_n^\G := [(\Sym^*(V^\vee) \rtimes (\G \times \k^\times)) \otimes (\Sym^*(V^\vee) \rtimes (\G \times \k^\times)) \otimes \dots \otimes (\Sym^*(V^\vee) \rtimes (\G \times \k^\times))] \rtimes S_n.$$
It is easy to check that all the $1$-morphisms and $2$-morphisms defined above are compatible with this extra $\k^\times$-structure. This means Theorem \ref{thm:main2} still holds so we get an action of $\H_\G$ on $\oplus_{n \ge 0} D(\hat{A}_n^\G \dgmod)$ (see also \cite{FJW2}).
\end{Remark}

\section{Proof of Theorem \ref{thm:main2}}\label{sec:action2}

\subsection{Composition of $X$'s}

We can define $X(b): \fP \rightarrow \fP [|b|]\{-|b|\}$ for any homogeneous $b \in B^\G$ by composing a sequence of $X(\gamma)$ and $X(v)$ where $\gamma \in \G$ and $v \in V$. In order for this to be well defined we need to check that 
\begin{enumerate}
\item $X(\gamma) X(\gamma') = X(\gamma \gamma')$ for any $\gamma, \gamma' \in \G$,
\item $X(\gamma) X(v) = X(\gamma \cdot v) X(\gamma)$ for any $\gamma \in \G$ and $v \in V$, and
\item $X(v') X(v) = - X(v) X(v')$.
\end{enumerate}

(1) The first assertion is clear since $\Gamma$ acts on $P^\G(n) = A_{n+1} \otimes_{A_n \otimes A_1} (A_n \otimes \k[\G])$ by right multiplication on $\k[\G]$. 

(2) To see the second assertion we replace $\k[\G]$ by its free resolution (\ref{eq:freeres}). In the free resolution $\G$ acts on the right only on the first factor of $A_1^\G \otimes \wedge^i V^\vee$. Writing out the composition we see that two pairs of compositions should agree:
\begin{itemize}
\item $A_1^\G \otimes V^\vee \xrightarrow{\gamma} A_1^\G \otimes V^\vee \xrightarrow{\phi_{v}} A_1^\G \{-1\}$ and \\
$A_1^\G \otimes V^\vee \xrightarrow{\phi_{\gamma v}} A_1^\G \{-1\} \xrightarrow{\gamma} A_1^\G \{-1\}$; 
\item $A_1^\G \otimes \wedge^2 V^\vee \xrightarrow{\gamma} A_1^\G \otimes \wedge^2 V^\vee \xrightarrow{\phi'_{v}} A_1^\G \otimes V^\vee \{-1\}$ and \\
$A_1^\G \otimes \wedge^2 V^\vee \xrightarrow{\phi'_{\gamma v}} A_1^\G \otimes V^\vee \{-1\} \xrightarrow{\gamma} A_1^\G \otimes V^\vee \{-1\}$. 
\end{itemize}
In the first pair the first composition is equal to 
$$(f, \gamma') \otimes w \mapsto (f, \gamma' \gamma) \otimes (\gamma^{-1} w) \mapsto \la \gamma^{-1} w, v \ra (f, \gamma' \gamma)$$
while the second composition is equal to 
$$(f, \gamma') \otimes w \mapsto \la w, \gamma v \ra (f, \gamma') \mapsto \la w, \gamma v \ra (f, \gamma' \gamma)$$
and these are the same since  $\la \cdot, \cdot \ra$ is invariant under the action of $\G$. 

Similarly, in the second pair the first composition is equal to
\begin{eqnarray*}
(f, \gamma') \otimes (w_1 \wedge w_2) 
&\mapsto& (f, \gamma' \gamma) \otimes (\gamma^{-1} w_1 \wedge \gamma^{-1} w_2) \\
&\mapsto& - \la \gamma^{-1} w_2, v \ra (f, \gamma' \gamma) \otimes (\gamma^{-1} w_1) + \la \gamma^{-1} w_1, v \ra (f, \gamma' \gamma) \otimes (\gamma^{-1} w_2)
\end{eqnarray*}
while the second composition is 
\begin{eqnarray*}
(f, \gamma') \otimes (w_1 \wedge w_2) 
&\mapsto& - \la w_2, \gamma v \ra (f, \gamma') \otimes w_1 + \la w_1, \gamma v \ra (f, \gamma') \otimes w_2 \\
&\mapsto& - \la w_2, \gamma v \ra (f, \gamma' \gamma) \otimes (\gamma^{-1} w_1) + \la w_1, \gamma v \ra (f, \gamma' \gamma) \otimes (\gamma^{-1} w_2) 
\end{eqnarray*}
which is clearly the same as the first composition. 

(3) To see the third assertion we again replace $\k[\G]$ by its free resolution (\ref{eq:freeres}) and then we need to show that the compositions
$$A_1^\G \otimes \wedge^2 V^\vee \xrightarrow{\phi'_{v'}} A_1^\G \otimes V^\vee \{-1\} \xrightarrow{\phi_v} A_1^\G \{-2\} \text{ and } A_1^\G \otimes \wedge^2 V^\vee \xrightarrow{\phi'_{v}} A_1^\G \otimes V^\vee \{-1\} \xrightarrow{\phi_{v'}} A_1^\G \{-2\}$$
differ precisely by a minus sign. The first composition is
\begin{eqnarray*}
(f, \gamma) \otimes (w_1 \wedge w_2) 
&\mapsto& - \la w_2, v' \ra (f, \gamma) \otimes w_1 + \la w_1, v' \ra (f, \gamma) \otimes w_2 \\
&\mapsto& - \la w_1, v \ra \la w_2, v' \ra (f, \gamma) + \la w_2, v \ra \la w_1, v' \ra (f, \gamma)
\end{eqnarray*}
and similarly the second composition is
$$(f, \gamma) \otimes (w_1 \wedge w_2) \mapsto - \la w_1, v' \ra \la w_2, v \ra (f, \gamma) + \la w_2, v' \ra \la w_1, v \ra (f, \gamma).$$
These two compositions clearly differ by multiplication by $-1$. 

This shows that $B^\G = \Lambda^*(V) \rtimes \G$ acts on $\fP$ and a similar proof shows that it also acts on $\fQ$. 

\subsubsection{$X's$ on different strands.}\label{subsec:diffstrand} 
We also need to check that elements of $B^\G$ on different strands supercommute, namely
$$(IX(b)) (X(b')I) = (-1)^{|b||b'|} (X(b')I) (IX(b)) \text{ for any homogeneous } b,b' \in B^{\G}.$$
This follows from the discussion at the beginning of this section, since $X(b)$ is a map of degree $|b|$ of complexes; whenever $f: A_{\bullet} \rightarrow A_{\bullet + |f|}$ and $g: A'_{\bullet} \rightarrow A'_{\bullet+|g|}$ are maps of complexes, it follows that $(f \otimes 1) (1 \otimes g) = (-1)^{|f||g|} (1 \otimes g) (f \otimes 1)$ as maps of complexes $A_{\bullet} \otimes B_{\bullet} \rightarrow A_{\bullet} \otimes B_{\bullet}[|f|+|g|]$. 

\subsection{Adjoint relations}\label{sec:adjointrels}

We now check that the following compositions involving adjunctions are all equal to the identity:
\begin{enumerate}
\item $\fP \xrightarrow{I \adj} \fP \fQ \fP [-1]\{1\} \xrightarrow{\adj I} \fP$ and $\fQ \xrightarrow{\adj I} \fQ \fP \fQ [-1]\{1\} \xrightarrow{I \adj} \fQ$
\item $\fP \xrightarrow{\adj I} \fP \fQ \fP [1]\{-1\} \xrightarrow{I \adj} \fP$ and $\fQ \xrightarrow{I \adj} \fQ \fP \fQ [1]\{-1\} \xrightarrow{\adj I} \fQ$.
\end{enumerate}
This will prove the graphical relations
$$
\begin{tikzpicture}
  \draw[very thick] (0,0) to (0,1) arc(180:0:.5) arc(180:360:.5)
  to (2,2);
  \draw (2.5,1) node {$=$};
  \draw [very thick] (3,0) -- (3,2);
  \draw (3.5,1) node {$=$};
   \draw[very thick] (4,2) to (4,1) arc(180:360:.5) arc(180:0:.5)
  to (6,0);
\end{tikzpicture}
$$
in both possible orientations.

\subsubsection{The composition $\fP \xrightarrow{I \adj} \fP \fQ \fP [-1]\{1\} \xrightarrow{\adj I} \fP$.} To show this is the identity it suffices to check that the composition $C_i \rightarrow D_i \{2\} \rightarrow C_i$ is the identity for $i=0,1,2$ where 
$$C_i := A_{n+1}^\G \otimes_{A_n^\G \otimes A_1^\G} (A_n^\G \otimes (A_1^\G \otimes \wedge^i V^\vee))$$
\begin{eqnarray*}
D_i &:=& A_{n+1}^\G \otimes_{A_n^\G \otimes A_1^\G} (A_n^\G \otimes (A_1^\G \otimes \wedge^i V^\vee)) \otimes_{A_n^\G} \\
& & \bigoplus_{k+l=2} (A_n^\G \otimes (\wedge^k V^\vee \otimes A_1^\G)) \otimes_{A_n^\G \otimes A_1^\G} A_{n+1}^\G \otimes_{A_1^\G \otimes A_n^\G} ((A_1^\G \otimes \wedge^l V^\vee) \otimes A_n^\G).
\end{eqnarray*}
We will prove the case $n=0$ which is the same as the general case but simplifies the notation. In this case the composition is
$$(A_1^\G \otimes \wedge^i V^\vee) \xrightarrow{f} (A_1^\G \otimes \wedge^i V^\vee) \otimes_{A_0^\G} \left( \bigoplus_{k+l=2} (\wedge^k V^\vee \otimes A_1^\G) \otimes_{A_1^\G} (A_1^\G \otimes \wedge^{l} V^\vee) \right) \{2\} \xrightarrow{g} (A_1^\G \otimes \wedge^i V^\vee).$$

The image of any element under $f$ consists of three terms. The only term that will be non-zero after applying $g$ will be the term 
$$ (A_1^\G \otimes \wedge^i V^\vee) \otimes_{A_0^\G} (\wedge^{2-i} V^\vee \otimes A_1^\G) \otimes_{A_1^\G} (A_1^\G \otimes \wedge^i V^\vee).$$

If $i=0$ and $a \in A_1^\G$ then 
$$a \mapsto - a \otimes ((w_1 \wedge w_2 \otimes 1) \otimes 1) \mapsto (a \otimes 1) \cdot \omega(w_1 \wedge w_2) = a$$
which proves that the composition is the identity when $i=0$. 

If $i=1$ then 
\begin{eqnarray*}
(f, \gamma) \otimes w 
&\mapsto& ((f, \gamma) \otimes w) \otimes \left( - (w_1 \otimes 1) \otimes (1 \otimes w_2) + (w_2 \otimes 1) \otimes (1 \otimes w_1) \right) \\
&\mapsto& - ((f, \gamma) \otimes w_2) \omega(w \wedge w_1) + ((f, \gamma) \otimes w_1) \omega(w \wedge w_2) 
\end{eqnarray*}
where $a = (f, \gamma)$. Since all maps are linear in the $w$'s it suffices to check the two cases $w = w_1$ and $w = w_2$. In the first case the first term is zero and we get
$$(f, \gamma) \otimes w_1 \mapsto ((f, \gamma) \otimes w_1) \omega(w_1 \wedge w_2) = (f, \gamma) \otimes w_1$$
while in the second case 
$$(f, \gamma) \otimes w_2 \mapsto -((f, \gamma) \otimes w_2) \otimes \omega(w_2 \wedge w_1) = (f,\gamma) \otimes w_2.$$

If $i=2$ and $a \in A_1^\G$ then 
$$a \otimes (w \wedge w') \mapsto (a \otimes (w \wedge w')) \otimes 1 \otimes (1 \otimes (w_1 \wedge w_2)) \mapsto a \otimes (w_1 \wedge w_2) \omega(w \wedge w')$$
and it is easy to check this equals $a \otimes (w \wedge w')$ so that the composition is the identity. 

This concludes the proof of the first relation in (1). The fact that the composition $\fQ \xrightarrow{\adj I} \fQ \fP \fQ [-1]\{1\} \xrightarrow{I \adj} \fQ$ is also the identity follows similarly. 

\subsubsection{The composition $\fP \xrightarrow{\adj I} \fP \fQ \fP [1]\{-1\} \xrightarrow{I \adj} \fP$.} This is the same as the composition 
\begin{equation*}
\xymatrix{
A_{n+1}^\G \otimes_{A_n^\G \otimes A_1^\G} (A_n^\G \otimes \k[\G]) \ar[r] & 
A_{n+1}^\G \otimes_{A_n^\G \otimes A_1^\G} (A_n^\G \otimes \k[\G] \otimes \k[\G]) \otimes_{A_n^\G \otimes A_1^\G} A_{n+1}^\G \otimes_{A_n^\G \otimes A_1^\G} (A_n^\G \otimes \k[\G]) \ar[d] \\
 & A_{n+1}^\G \otimes_{A_n^\G \otimes A_1^\G} (A_n^\G \otimes \k[\G]).
}
\end{equation*}
Using (\ref{eq:adj4}) the first map is given by
$$1 \otimes (1 \otimes 1) \mapsto \sum_{i=0}^n \sum_{\gamma \in \G} s_i \dots s_n \otimes (1 \otimes \gamma \otimes \gamma^{-1}) \otimes s_n \dots s_i \otimes (1 \otimes 1).$$
Composing with the second map all terms are zero unless $i=0$ (recall that $s_i \dots s_n = 1$ if $i=0$) in which case we get the composition 
\begin{eqnarray*}
1 \otimes (1 \otimes 1) 
&\mapsto& \sum_{\gamma \in \G} 1 \otimes (1 \otimes \gamma \otimes \gamma^{-1}) \otimes 1 \otimes (1 \otimes 1) \\
&\mapsto& \sum_{\gamma \in \G} 1 \otimes_{A_n^\G \otimes A_1^\G} (1 \otimes \gamma) \otimes_{A_n^\G} (\tr(\gamma^{-1})) \\
&=& 1 \otimes_{A_n^\G \otimes A_1^\G} (1 \otimes 1) \otimes_{A_n^\G} \tr(1) \\
&=& 1 \otimes (1 \otimes 1)
\end{eqnarray*}
where two get the last two equalities we use that $\tr(\gamma^{-1}) = 0$ unless $\gamma = 1$ in which case $\tr(1) = 1$. This proves that the composition $\fP \xrightarrow{\adj I} \fP \fQ \fP [1] \xrightarrow{I \adj} \fP$ is the identity. The fact that $\fQ \xrightarrow{I \adj} \fQ \fP \fQ [1] \xrightarrow{\adj I} \fQ$ is the identity follows similarly. 

\subsection{Dots and adjunctions} 

Next we study how natural transformations of $\fP$ and $\fQ$ interact with the adjunctions. We will show that the following pairs of maps are equal for any $b \in B^\G$ (for simplicity we omit shifts in this section)
\begin{enumerate}
\item $\fQ \fP \xrightarrow{I X(b)} \fQ \fP \xrightarrow{\adj} \id$ and $\fQ \fP \xrightarrow{X(b) I} \fQ \fP \xrightarrow{\adj} \id$
\item $\fP \fQ \xrightarrow{X(b) I} \fP \fQ \xrightarrow{\adj} \id$ and $\fP \fQ \xrightarrow{I X(b)} \fP \fQ \xrightarrow{\adj} \id$
\item $\id \xrightarrow{\adj} \fP \fQ \xrightarrow{X(b) I} \fP \fQ $ and $\id \xrightarrow{\adj} \fP \fQ \xrightarrow{I X(b)} \fP \fQ $
\item $\id \xrightarrow{\adj} \fQ \fP \xrightarrow{I X(b)} \fQ \fP $ and $\id \xrightarrow{\adj} \fQ \fP \xrightarrow{X(b) I} \fQ \fP $.
\end{enumerate}
Equalities (1) and (2) are equivalent to the following defining relations for 2-morphisms in $\H'_\G$:
$$
\begin{tikzpicture}
\draw (0,0) -- (0,-.5)[->] [very thick];
\filldraw [blue] (0,-.25) circle (2pt);
\draw (0,-.25) node [anchor=east] [black] {$b$};
\draw (1,0) -- (1,-.5) [very thick];
\draw (0,0) arc (180:0:.5) [very thick];
\draw (1.5,0) node {=};
\draw (2,0) -- (2,-.5)[->] [very thick];
\filldraw [blue] (3,-.25) circle (2pt);
\draw (3,-.25) node [anchor=west] [black] {$b$};
\draw (3,0) -- (3,-.5) [very thick];
\draw (2,0) arc (180:0:.5) [very thick];

\draw (5,0) -- (5,-.5) [very thick];
\filldraw [blue] (5,-.25) circle (2pt);
\draw (5,-.25) node [anchor=east] [black] {$b$};
\draw (6,0) -- (6,-.5)[->] [very thick];
\draw (5,0) arc (180:0:.5) [very thick];
\draw (6.5,0) node {=};
\draw (7,0) -- (7,-.5) [very thick];
\filldraw [blue] (8,-.25) circle (2pt);
\draw (8,-.25) node [anchor=west] [black] {$b$};
\draw (8,0) -- (8,-.5) [->][very thick];
\draw (7,0) arc (180:0:.5) [very thick];
\end{tikzpicture}
$$
Notice that relations (3) and (4) above follow formally from (1) and (2) via the adjointness properties. Subsequently we only check (1) and (2). Also, since any $b$ is the composition of $v$'s and $\gamma$'s it suffices to consider the cases when $b = \gamma$ and $b = v$. 

\subsubsection{Relation (1) when $b = \gamma$.} 
Recall that
\begin{equation}\label{eq:local1}
(n)Q^\G \otimes_{A_{n+1}^\G} P^\G(n) = (A_n^\G \otimes \k[\G]) \otimes_{A_n^\G \otimes A_1^\G} A_{n+1}^\G \otimes_{A_n^\G \otimes A_1^\G} (A_n^\G \otimes \k[\G]).
\end{equation}
Now the map $\fQ \fP \xrightarrow{I X(\gamma)} \fQ \fP$ in (1) is 
$$(1 \otimes \gamma_1) \otimes a \otimes (1 \otimes \gamma_2) 
\mapsto (1 \otimes \gamma_1) \otimes a \otimes (1 \otimes \gamma_2 \gamma)$$
where we can assume $a = 1$ or $a = (n, n+1)$. If $a = 1$ then $\fQ \fP \xrightarrow{\adj} \id $ takes this to $\tr(\gamma_1 \gamma_2 \gamma) 1 \in A_n^\G$. If $a = (n, n+1)$ then this is mapped to zero. On the other hand, $\fQ \fP \xrightarrow{X(\gamma) I} \fQ \fP$ in (1) is 
$$(1 \otimes \gamma_1) \otimes a \otimes (1 \otimes \gamma_2) \mapsto (1 \otimes \gamma \gamma_1) \otimes a \otimes (1 \otimes \gamma_2)$$
which is then mapped to $\tr(\gamma \gamma_1 \gamma_2) 1 \in A_n^\G$ if $a = 1$ and to zero if $a = (n, n+1)$. So the two compositions agree. One can prove (2) when $b=\gamma$ similarly. 

\subsubsection{Relation (1) when $b=v$.}
To write down the composition $\fQ \fP \xrightarrow{I X(v)} \fQ \fP \xrightarrow{\adj} \id$ we resolve both copies of $\k[\G]$ in (\ref{eq:local1}). We get a map of complexes which is determined by the composition 
$$C_1 \xrightarrow{I \phi_v \oplus 0} C_2 \rightarrow C_3 \xrightarrow{\adj} A_n^\G$$
where 
$$C_1 := \oplus_{k+l=1} ((A_n^\G \otimes (\wedge^k V^\vee \otimes A_1^\G)) \otimes_{A_n^\G \otimes A_1^\G} A_{n+1}^\G \otimes_{A_n^\G \otimes A_1^\G} ((A_n^\G \otimes (A_1^\G \otimes \wedge^l V^\vee))$$
$$C_2 := (A_n^\G \otimes A_1^\G) \otimes_{A_n^\G \otimes A_1^\G} A_{n+1}^\G \otimes_{A_n^\G \otimes A_1^\G} (A_n^\G \otimes A_1^\G)$$
$$C_3 := (A_n^\G \otimes \k[\G]) \otimes_{A_n^\G \otimes A_1^\G} A_{n+1}^\G \otimes_{A_n^\G \otimes A_1^\G} (A_n^\G \otimes \k[\G]).$$
The second map is induced by the natural projection $A_1^\G \rightarrow \k[\G]$. 

If $k=1$ the first map is zero and if $k=0$ it is given by
$$(1 \otimes a) \otimes a' \otimes (1 \otimes (a'' \otimes w)) \mapsto (1 \otimes a) \otimes a' \otimes (1 \otimes a'' \la w, v \ra) \in C_2$$
where $a' \in A_{n+1}^\G$ is either $1$ or the transposition $(n,n+1)$. The only case when this term is not mapped to zero is when $a'=1$ and both $a, a''$ lie in degree zero. So let $a = \gamma$ and $a'' = \gamma''$, in which case
\begin{equation}\label{eq:a}
\adj \circ (I \phi_v) ((1 \otimes \gamma) \otimes 1 \otimes (1 \otimes (\gamma'' \otimes w)) = 1 \la w, v \ra \tr(\gamma \gamma'') \in A_n^\G.
\end{equation}

Similarly, $\fQ \fP \xrightarrow{X(v) I} \fQ \fP \xrightarrow{\adj} \id$ is determined by the composition $C_1 \xrightarrow{0 \oplus \psi_v I} C_2 \rightarrow C_3 \xrightarrow{\adj} A_n^\G$. This is zero if $k=0$.  If $k=1$, the composition is also zero except on terms of the form 
\begin{equation}\label{eq:b}
(1 \otimes (w \otimes \gamma)) \otimes 1 \otimes (1 \otimes \gamma'') \mapsto - (1 \otimes \la v, w \ra \gamma ) \otimes 1 \otimes (1 \otimes \gamma'') \mapsto - 1 \la v, w \ra \tr(\gamma \gamma'') \in A_n^\G.
\end{equation}

Although these two compositions are not equal as maps of complexes, they are homotopic to each other. To see this we need to define a map $h_v$ of $(A_n^\G, A_n^\G)$ bimodules such that the following diagram commutes
\begin{equation}\label{diagram1}
\xymatrix{
(A_n^\G \otimes A_1^\G) \otimes_{A_n^\G \otimes A_1^\G} A_{n+1}^\G \otimes_{A_n^\G \otimes A_1^\G} (A_n^\G \otimes A_1^\G) \ar@{-->}[rd]^{h_v} & \\
C_1 \ar[r]_{\adj \circ (I \phi_v \oplus 0) - \adj \circ (0 \oplus \psi_v I)} \ar[u] & A_n^\G.
}
\end{equation}
The map $h_v$ is determined by where it takes elements of the form 
$$(1 \otimes a) \otimes a' \otimes (1 \otimes a'')$$
where $a,a'' \in A_1^\G$ and either $a' = 1$ or $a' = (n, n+1)$. If $a' = (n,n+1)$ then $h_v$ sends $(1 \otimes a) \otimes a' \otimes (1 \otimes a'')$ to zero. If $a'=1$ then $h_v$ maps it to $\la aa'', v \ra 1$, where $\la aa'', v \ra$ is zero unless $aa''$ is of degree one, in which case $\la aa'', v \ra = \la w, v \ra \tr(\gamma)$ for $aa'' = (w, \gamma) \in A_1^\G$.

We now show that the diagram (\ref{diagram1}) communtes.  Consider first the term where $k=0$ in $C_1$. We see that
$$C_1 \ni (1 \otimes \gamma) \otimes 1 \otimes (1 \otimes (\gamma'' \otimes w)) \mapsto (1 \otimes \gamma) \otimes 1 \otimes (1 \otimes (\gamma'' w, \gamma'')) \xrightarrow{h_v} 1 \la \gamma \gamma'' w, v \ra \tr(\gamma \gamma'').$$ 
This composition is zero if $\gamma \gamma'' \ne 1$ and equal to $1 \la w, v \ra \tr(1)$ otherwise. 

Now $\adj \circ (0 \oplus \psi_v I)$ kills this term so, using (\ref{eq:a}), we see this is the same as 
$$(\adj \circ (I \phi_v \oplus 0) - \adj \circ (0 \oplus \psi_v I)) \left( (1 \otimes \gamma) \otimes 1 \otimes (1 \otimes (\gamma'' \otimes w)) \right).$$
So diagram (\ref{diagram1}) commutes when $k=0$. 

For the $k=1$ term in $C_1$ we have 
$$C_1 \ni (1 \otimes (w \otimes \gamma)) \otimes 1 \otimes (1 \otimes \gamma'') \mapsto (1 \otimes (w, \gamma)) \otimes 1 \otimes (1 \otimes \gamma'') \xrightarrow{h_v} 1 \la w, v \ra \tr(\gamma \gamma'').$$
Also, $\adj \circ (I \phi_v \oplus 0)$ acts by zero so using (\ref{eq:b}) we see that
$$(\adj \circ (I \phi_v \oplus 0) - \adj \circ (0 \oplus \psi_v I)) \left(1 \otimes (w \otimes \gamma)) \otimes 1 \otimes (1 \otimes \gamma'') \right) = 1 \la v, w \ra \tr(\gamma \gamma'').$$
So diagram (\ref{diagram1}) also commutes when $k=1$.  

\subsubsection{Relation (2) when $b=v$.} 
After resolving both copies of $\k[\G]$ in 
$$P^\G(n) \otimes_{A_n^\G} (n)Q^\G = A_{n+1}^\G \otimes_{A_n^\G \otimes A_1^\G} (A_n^\G \otimes \k[\G] \otimes \k[\G]) \otimes_{A_n^\G \otimes A_1^\G} A_{n+1}^\G$$
the composition $\fP \fQ \xrightarrow{X(v) I} \fP \fQ \xrightarrow{\adj} \id$ is induced by the map 
$$\oplus_{k+l=3} (A_1^\G \otimes \wedge^k V^\vee) \otimes (\wedge^l V^\vee \otimes A_1^\G) \xrightarrow{\phi_v I \oplus \phi'_v I} \oplus_{k+l=2} (A_1^\G \otimes \wedge^k V^\vee) \otimes (\wedge^l V^\vee \otimes A_1^\G) \rightarrow A_1^\G$$
given by 
\begin{eqnarray*}
(1 \otimes (w \wedge w')) \otimes (w'' \otimes 1) 
&\mapsto& \left( - \la w', v \ra 1 \otimes w) + \la w, v \ra 1 \otimes w' \right) \otimes (w'' \otimes 1) \\
&\mapsto& 1 \left( - \la w', v \ra \omega(w \wedge w'') + \la w, v \ra \omega(w' \wedge w'') \right)
\end{eqnarray*}
on one summand and similarly
\begin{eqnarray*}
(1 \otimes w) \otimes ((w' \wedge w'') \otimes 1) 
&\mapsto& 1 \la w, v \ra \otimes ((w' \wedge w'') \otimes 1) \\
&\mapsto& - 1 \la w, v \ra \omega(w' \wedge w'')
\end{eqnarray*}
on the second summand. 

On the other hand, the composition $\fP \fQ \xrightarrow{I X(v)} \fP \fQ \xrightarrow{\adj} \id$ is given by 
\begin{eqnarray*}
(1 \otimes (w \wedge w')) \otimes (w'' \otimes 1) 
&\mapsto& (1 \otimes (w \wedge w')) \otimes - \la v, w'' \ra 1 \\
&\mapsto& - 1 \la v, w'' \ra \omega(w \wedge w')
\end{eqnarray*}
on one summand and similarly
\begin{eqnarray*}
(1 \otimes w) \otimes ((w' \wedge w'') \otimes 1) 
&\mapsto& (1 \otimes w) \otimes \left( - \la v, w' \ra w'' \otimes 1 + \la v, w'' \ra w' \otimes 1 \right) \\
&\mapsto& 1 \left( - \la v, w' \ra \omega(w \wedge w'') + \la v, w'' \ra \omega(w \wedge w') \right)
\end{eqnarray*}
on the other summand. It is now easy to check that these maps are equal (by the linearity in $w$ it suffices to check the three cases $w=w'$, $w=w''$ and $w'=w''$). 

\subsection{Pitchfork relations}

Next we check that the following compositions are equal:
\begin{enumerate}
\item $\fP \xrightarrow{I \adj} \fP \fP \fQ [1]\{-1\} \xrightarrow{TI} \fP \fP \fQ [1]\{-1\}$ and $\fP \xrightarrow{\adj I} \fP \fQ \fP [1]\{-1\} \xrightarrow{IT} \fP \fP \fQ [1]\{-1\}$
\item $\fQ \xrightarrow{I \adj} \fQ \fP \fQ [1]\{-1\} \xrightarrow{TI} \fP \fQ \fQ [1]\{-1\}$ and $\fQ \xrightarrow{\adj I} \fP \fQ \fQ [1]\{-1\} \xrightarrow{IT} \fP \fQ \fQ [1]\{-1\}$
\item $\fP \fP \fQ \xrightarrow{IT} \fP \fQ \fP \xrightarrow{\adj I} \fP [1]\{-1\}$ and $\fP \fP \fQ \xrightarrow{TI} \fP \fP \fQ \xrightarrow{I \adj} \fP [1]\{-1\}$
\item $\fP \fQ \fQ \xrightarrow{IT} \fP \fQ \fQ \xrightarrow{\adj I} \fQ [1]\{-1\}$ and $\fP \fQ \fQ \xrightarrow{TI} \fQ \fP \fQ \xrightarrow{I \adj} \fQ [1]\{-1\}$
\end{enumerate}
This will check the following defining isotopy relations for 2-morhpisms in $\H'_\G$:
\begin{align*}
\begin{tikzpicture}
   \draw[<-,very thick] (0.5,1) .. controls (0.5,0) and (1.5,0) .. (1.5,1);
  \draw[<-,very thick] (1,1) to (0.5,0);
  \draw (2.25,0.5) node {=};
  \draw[<-,very thick] (3,1) .. controls (3,0) and (4,0) .. (4,1);
  \draw[<-,very thick] (3.5,1) to (4,0);
   \draw[<-,very thick] (6.5,1) .. controls (6.5,0) and (7.5,0) .. (7.5,1);
  \draw[->,very thick] (7,1) to (6.5,0);
  \draw (8.25,0.5) node {=};
  \draw[<-,very thick] (9,1) .. controls (9,0) and (10,0) .. (10,1);
  \draw[->,very thick] (9.5,1) to (10,0);
\end{tikzpicture}
 \end{align*}

\begin{align*}
\begin{tikzpicture}
  \draw[->,very thick] (0.5,0) .. controls (0.5,1) and (1.5,1) .. (1.5,0);
  \draw[<-,very thick] (1.5,1) to (1,0);
  \draw (2.25,0.5) node {=};
  \draw[->,very thick] (3,0) .. controls (3,1) and (4,1) .. (4,0);
  \draw[<-,very thick] (3,1) to (3.5,0);
  \draw[->,very thick] (6.5,0) .. controls (6.5,1) and (7.5,1) .. (7.5,0);
  \draw[->,very thick] (7.5,1) to (7,0);
  \draw (8.25,0.5) node {=};
  \draw[->,very thick] (9,0) .. controls (9,1) and (10,1) .. (10,0);
  \draw[->,very thick] (9,1) to (9.5,0);
\end{tikzpicture}
 \end{align*}

Notice that one also has four more relations where one reverses the orientation of the cups or cups, but these all follow formally from the relations above and the adjointness relations from Section \ref{sec:adjointrels}.  These pitchfork relations, together with the adjunction relations which allow the straightening of an S-shape, imply that any two diagrams without dots differing by a rel boundary planar isotopy define the same 2-morphism. 

\subsubsection{Pitchfork relations (1).} 
The first composition is a map 
$$A_{n+1}^\G \otimes_{A_n^\G \otimes A_1^\G} (A_n^\G \otimes \k[\G]) \xrightarrow{I \adj} C \xrightarrow{TI} C$$ 
where 
$$C := A_{n+1}^\G \otimes_{A_n^\G \otimes A_1^\G} (A_n^\G \otimes \k[\G]) \otimes_{A_n^\G} A_n^\G \otimes_{A_{n-1}^\G \otimes A_1^\G} (A_{n-1}^\G \otimes \k[\G] \otimes \k[\G]) \otimes_{A_{n-1}^\G \otimes A_1^\G} A_n^\G.$$
This composition is determined by 
\begin{eqnarray*}
1 \otimes (1 \otimes 1) 
&\mapsto& \sum_{i=0}^{n-1} \sum_{\gamma \in \G} 1 \otimes (1 \otimes 1) \otimes (s_i \dots s_{n-1}) \otimes (1 \otimes \gamma \otimes \gamma^{-1}) \otimes (s_{n-1} \dots s_i) \\
&=& \sum_{i=0}^{n-1} \sum_{\gamma \in \G} (s_i \dots s_{n-1}) \otimes (1 \otimes 1) \otimes 1 \otimes (1 \otimes \gamma \otimes \gamma^{-1}) \otimes (s_{n-1} \dots s_i) \\
&\mapsto& \sum_{i=0}^{n-1} \sum_{\gamma \in \G} (s_i \dots s_{n-1} s_n) \otimes (1 \otimes \gamma) \otimes 1 \otimes (1 \otimes 1 \otimes \gamma^{-1}) \otimes (s_{n-1} \dots s_i).
\end{eqnarray*}
On the other hand, the second composition in (1) is a map 
$$A_{n+1}^\G \otimes_{A_n^\G \otimes A_1^\G} (A_n^\G \otimes \k[\G]) \xrightarrow{\adj I} C_1 \xrightarrow{IT} C_2$$
where 
$$C_1 := A_{n+1}^\G \otimes_{A_n^\G \otimes A_1^\G} (A_n^\G \otimes \k[\G] \otimes \k[\G]) \otimes_{A_n^\G \otimes A_1^\G} A_{n+1}^\G \otimes_{A_n^\G \otimes A_1^\G} (A_n^\G \otimes \k[\G])$$
$$C_2 := A_{n+1}^\G \otimes_{A_n^\G \otimes A_1^\G} (A_n^\G \otimes \k[\G]) \otimes_{A_n^\G} A_n^\G \otimes_{A_{n-1}^\G \otimes A_1^\G} (A_{n-1}^\G \otimes \k[\G] \otimes \k[\G]) \otimes_{A_{n-1}^\G \otimes A_1^\G} A_n^\G.$$
It is determined by
\begin{eqnarray*}
1 \otimes (1 \otimes 1)
&\mapsto& \sum_{i=0}^{n} \sum_{\gamma \in \G} (s_i \dots s_n) \otimes (1 \otimes \gamma \otimes \gamma^{-1}) \otimes (s_n \dots s_i) \otimes (1 \otimes 1) \\
&=& \sum_{i=1}^{n} \sum_{\gamma \in \G} (s_i \dots s_n) \otimes (1 \otimes \gamma \otimes \gamma^{-1}) \otimes s_n \otimes (s_{n-1} \dots s_i \otimes 1)  \\
& & + \sum_{\gamma \in \G} 1 \otimes (1 \otimes \gamma \otimes \gamma^{-1}) \otimes 1 \otimes (1 \otimes 1) \\
&\mapsto& \sum_{i=1}^{n} \sum_{\gamma \in \G} (s_i \dots s_n) \otimes (1 \otimes \gamma) \otimes 1 \otimes (1 \otimes 1 \otimes \gamma^{-1}) \otimes (s_{n-1} \dots s_i).
\end{eqnarray*}
The two compositions are equal, thus we see that both sides of the first pitchfork relation are equal.

The proof of relation (2) is very similar to the proof of relation (1), and we omit the details.

\subsubsection{Pitchfork relation (3).} 
Resolving every copy of $\k[\G]$ the first composition in (3) is 
$$C_1 [-1]\{1\} \xrightarrow{IT} C_2 [-1]\{1\} \xrightarrow{\adj I} \bigoplus_l A_{n+1}^\G \otimes_{A_n^\G \otimes A_1^\G} (A_n^\G \otimes (A_1^\G \otimes \wedge^l V^\vee)) [1]\{-1\}$$
where 
\begin{eqnarray*}
C_1 &:=& \bigoplus_{k,l,m} A_{n+1}^\G \otimes_{A_n^\G \otimes A_1^\G} (A_n^\G \otimes (A_1^\G \otimes \wedge^k V^\vee)) \otimes_{A_n^\G} A_n^\G \otimes_{A_{n-1}^\G \otimes A_1^\G} (A_{n-1}^\G \otimes (A_1^\G \otimes \wedge^l V^\vee)) \\
& & \otimes_{A_{n-1}^\G} (A_{n-1}^\G \otimes (A_1^\G \otimes \wedge^m V^\vee)) \otimes_{A_{n-1}^\G \otimes A_1^\G} A_n^\G \\
C_2 &:=& \bigoplus_{k,l,m} A_{n+1}^\G \otimes_{A_n^\G \otimes A_1^\G} (A_n^\G \otimes (A_1^\G \otimes \wedge^k V^\vee)) \\
& & \otimes_{A_n^\G} (A_n^\G \otimes (\wedge^m V^\vee \otimes A_1^\G)) \otimes_{A_n^\G \otimes A_1^\G} A_{n+1}^\G \otimes_{A_n^\G \otimes A_1^\G} (A_n^\G \otimes (A_1^\G \otimes \wedge^l V^\vee)).
\end{eqnarray*}
The maps are given by 
\begin{eqnarray*}
& & 1 \otimes (1 \otimes (1 \otimes w)) \otimes 1 \otimes (1 \otimes (1 \otimes w')) \otimes (1 \otimes (1 \otimes w'')) \otimes 1 \\
&\mapsto& (-1)^{lm} 1 \otimes (1 \otimes (1 \otimes w)) \otimes (1 \otimes (1 \otimes w'')) \otimes s_n \otimes (1 \otimes (1 \otimes w')) \\
&\mapsto& (-1)^{lm} s_n \otimes (1 \otimes (1 \otimes w')) \omega(w \wedge w'')
\end{eqnarray*}
where $w \in \wedge^k V^\vee, w' \in \wedge^l V^\vee, w'' \in \wedge^m V^\vee$. Here we use the fact that the right crossing $T: \fP \fQ \rightarrow \fQ \fP$ acts on the resolutions of $\fP$ and $\fQ$ just as before, that is, by switching factors and adding a copy of $s_n$ in the middle.

On the other hand, the second composition in (3) is 
$$C_1 [-1]\{1\} \xrightarrow{TI} C_1 [-1]\{1\} \xrightarrow{I \adj} \bigoplus_l A_{n+1}^\G \otimes_{A_n^\G \otimes A_1^\G} (A_n^\G \otimes (A_1^\G \otimes \wedge^l V^\vee)) [1]\{-1\}.$$
The maps are given by 
\begin{eqnarray*}
& & 1 \otimes (1 \otimes (1 \otimes w)) \otimes 1 \otimes (1 \otimes (1 \otimes w')) \otimes (1 \otimes (1 \otimes w'')) \otimes 1 \\ 
&\mapsto& (-1)^{kl} s_n \otimes (1 \otimes (1 \otimes w')) \otimes 1 \otimes (1 \otimes (1 \otimes w)) \otimes (1 \otimes (1 \otimes w'')) \otimes 1 \\
&\mapsto& (-1)^{kl} s_n \otimes (1 \otimes (1 \otimes w')) \omega(w \wedge w'').
\end{eqnarray*}
As long as $kl \equiv lm \bmod 2$, this is the same composition as before. The only other possibility is that $k+m \equiv 1 \bmod 2$, in which case $\omega(w \wedge w'') = 0$ (and thus both sides of (3) vanish).  Thus relation (3) holds. Relation (4) is proved similarly. 

\subsection{Dots and crossings}

Next we check that dots move freely through crossings, namely that the following compositions are equal (for simplicity we omit shifts in this section):
\begin{enumerate}
\item $\fP \fP \xrightarrow{X(b) I} \fP \fP \xrightarrow{T} \fP \fP$ and $\fP \fP \xrightarrow{T} \fP \fP \xrightarrow{I X(b)} \fP \fP$ where $b \in B^\G$
\item $\fP \fP \xrightarrow{I X(b)} \fP \fP \xrightarrow{T} \fP \fP$ and $\fP \fP \xrightarrow{T} \fP \fP \xrightarrow{X(b) I} \fP \fP$ where $b \in B^\G$.
\end{enumerate}
$$
\begin{tikzpicture}[>=stealth]
\draw [->](0,0) -- (1,1) [very thick];
\draw [->](1,0) -- (0,1) [very thick];
\filldraw [blue](.25,.25) circle (2pt);
\draw (.25,.25) node [anchor=west] [black] {$b$};
\draw (1.5,.5) node{$=$};
\draw [->](2,0) -- (3,1) [very thick];
\draw [->](3,0) -- (2,1) [very thick];
\filldraw [blue](2.75,0.75) circle (2pt);
\draw (2.75,.75) node [anchor=east] [black] {$b$};
\draw  [shift={+(5,0)}][->](0,0) -- (1,1) [very thick];
\draw  [shift={+(5,0)}][->](1,0) -- (0,1) [very thick];
\filldraw  [shift={+(5,0)}][blue](.75,.25) circle (2pt);
\draw [shift={+(5,0)}] (.75,.25) node [anchor=east] [black] {$b$};
\draw  [shift={+(5,0)}](1.5,.5) node{$=$};
\draw  [shift={+(5,0)}][->](2,0) -- (3,1) [very thick];
\draw  [shift={+(5,0)}][->](3,0) -- (2,1) [very thick];
\filldraw  [shift={+(5,0)}][blue](2.25,0.75) circle (2pt);
\draw [shift={+(5,0)}] (2.25,.75) node [anchor=east] [black] {$b$};
\end{tikzpicture}
$$
Notice that the analogous relations involving the other three oriented crossings follow formally from the relations above and the adjunctions from the previous sections. 

\subsubsection{Relation (1) when $b = \gamma$.}
The first composition is $C \xrightarrow{X(\gamma) I} C \xrightarrow{T} C$ where 
$$C := A_{n+2}^\G \otimes_{A_n^\G \otimes A_1^\G \otimes A_1^\G} (A_n^\G \otimes \k[\G] \otimes \k[\G]).$$ 
The map is given by 
$$1 \otimes (1 \otimes 1 \otimes 1) \mapsto 1 \otimes (1 \otimes \gamma \otimes 1) \mapsto s_{n+1} \otimes (1 \otimes 1 \otimes \gamma).$$
Similarly, the second composition $C \xrightarrow{T} C \xrightarrow{I X(\gamma)} C$ is given by
$$1 \otimes (1 \otimes 1 \otimes 1) \mapsto s_{n+1} \otimes (1 \otimes 1 \otimes 1) \mapsto s_{n+1} \otimes (1 \otimes 1 \otimes \gamma).$$
Clearly these two compositions are the same. Relation (2) is proved similarly. 

\subsubsection{Relation (2) when $b = v$.}
We first replace each $\k[\G]$ by its resolution. Then $T$ lifts to a map of complexes
$$A_n^\G \otimes (A_1^\G \otimes \wedge^k V^\vee) \otimes (A_1^\G \otimes \wedge^l V^\vee) \rightarrow A_n^\G \otimes (A_1^\G \otimes \wedge^l V^\vee) \otimes (A_1^\G \otimes \wedge^k V^\vee)$$
given by 
$$a \otimes (a' \otimes w') \otimes (a'' \otimes w'') \mapsto (-1)^{kl} a \otimes (a'' \otimes w'') \otimes (a' \otimes w')$$
where $w' \in \wedge^k V^\vee$ and $w'' \in \wedge^l V^\vee$.

Now the composition $T \circ X(v)I$ is induced by 
$$a \otimes (a' \otimes w') \otimes (a'' \otimes w'') \mapsto a \otimes \phi_v(a' \otimes w') \otimes (a'' \otimes w'') \mapsto (-1)^{(k-1)l} a \otimes (a'' \otimes w'') \otimes \phi_v(a' \otimes w')$$
while $IX(v) \circ T$ is induced by 
$$a \otimes (a' \otimes w') \otimes (a'' \otimes w'') \mapsto (-1)^{kl} a \otimes (a'' \otimes w'') \otimes (a' \otimes w') \mapsto (-1)^{kl} a \otimes (a'' \otimes w'') \otimes (-1)^l \phi_v(a' \otimes w').$$
These two compositions are clearly equal which proves relation (1). Relation (2) is proved similarly.

\subsection{Composition of crossings relations}

First we have the symmetric group relations, namely that the following compositions are equal:
\begin{enumerate}
\item $\fP \fP \xrightarrow{T} \fP \fP \xrightarrow{T} \fP \fP$ and $\fP \fP \xrightarrow{II} \fP \fP$
\item $\fP \fP \fP \xrightarrow{IT} \fP \fP \fP \xrightarrow{TI} \fP \fP \fP \xrightarrow{IT} \fP \fP \fP$ and $\fP \fP \fP \xrightarrow{TI} \fP \fP \fP \xrightarrow{IT} \fP \fP \fP \xrightarrow{TI} \fP \fP \fP$.
\end{enumerate}
$$
\begin{tikzpicture}
\draw (0,0) .. controls (1,1) .. (0,2)[->][very thick] ;
\draw (1,0) .. controls (0,1) .. (1,2)[->] [very thick];
\draw (1.5,1) node {=};
\draw (2,0) --(2,2)[->][very thick] ;
\draw (3,0) -- (3,2)[->][very thick] ;

\draw  [shift={+(7,0)}](0,0) -- (2,2)[->][very thick];
\draw  [shift={+(7,0)}](2,0) -- (0,2)[->][very thick];
\draw  [shift={+(7,0)}](1,0) .. controls (0,1) .. (1,2)[->][very thick];
\draw  [shift={+(7,0)}](2.5,1) node {=};
\draw  [shift={+(7,0)}](3,0) -- (5,2)[->][very thick];
\draw  [shift={+(7,0)}](5,0) -- (3,2)[->][very thick];
\draw  [shift={+(7,0)}](4,0) .. controls (5,1) .. (4,2)[->][very thick];
\end{tikzpicture}
$$

These follow immediately from the definition of $T$ on upward pointing strands, since the action of $T$ is given by multiplication by a simple reflection in the symmetric group.  

On the other hand we also have the following relations involving oppositely pointing strands (for simplicity we omit shifts in the rest of this section):
\begin{enumerate}
\item $\fP \fQ \xrightarrow{II} \fP \fQ$ and $\fP \fQ \xrightarrow{T} \fQ \fP \xrightarrow{T} \fP \fQ$ 
\item the identity $\fQ \fP \xrightarrow{II} \fQ \fP$ can be decomposed as the composition $\fQ \fP \xrightarrow{T} \fP \fQ \xrightarrow{T} \fQ \fP$ plus the sum over a basis of elements $b \in B^\G$ of the composition 
$$\fQ \fP \xrightarrow{I X(b)} \fQ \fP \xrightarrow{\adj} \id \xrightarrow{\adj} \fQ \fP \xrightarrow{I X(b^\vee)} \fQ \fP.$$
\end{enumerate}
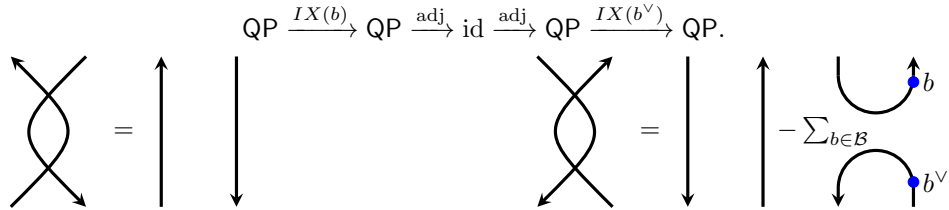
$$
\begin{tikzpicture}[>=stealth]
\draw  [shift={+(7,0)}](0,0) .. controls (1,1) .. (0,2)[<-][very thick];
\draw  [shift={+(7,0)}](1,0) .. controls (0,1) .. (1,2)[->] [very thick];
\draw  [shift={+(7,0)}](1.5,1) node {=};
\draw  [shift={+(7,0)}](2,0) --(2,2)[<-][very thick];
\draw  [shift={+(7,0)}](3,0) -- (3,2)[->][very thick];

\draw  [shift={+(7.3,0)}](3.5,1) node{$-\sum_{b \in \mathcal{B}}$};

\draw  [shift={+(7,0)}](4,1.75) arc (180:360:.5) [very thick];
\draw  [shift={+(7,0)}](4,2) -- (4,1.75) [very thick];
\draw  [shift={+(7,0)}](5,2) -- (5,1.75) [very thick][<-];
\draw  [shift={+(7,0)}](5,.25) arc (0:180:.5) [very thick];
\filldraw [blue]  [shift={+(7,0)}](5,1.66) circle (2pt);
\draw  [shift={+(7,0)}](5,1.66) node [anchor=west] {$b$};
\filldraw [blue]  [shift={+(7,0)}](5,0.33) circle (2pt);
\draw  [shift={+(7,0)}](5,.33) node [anchor=west] {$b^\vee$};
\draw  [shift={+(7,0)}](5,0) -- (5,.25) [very thick];
\draw  [shift={+(7,0)}](4,0) -- (4,.25) [very thick][<-];

\draw (0,0) .. controls (1,1) .. (0,2)[->][very thick];
\draw (1,0) .. controls (0,1) .. (1,2)[<-] [very thick];
\draw (1.5,1) node {=};
\draw (2,0) --(2,2)[->][very thick];
\draw (3,0) -- (3,2)[<-][very thick];
\end{tikzpicture}
$$
Relation (1) is straight-forward to check. The first composition is $C_1 \xrightarrow{T} C_2 \xrightarrow{T} C_1$ where 
\begin{eqnarray*}
C_1 &:=& A_n^\G \otimes_{A_{n-1}^\G \otimes A_1^\G} (A_{n-1}^\G \otimes \k[\G] \otimes \k[\G]) \otimes_{A_{n-1}^\G \otimes A_1^\G} A_n^\G \\
C_2 &:=& (A_n^\G \otimes \k[\G]) \otimes_{A_n^\G \otimes A_1^\G} A_{n+1}^\G \otimes_{A_n^\G \otimes A_1^\G} (A_n^\G \otimes \k[\G]).
\end{eqnarray*}
This composition is given by 
$$1 \otimes (1 \otimes 1 \otimes 1) \otimes 1 \mapsto (1 \otimes 1) \otimes s_n \otimes (1 \otimes 1) \mapsto 1 \otimes (1 \otimes 1 \otimes 1) \otimes 1$$
and is thus equal to the identity. 

Relation (2) is more interesting. The composition $\fQ \fP$ is given by the $(A_n^\G, A_n^\G)$-bimodule
$$(A_n^\G \otimes \k[\G]) \otimes_{A_n^\G \otimes A_1^\G} A_{n+1}^\G \otimes_{A_n^\G \otimes A_1^\G} (A_n^\G \otimes \k[\G]).$$
Since the tensor products are derived we need to resolve one copy of $\k[\G]$. We resolve the second copy of $\k[\G]$. We end with a complex where the terms are 
$$\bigoplus_{k=0}^2 (A_n^\G \otimes \k[\G]) \otimes_{A_n^\G \otimes A_1^\G} A_{n+1}^\G \otimes_{A_n^\G \otimes A_1^\G} (A_n^\G \otimes (A_1^\G \otimes \wedge^k V^\vee)) [k].$$
Now, up to the $(A_n^\G, A_n^\G)$ action any element is of the form 
$$(1 \otimes \gamma') \otimes 1 \otimes (1 \otimes (1 \otimes w)) \text{ or } (1 \otimes \gamma') \otimes s_n \otimes (1 \otimes (1 \otimes w))$$
where $w \in \wedge^k V^\vee$. 

If it is an element of the first form then the differential becomes zero and the map $T: \fQ \fP \rightarrow \fP \fQ$ maps it to zero. On the other hand, suppose $w = w_1$, then the composition
$$\fQ \fP \xrightarrow{I X(b)} \fQ \fP \xrightarrow{\adj} \id \xrightarrow{\adj} \fQ \fP \xrightarrow{I X(b^\vee)} \fQ \fP$$ 
where $b = (v, \gamma) \in B^\G$ is zero unless $v \in V$ in which case we get 
\begin{eqnarray*}
& & (1 \otimes \gamma') \otimes 1 \otimes (1 \otimes (1 \otimes w_1)) \\
&\xrightarrow{IX(v)}& (1 \otimes \gamma') \otimes 1 \otimes (1 \otimes (1 \la v, w_1 \ra))  \\
&\xrightarrow{IX(\gamma)}& (1 \otimes \gamma') \otimes 1 \otimes (1 \otimes (\gamma \la v, w_1 \ra)) \\
&\xrightarrow{\adj}& \tr(\gamma' \gamma) \la v, w_1 \ra \\
&\xrightarrow{\adj}& \tr(\gamma' \gamma) \la v, w_1 \ra (1 \otimes 1) \otimes 1 \otimes (1 \otimes (1 \otimes w_1 \wedge w_2)) \\
&\xrightarrow{IX(\gamma^{-1})}& \tr(\gamma' \gamma) \la v, w_1 \ra (1 \otimes 1) \otimes 1 \otimes (1 \otimes (\gamma^{-1} \otimes w_1 \wedge w_2)) \\
&\xrightarrow{IX(v^\vee)}& \tr(\gamma' \gamma) \la v, w_1 \ra (1 \otimes 1) \otimes 1 \otimes (1 \otimes (\gamma^{-1} \otimes w_1 \la w_2, v^\vee \ra - \gamma^{-1} \otimes w_2 \la w_1, v^\vee \ra)) 
\end{eqnarray*}
Now we sum over all $\gamma \in \G$ and over $v=v_1, v_2$. Since $\la w_i, v_j \ra = \delta_{ij}$ and $\tr(\gamma' \gamma) = 0$ unless $\gamma' = \gamma^{-1}$ we get 
$$ \la v_1, w_1 \ra (1 \otimes 1) \otimes 1 \otimes (1 \otimes (\gamma' \otimes w_1 \la w_2, v_2 \ra - \gamma' \otimes w_2 \la w_1, v_2 \ra )) = (1 \otimes \gamma') \otimes 1 \otimes (1 \otimes (1 \otimes w_1)).$$
Thus the composition acts by the identity if $w = w_1$ and similarly if $w = w_2$. Likewise, one can show that it also acts by the identity if $w = 1$ or $w = w_1 \wedge w_2$. 

If we are dealing with an element of the second form then it is in the image of the differential unless $k=0$. So we can assume $k=0$ in which case we need to consider 
$$(1 \otimes \gamma) \otimes s_n \otimes (1 \otimes 1) \in (A_n^\G \otimes \k[\G]) \otimes_{A_n^\G \otimes A_1^\G} A_{n+1}^\G \otimes_{A_n^\G \otimes A_1^\G} (A_n^\G \otimes A_1^\G).$$ 
The composition involving cups and caps maps it to zero (because the cap takes it to zero). Meanwhile the composition $\fQ \fP \xrightarrow{T} \fP \fQ \xrightarrow{T} \fP \fQ$ maps it as follows
$$(1 \otimes \gamma) \otimes s_n \otimes (1 \otimes 1) \mapsto 1 \otimes (1 \otimes 1 \otimes \gamma) \otimes 1 \mapsto (1 \otimes 1) \otimes s_n \otimes (\gamma \otimes 1) = (1 \otimes \gamma) \otimes s_n \otimes (1 \otimes 1).$$ 
Thus in both cases we see that the sum of the two compositions takes the element to itself. This proves relation (2). 

\subsection{Counter-clockwise circles and curls}

Finally we show that
\begin{enumerate}
\item the composition $\fP \xrightarrow{\adj I} \fQ \fP \fP [-1]\{1\} \xrightarrow{IT} \fQ \fP \fP [-1]\{1\} \xrightarrow{\adj I} \fP [-2]\{2\}$ is zero
\item the composition 
$$\id \xrightarrow{\adj} \fQ \fP [-1]\{1\} \xrightarrow{IX(b)} \fQ \fP [-1+|b|]\{1-|b|\} \xrightarrow{\adj} \id [-2+|b|]\{2-|b|\}$$ 
is equal to $\id \xrightarrow{\tr(b) I} \id$ for any $b \in B^\G$.
\end{enumerate}

$$
\begin{tikzpicture}[>=stealth]
\draw [shift={+(5,0)}](0,0) arc (180:360:0.5cm) [very thick];
\draw [shift={+(5,0)}][->](1,0) arc (0:180:0.5cm) [very thick];
\filldraw [shift={+(6,0)}][blue](0,0) circle (2pt);
\draw [shift={+(5,0)}](1,0) node [anchor=east] {$b$};
\draw [shift={+(5,0)}](1.75,0) node{$= \tr(b).$};

\draw  (0,0) .. controls (0,.5) and (.7,.5) .. (.9,0) [very thick];
\draw  (0,0) .. controls (0,-.5) and (.7,-.5) .. (.9,0) [very thick];
\draw  (1,-1) .. controls (1,-.5) .. (.9,0) [very thick];
\draw  (.9,0) .. controls (1,.5) .. (1,1) [->] [very thick];
\draw  (1.5,0) node {$=$};
\draw  (2,0) node {$0,$};

\end{tikzpicture}
$$
The first composition is straight forward and follows for the same reason that the composition 
$\id \rightarrow \fQ \fP [-1]\{1\} \rightarrow \id [-2]\{2\}$ is zero. Alternatively, the left-twist curl is a map $\fP \rightarrow \fP [-2]\{2\}$ and every such map is zero by degree considerations. 

For the second composition we just need to check the case $b = v_1 \wedge v_2$ (see Section \ref{sec:relremarks}). In that case one may note that working $\k^\times$-equivariantly the only degree zero maps $\id \rightarrow \id$ are given by scalars.  The result follows by the remark in the first part of Section \ref{sec:relremarks}. 

Alternatively, one can prove the second relation above directly. We must show that the composition
$$\id \xrightarrow{\adj} \fQ \fP [-1]\{1\} \xrightarrow{IX(v_1)} \fQ \fP \xrightarrow{IX(v_2)} \fQ \fP [1]\{-1\} \xrightarrow{\adj} \id$$ 
is equal to the identity. 

The first map is given by 
\begin{eqnarray*}
1 &\mapsto& (1 \otimes 1) \otimes 1 \otimes ((1 \otimes w_1 \wedge w_2) \otimes 1) \\
& & - (1 \otimes (w_1 \otimes 1)) \otimes 1 \otimes ((1 \otimes w_2) \otimes 1) + (1 \otimes (w_2 \otimes 1)) \otimes 1 \otimes ((1 \otimes w_1) \otimes 1) \\
& & - (1 \otimes (w_1 \wedge w_2 \otimes 1) \otimes 1 \otimes (1 \otimes 1)
\end{eqnarray*}
where the right hand terms lie in $\oplus_{k+l=2} (A_n^\G \otimes (\wedge^k V^\vee \otimes A_1^\G)) \otimes_{A_n^\G \otimes A_1^\G} A_{n+1}^\G \otimes_{A_1^\G \otimes A_n^\G} ((A_1^\G \otimes \wedge^l V^\vee) \otimes A_n^\G)$. The only term that survives after acting by $IX(v_1)$ and $IX(v_2)$ is the first one, which is then mapped 
\begin{eqnarray*}
(1 \otimes 1) \otimes 1 \otimes ((1 \otimes w_1 \wedge w_2) \otimes 1) 
&\xrightarrow{IX(v_1)}& (1 \otimes 1) \otimes 1 \otimes ((\la w_1, v_1 \ra 1 \otimes w_2 - \la w_2, v_1 \ra 1 \otimes w_2) \otimes 1) \\
&\xrightarrow{IX(v_2)}& (1 \otimes 1) \otimes 1 \otimes (\la w_2, v_2 \ra \la w_1, v_1 \ra 1 \otimes 1) \\
&\xrightarrow{\adj}& \tr(1) 1 = 1.
\end{eqnarray*}
Here we have used the fact that $\la w_i, v_j \ra = \delta_{ij}$.

%%%%%%%%%%%%%%%%%%%%%%%%%%%%%%%%%%%%%%%%%%%%%%%%%%%%%%%%%%%%%%%%%%%%%%%%%%%%%%%%%%%%%%%%%%%%%%%%%%

\section{The alternate 2-category $\H^\G$}\label{sec:graphical2}
In order to describe the structure of the 2-category $\H_\G$, it is convenient to introduce yet another 2-category $\H^\G$. The 2-categories $\H^\G$ and $\H_\G$ are not equivalent, but are rather ``Morita equivalent''. This means that they have the same 0-morphisms and 1-morphisms and that for any 1-morphism $A$ the algebras $\End_{\H_\G}(A)$ and $\End_{\H^\G}(A)$ are Morita equivalent. This should mean that the 2-representation theories of $\H^\G$ and $\H_\G$ are equivalent. 

To be more precise, recall that if $\G$ is a finite group then $\k[\G]$ is Morita equivalent (though not necessarily isomorphic) to a basic algebra $B(\G)$ which is a direct sum of several copies of $\k$, one copy for each irreducible representation of $\G$. It is sometimes convenient to replace $\k[\G]$ with $B(\G)$, whose representation theory is the same. 

In our case, if $\G$ is of type A then $\k[\G] \cong B(\G)$ and hence $\H_\G$ and $\H^\G$ are actually isomorphic. If $\G$ is not of type A then $\k[\G]$ will contain blocks of dimension greater than one and then $\k[\G] \not\cong B(\G)$. To go from $\H_\G$ to $\H^\G$ we simply replace $\k[\G]$ by $B(\G)$ in our constructions. Thus $\H^\G$ will be ``Morita equivalent'' to $\H_\G$ but its graphical calculus for 2-morphisms is easier to describe and visualize because there are fewer idempotents. Moreover, there is a canonical isomorphism of algebras $K_0(\H^\G) \cong K_0(\H_\G)$ (Proposition \ref{prop:Kgroups}) and so $\H^\G$ is also a categorification of $\h_\G$ which is of independent interest.

\subsection{The 2-category $\H^\G$}
We now give an explicit description of $\H^\G$. Fix an orientation $\epsilon$ of the Dynkin diagram.  If vertices $i$ and $j$ are connected by an edge in the diagram, we set $\epsilon_{ij} = 1$ if oriented edge with tail $i$ and head $j$ agrees with the orientation $\epsilon$, and $\epsilon_{ij} = -1$ if the oriented edge with tail $i$ and head $j$ disagrees with $\epsilon$.  We also set $\epsilon_{ij} = 0$ if $i$ and $j$ are not connected by an edge. Thus $\epsilon_{ij} = -\epsilon_{ji}$.  

We define an additive $\k$-linear 2-category ${\H'}^\G$ as follows. As before, the objects of ${\H'}^\G$ are indexed by the integers $\Z$. The 1-morphisms are generated by objects $P_i$ and $Q_i$, one for each $i \in I_\G$. Thus a 1-morphism of ${\H'}^\G$ is a finite composition (sequence) of $P_i$'s and $Q_j$'s.

The space of 2-morphisms between two 1-morphisms is the $\k$-algebra generated by suitable planar diagrams modulo local relations. The diagrams consist of oriented compact one-manifolds (possibly carrying dots and crossings) immersed into the plane strip $\R \times [0,1]$ modulo isotopies fixing the boundary. Each such strand is labeled by some $i \in I_\G$. 

Corresponding to the edge in the Dynkin diagram connecting vertex $i$ to vertex $j$ there is a 2-morphism $X: P_i \rightarrow P_j \la 1 \ra$ denoted by a solid dot: 
$$
\begin{tikzpicture}[>=stealth]
\draw (0,0) -- (0,1) [->][very thick];
\draw (0,-.25) node {$i$};
\filldraw [blue] (0,.5) circle (2pt);
\draw (0,1.25) node {$j$};
\end{tikzpicture}
$$
These dots are allowed to move freely along the one-manifold, including along the cups and caps and through crossings. However, as was the case with degree one dots in $\H'_\G$, dots between strands labeled by distinct nodes $i$ and $j$(which have degree one) on far away strands anticommute when they move past one another.

\begin{comment}
$$
\begin{tikzpicture}[>=stealth]
\draw (0,0) -- (0,2)[->][very thick] ;
\filldraw [blue] (0,.66) circle (2pt);
\draw (.5,.5) node {$\dots$};
\draw (1,0)--(1,2)[->] [very thick];
\filldraw [blue] (1,1.33) circle (2pt);
\draw (3,1) node {$= \ -$};
\draw [shift={+(2,0)}](2,0) --(2,2)[->][very thick] ;
\filldraw [shift={+(2,0)}][blue] (2,1.33) circle (2pt);
\draw [shift={+(2,0)}](2.5,.5) node {$\dots$};
\filldraw [shift={+(2,0)}][blue] (3,.66) circle (2pt);
\draw [shift={+(2,0)}](3,0) -- (3,2)[->][very thick] ;
\end{tikzpicture}
.
$$
\end{comment}

For each $i \in I_\G$ we also include as a defining 2-morphism a solid dot on a strand on a strand whose endpoints are both labeled $i$
$$
\begin{tikzpicture}[>=stealth]
\draw (0,0) -- (0,1) [->][very thick];
\draw (0,-.25) node {$i$};
\filldraw [blue] (0,.5) circle (2pt);
\draw (0,1.25) node {$i$};
\end{tikzpicture}
$$
and denote this map $X: P_i \rightarrow P_i \la 2 \ra$. These $ii$ dots also move freely on strands 

There is a relation which governs the composition of dots: 
\begin{equation}\label{eq:rel0'}
\begin{tikzpicture}[>=stealth]
\draw (0,0) -- (0,2) [->][very thick];
\draw (0,-.25) node {$i$};
\filldraw [blue] (0,.66) circle (2pt);
\draw (-.25,1) node {$j$};
\filldraw [blue] (0,1.33) circle (2pt);
\draw (0,2.25) node {$k$};
\draw (.5,1) node {$=$};
\draw (1.33,1) node {$\delta_{ik}\epsilon_{ij}$} ;
\draw (2,0) -- (2,2) [->][very thick];
\draw (2,-.25) node {$i$};
\filldraw [blue] (2,1) circle (2pt);
\draw (2,2.25) node {$i$};
\end{tikzpicture}
\end{equation}

Next, we have the following relations analogous to relations (\ref{eq:rel1}) - (\ref{eq:rel3}). For all $i,j,k \in I_\G$ we have 
\begin{equation}\label{eq:rel1'}
\begin{tikzpicture}[>=stealth]
\draw (0,0) .. controls (1,1) .. (0,2)[->][very thick] ;
\draw (1,0) .. controls (0,1) .. (1,2)[->] [very thick];
\draw (1.5,1) node {=};
\draw (2,0) --(2,2)[->][very thick] ;
\draw (3,0) -- (3,2)[->][very thick] ;
\draw (0,-.25) node {$i$};
\draw (1,-.25) node {$j$};
\draw (2,-.25) node {$i$};
\draw (3,-.25) node {$j$};

\draw [shift={+(5,0)}](0,0) -- (2,2)[->][very thick];
\draw [shift={+(5,0)}](2,0) -- (0,2)[->][very thick];
\draw [shift={+(5,0)}](1,0) .. controls (0,1) .. (1,2)[->][very thick];
\draw [shift={+(5,0)}](0,-.25) node {$i$};
\draw [shift={+(5,0)}](1,-.25) node {$j$};
\draw [shift={+(5,0)}](2,-.25) node {$k$};
\draw [shift={+(5,0)}](2.5,1) node {=};
\draw [shift={+(5,0)}](3,0) -- (5,2)[->][very thick];
\draw [shift={+(5,0)}](5,0) -- (3,2)[->][very thick];
\draw [shift={+(5,0)}](4,0) .. controls (5,1) .. (4,2)[->][very thick];
\draw [shift={+(5,0)}](3,-.25) node {$i$};
\draw [shift={+(5,0)}](4,-.25) node {$j$};
\draw [shift={+(5,0)}](5,-.25) node {$k$};
\end{tikzpicture}
\end{equation}

\begin{equation}\label{eq:rel2'}
\begin{tikzpicture}[>=stealth]
\draw [shift={+(2,1)}](0,0) arc (180:360:0.5cm) [very thick];
\draw [shift={+(2,1)}][->](1,0) arc (0:180:0.5cm) [very thick];
\filldraw [shift={+(2,1)}] [blue](1,0) circle (2pt);
\draw [shift={+(2,1)}](1.50,0) node{$= \mathbf{1}$};
\draw [shift={+(2,1)}](.5,-.75) node {$i$};
\draw [shift={+(2,1)}](.5,.75) node {$i$};

\draw  [shift={+(6,1)}](0,0) .. controls (0,.5) and (.7,.5) .. (.9,0) [very thick];
\draw  [shift={+(6,1)}](0,0) .. controls (0,-.5) and (.7,-.5) .. (.9,0) [very thick];
\draw  [shift={+(6,1)}](1,-1) .. controls (1,-.5) .. (.9,0) [very thick];
\draw  [shift={+(6,1)}](.9,0) .. controls (1,.5) .. (1,1) [->] [very thick];
\draw  [shift={+(6,1)}](1.5,0) node {$=$};
\draw  [shift={+(6,1)}](2,0) node {$0.$};
\draw (7,-0.25) node {$i$};
\end{tikzpicture}
\end{equation}

If $i \ne j$ then 
\begin{equation}\label{eq:rel3'}
\begin{tikzpicture}[>=stealth]
\draw (0,0) .. controls (1,1) .. (0,2)[<-][very thick];
\draw (1,0) .. controls (0,1) .. (1,2)[->] [very thick];
\draw (0,-.25) node {$i$};
\draw (1,-.25) node {$j$};
\draw (1.5,1) node {=};
\draw (2,0) --(2,2)[<-][very thick];
\draw (3,0) -- (3,2)[->][very thick];
\draw (2,-.25) node {$i$};
\draw (3,-.25) node {$j$};
\draw (3.75,1) node {$-\ \epsilon_{ij}$};

\draw (4,1.75) arc (180:360:.5) [very thick];
\draw (4,2) -- (4,1.75) [very thick];
\draw (5,2) -- (5,1.75) [very thick][<-];
\draw (5,.25) arc (0:180:.5) [very thick];
\filldraw [blue] (4.5,1.25) circle (2pt);
\filldraw [blue] (4.5,0.75) circle (2pt);
\draw (5,0) -- (5,.25) [very thick];
\draw (4,0) -- (4,.25) [very thick][<-];
\draw (4,-.25) node {$i$};
\draw (5,-.25) node {$j$};
\draw (4,2.25) node {$i$};
\draw (5,2.25) node {$j$};
\end{tikzpicture}
\end{equation}
while 
\begin{equation}\label{eq:rel4'}
\begin{tikzpicture}[>=stealth]
\draw (0,0) .. controls (1,1) .. (0,2)[<-][very thick];
\draw (1,0) .. controls (0,1) .. (1,2)[->] [very thick];
\draw (0,-.25) node {$i$};
\draw (1,-.25) node {$i$};
\draw (1.5,1) node {=};
\draw (2,0) --(2,2)[<-][very thick];
\draw (3,0) -- (3,2)[->][very thick];
\draw (2,-.25) node {$i$};
\draw (3,-.25) node {$i$};
\draw (3.5,1) node {$-$};
\draw (4,1.75) arc (180:360:.5) [very thick];
\draw (4,2) -- (4,1.75) [very thick];
\draw (5,2) -- (5,1.75) [very thick][<-];
\draw (5,.25) arc (0:180:.5) [very thick];
\filldraw [blue] (4.5,1.25) circle (2pt);
\draw (5,0) -- (5,.25) [very thick];
\draw (4,0) -- (4,.25) [very thick][<-];
\draw (4,-.25) node {$i$};
\draw (5,-.25) node {$i$};
\draw (4,2.25) node {$i$};
\draw (5,2.25) node {$i$};
\draw (5.5,1) node {$-$};
\filldraw [blue] (6.5,0.75) circle (2pt);
\draw (6,1.75) arc (180:360:.5) [very thick];
\draw (6,2) -- (6,1.75) [very thick];
\draw (7,2) -- (7,1.75) [very thick][<-];
\draw (7,.25) arc (0:180:.5) [very thick];
\draw (7,0) -- (7,.25) [very thick];
\draw (6,0) -- (6,.25) [very thick][<-];
\draw (6,-.25) node {$i$};
\draw (7,-.25) node {$i$};
\draw (6,2.25) node {$i$};
\draw (7,2.25) node {$i$};
\end{tikzpicture}
\end{equation}

We assign a $\Z$ grading on the space of planar diagrams by defining
$$
\begin{tikzpicture}[>=stealth]
\draw  (-.5,.5) node {$deg$};
\draw [->](0,0) -- (1,1) [very thick];
\draw [->](1,0) -- (0,1) [very thick];
\draw (1.5,.5) node{$ = 0$};
\end{tikzpicture}
$$
$$
\begin{tikzpicture}[>=stealth]
\draw  (-.5,-.25) node {$deg$};
\draw (0,0) arc (180:360:.5)[->] [very thick];
\draw (1.75,-.25) node{$ = deg$};
\draw (3.5,-.5) arc (0:180:.5) [->][very thick];
\draw (4.5,-.25) node{$ =-1$};
\end{tikzpicture}
$$
$$
\begin{tikzpicture}[>=stealth]
\draw  (-.5,-.25) node {$deg$};
\draw (0,0) arc (180:360:.5)[<-] [very thick];
\draw (1.75,-.25) node{$ = deg$};
\draw (3.5,-.5) arc (0:180:.5) [<-][very thick];
\draw (4.5,-.25) node{$ = 1$};
\end{tikzpicture}
$$
We define the degree of an $ij$ dot to be one (if $i \ne j$ are joined by an edge) and consequently the degree of an $ii$ dot to be two. Equipped with these assignments all the graphical relations are graded. This provides a $\Z$ grading on ${\H'}^{\G}$. 

Define $\H^\G$ to be the Karoubi envelope of ${\H'}^\G$. Since $\k[S_n] \subset \End(P_i^n)$ we let $P_i^{\lambda} := (P_i^n, e_\lambda)$ where $e_\lambda$ is the minimal idempotent of $\k[S_n]$ associated to the partition $\lambda$ of $n$. We define $Q_i^\lambda$ similarly. 

\subsection{The functor $\eta: \H^\G \longrightarrow \H_\G$}

We will define a functor ${\H'}^\G \longrightarrow \H_\G$ (this induces the functor $\eta: \H^\G \rightarrow \H_\G$ since $\H_\G$ is idempotent complete). 

To define this functor we map $P_i$ to $(P, e_{i,1})$ and by adjunction $Q_i$ to $(Q, e_{i,1})$. Now the 2-morphism $X: P_i \rightarrow P_i \la 2 \ra$ is mapped to the 2-morphism 
$$\frac{|\G|}{\dim(V_i)} e_{i,1} \cdot \omega \cdot e_{i,1}: (P, e_{i,1}) \rightarrow (P, e_{i,1}) \la 2 \ra$$
(the factor $\frac{|\G|}{\dim(V_i)}$ is there to ensure that the left relation in (\ref{eq:rel2'}) holds). 

Now suppose $i, j \in I_\G$ are joined by an oriented edge $i \rightarrow j$. Then $\Hom_{\H_\G}((P, e_{i,1}), (P, e_{j,1}))$ is one-dimensional, spanned by some map $\alpha_{ij}$. So we map the dot $X: P_i \rightarrow P_j \la 1 \ra$ to some multiple of $\alpha_{ij}$. To determine this multiple consider the composition 
$$\alpha_{ji} \circ \alpha_{ij}: (P, e_{i,1}) \rightarrow (P, e_{i,1}) \la 2 \ra$$
which is non-zero because of the structure of $B^\G$. Since $\Hom_{\H_\G}((P, e_{i,1}), (P, e_{i,1}))$ is spanned by $e_{i,1} \cdot \omega \cdot e_{i,1}$ this means that $\alpha_{ji} \circ \alpha_{ij}$ equals $e_{i,1} \cdot \omega \cdot e_{i,1}$ up to some non-zero multiple. We rescale $\alpha_{ij}$ so that $\alpha_{ji} \circ \alpha_{ij} = \frac{|\G|}{\dim(V_i)} e_{i,1} \cdot \omega \cdot e_{i,1}$. 

Finally, we map crossings $T:P_iP_j \rightarrow P_jP_i$ to crossings $T: (P, e_{i,1})(P, e_{j,1}) \rightarrow (P, e_{j,1}) (P, e_{i,1})$. 

\begin{prop}\label{prop:H^G->H_G} The above morphisms define a functor  $\eta: \H^\G \longrightarrow \H_\G$.
\end{prop}
\begin{proof}
We need to check relations (\ref{eq:rel0'}) - (\ref{eq:rel4'}). Relations (\ref{eq:rel1'}) and the right hand relation in (\ref{eq:rel2'}) are clear. The left hand relation in (\ref{eq:rel2'}) holds since 
$$\tr(\omega \cdot  e_{i,1}) = \frac{\dim(V_i)}{|\G|}$$
and we mapped $X: P_i \rightarrow P_i \la 2 \ra$ to $\frac{|\G|}{\dim(V_i)} e_{i,1} \cdot \omega \cdot e_{i,1}$. 

Relation (\ref{eq:rel0'}) holds by definition if $\epsilon_{ij}=1$. If $\epsilon_{ij}=-1$ then $\epsilon_{ji} = 1$ and so we have 
$$
\begin{tikzpicture}[>=stealth]
\draw (0,0) arc (180:360:0.5cm) [very thick];
\draw [->](1,0) arc (0:180:0.5cm) [very thick];
\filldraw [blue](1,0) circle (2pt);
\draw (- 0.80,0) node{${\mathbf{1}} = $};
\draw (.5,-.75) node {$i$};
\draw (.5,.75) node {$i$};
\draw (1.50,0) node{$ = $};

\draw [shift={+(2,0)}](0,0) arc (180:360:0.5cm) [very thick];
\draw [shift={+(2,0)}][->](1,0) arc (0:180:0.5cm) [very thick];
\filldraw [shift={+(2,0)}][blue](0.80,.40) circle (2pt);
\filldraw [shift={+(2,0)}][blue](0.80,-.40) circle (2pt);
\draw [shift={+(2,0)}](.5,-.75) node {$i$};
\draw [shift={+(2,0)}](.5,.75) node {$i$};
\draw [shift={+(2,0)}](1.25,0) node {$j$};
\draw [shift={+(2,0)}](1.75,0) node{$ = $};

\draw [shift={+(4.75,0)}](0,0) arc (180:360:0.5cm) [very thick];
\draw [shift={+(4.75,0)}][->](1,0) arc (0:180:0.5cm) [very thick];
\filldraw [shift={+(4.75,0)}][blue](0.80,.40) circle (2pt);
\filldraw [shift={+(4.75,0)}][blue](0.80,-.40) circle (2pt);
\draw [shift={+(4.75,0)}](.5,-.75) node {$j$};
\draw [shift={+(4.75,0)}](.5,.75) node {$j$};
\draw [shift={+(4.75,0)}](1.25,0) node {$i$};
\draw [shift={+(4.5,0)}](-.25,0) node {$-$};

\end{tikzpicture}
$$
where we use relation (\ref{eq:rel0'}) to obtain the second equality and the minus sign appears when one passes the two degree one dots past each other. It follows that 
$$\alpha_{ij} \circ \alpha_{ji} = - \alpha_{ji} \circ \alpha_{ij} =  - \frac{|\G|}{\dim(V_i)} e_{i,1} \cdot \omega \cdot e_{i,1}$$
which proves relation (\ref{eq:rel0'}) when $\epsilon_{ij}=-1$.

Next we prove relation (\ref{eq:rel4'}). This follows from the left relation in (\ref{eq:rel2}) in the definition of $\H'_\G$. More precisely, multiply endpoints of all strands of the left relation in (\ref{eq:rel2}) by the idempotent $e_{i,1}$. Now $e_{i,1} b e_{i,1} = 0$ unless $b$ has degree zero or two. Now if $b$ has degree zero then it belongs to $\k[\G]$. We use the basis of $\k[\G]$ consisting of matrix units $b$ which are zero everywhere except in one entry. Then $e_{i,1} b e_{i,1} = 0$ unless $b = e_{i,1}$. Thus the sum in (\ref{eq:rel2}) collapses and we get only one term corresponding to $b = e_{i,1}$. 

Similarly, if $b$ has degree two then the sum collapses and we get the term $b = (\omega, e_{i,1})$. Now 
$$e_{i,1}^\vee = \frac{|\G|}{\dim(V_i)} e_{i,1} \cdot \omega \cdot e_{i,1}$$
since $\tr(\omega \cdot e_{i,1}) = \frac{\dim(V_i)}{|\G|}$. This is precisely the image of $X: P_i \rightarrow P_i \la 2 \ra$ under $\eta$. Thus we get relation (\ref{eq:rel4'}). 

Finally, relation (\ref{eq:rel3'}) is similar. In this case only one of the terms in the sum  in (\ref{eq:rel2}) survives. To compute the coefficient $-\epsilon_{ij}$, cap off everything with a cap containing an $ij$ dot and use relations (\ref{eq:rel0'}) and (\ref{eq:rel2'}).
\end{proof}

We will prove in Section \ref{sec:K-theory} that $\eta$ induces an isomorphism on Grothendieck groups: 
 
\begin{prop}\label{prop:Kgroups}
The functor $\eta$ induces an isomorphism
$$ K_0(\eta): K_0(\H^\G) \longrightarrow K_0(\H_\G).$$
Thus both $\H^\G$ and $\H_\G$ categorify the Heisenberg algebra $\h_\G$.
\end{prop}

\subsection{Proof of Proposition \ref{prop:basicrel2}}\label{sec:proofprop2}

In this section we prove that relations from Proposition \ref{prop:basicrel2} hold in $\H^\G$. Sicne we have a functor $\eta: \H^\G \rightarrow \H_\G$ which sends $P_i$ to $P_i$ and $Q_i$ to $Q_i$ these relations also hold in $\H_\G$.  

The idempotent $e_{triv} = \frac{1}{n!} \sum_{g\in S_n} g$ corresponding to the trivial representation in $\k[S_n]$ defines the summand $P_i^{(n)} = (P_i^n,e_{triv})$ of $P_i^n$ in the category $\H^\G$.  These idempotents also define $Q_i^{(n)}$ by adjunction. We will draw $e_{triv}$ as a white rectangle labeled by $i^n$. We will also write a single line connecting two such rectangles instead of $n$ lines, for simplicity. Thus the identity 2-morphism of $P_i^{(n)}$ is drawn as one of the pictures
$$
\begin{tikzpicture}[>=stealth]
\draw (0,0) rectangle (1,.5);
\draw (.5,.25) node {$i^n$};
\draw (0,1.5) rectangle (1,2);
\draw (.5,1.75) node {$i^n$};
\draw (0,.5) -- (0,1.5) [->][very thick];
\draw (.25,.5) -- (.25,1.5) [->][very thick];
\draw (.66,1) node {$\dots$};
\draw (1,.5) -- (1,1.5) [->][very thick];
\draw  (1.5,1) node{$=$};
\draw (2,0) rectangle (3,.5);
\draw (2.5,.25) node {$i^n$};
\draw (2,1.5) rectangle (3,2);
\draw (2.5,1.75) node {$i^n$};
\draw (2.5,.5) -- (2.5,1.5) [->][very thick];
\end{tikzpicture}
$$

We first prove that
$$
P_i^{(n)} P_j^{(m)} \cong P_j^{(m)} P_i^{(n)} \text{ for all } i,j \in I_\G.
$$
The isomorphism is given by the 2-morphism
$$
\begin{tikzpicture}[>=stealth]
\draw (0,0) rectangle (1,.5);
\draw (.5,.25) node {$j^m$};
\draw (0,1.5) rectangle (1,2);
\draw (.5,1.75) node {$i^n$};
\draw (.5,.5) -- (2.5,1.5) [->][very thick];

\draw (2,0) rectangle (3,.5);
\draw (2.5,.25) node {$i^n$};
\draw(2,1.5) rectangle (3,2);
\draw (2.5,1.75) node {$j^m$};
\draw (2.5,.5) -- (.5,1.5) [->][very thick];
\end{tikzpicture}
$$
with inverse given by
$$
\begin{tikzpicture}[>=stealth]
\draw (0,0) rectangle (1,.5);
\draw (.5,.25) node {$i^n$};
\draw (0,1.5) rectangle (1,2);
\draw (.5,1.75) node {$j^m$};
\draw (.5,.5) -- (2.5,1.5) [->][very thick];

\draw (2,0) rectangle (3,.5);
\draw (2.5,.25) node {$j^m$};
\draw(2,1.5) rectangle (3,2);
\draw (2.5,1.75) node {$i^n$};
\draw (2.5,.5) -- (.5,1.5) [->][very thick];
\end{tikzpicture}
$$
Since we can pull strands apart using the left hand relation in (\ref{eq:rel1'}) it is easy to see that these two 2-morphisms are inverse of each other. The proof that
$$ Q_i^{(n)}Q_j^{(m)} \cong Q_j^{(m)} Q_i^{(n)} \text{ for all } i,j \in I_\G $$
is the same except that all strands now point downwards. The relation $Q_i^{(n)}P_j^{(m)} \cong P_j^{(m)}Q_i^{(m)}$ when $\la i,j \ra = -1$ follows via the same computation using relation (\ref{eq:rel3'}).

Next we show the more interesting relation
$$ Q_i^{(n)} P_i^{(m)} \cong \oplus_{k \ge 0} P_i^{(m-k)}Q_i^{(n-k)} \otimes H^\star(\P^k) \text{ for all } i \in I_\G.
$$
To define maps $Q_i^{(n)} P_i^{(m)} \rightarrow P_i^{(m-k)}Q_i^{(n-k)}$ consider the 2-morphisms
$$
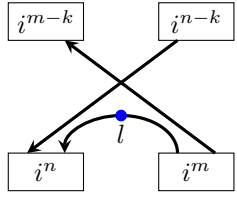
\begin{tikzpicture}[>=stealth]
\draw (0,0) rectangle (1,.5);
\draw (.5,.25) node {$i^n$};
\draw (0,2) rectangle (1,2.5);
\draw (.5,2.25) node {$i^{m-k}$};
\draw (.25,.5) -- (2.25,2) [<-][very thick];
\draw (2.25,.5) arc (0:180:.75cm and .5cm)[->][very thick];
\filldraw [blue](1.5,1) circle (2pt);
\draw (1.5,1) node [anchor=north] {$l$};
\draw (2,0) rectangle (3,.5);
\draw (2.5,.25) node {$i^{m}$};
\draw (2,2) rectangle (3,2.5);
\draw (2.5,2.25) node {$i^{n-k}$};
\draw (2.75,.5) -- (.75,2) [->][very thick];
\end{tikzpicture}
$$
where $0 \le l \le k$. In this picture $k$ of the strands coming out from the bottom are involved in the $k$ caps while the other strands are involved in crossings. The $l$ tells us that we put a single $ii$ dot on the first $l$ of the $k$ caps (because we have the idempotents it does not matter on which $l$ strands we place the dots). 

Taking the direct sum as $l$ ranges from $0$ to $k$, defines a 2-morphism from 
$$Q_i^{(n)} P_i^{(m)} \rightarrow P_i^{(m-k)} Q_i^{(n-k)} \otimes H^\star(\P^k).$$  
Taking the direct sum over $k$ we obtain a 2-morphism
$$ f: Q_i^{(n)} P_i^{(m)} \longrightarrow \oplus_{k \ge 0} P_i^{(m-k)} Q_i^{(n-k)} \otimes H^\star(\P^k). $$
We claim this morphism is invertible and we explicitly construct its inverse as a linear combinations of diagrams of the form
$$
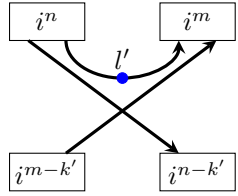
\begin{tikzpicture}[>=stealth]
\draw (0,0) rectangle (1,.5);
\draw (.5,.25) node {$i^{m-k'}$};
\draw (0,2) rectangle (1,2.5);
\draw (.5,2.25) node {$i^{n}$};
\draw (.75,.5) -- (2.75,2) [->][very thick];
\draw (.75,2) arc (180:360:.75cm and .5cm)[->][very thick];
\draw (1.5,1.5) node [anchor=south] {$l'$};
\filldraw [blue](1.5,1.5) circle (2pt);
\draw (2,0) rectangle (3,.5);
\draw (2.5,.25) node {$i^{n-k'}$};
\draw (2,2) rectangle (3,2.5);
\draw (2.5,2.25) node {$i^{m}$};
\draw (2.25,.5) -- (.25,2) [<-][very thick];
\end{tikzpicture}
$$
where we put $l'$ dots on $l'$ of the $k'$ cup-like strands.

To do this we first compute the following graphical relation in $\End(Q_i^{(n)}P_i^{(m)})$:
\begin{equation}\label{eq:main}
\begin{tikzpicture}[>=stealth]
\draw  (-1.5,2.25) node {$\sum_{0 \le l \le k} c_{m,n}^{k,l}$};
\draw (0,0) rectangle (1,.5);
\draw (.5,.25) node {$i^n$};
\draw (0,2) rectangle (1,2.5);
\draw (.5,2.25) node {$i^{m-k}$};
\draw (.25,.5) -- (2.25,2) [<-][very thick];
\draw (2.25,.5) arc (0:180:.75cm and .5cm)[->][very thick];
\filldraw [blue](1.5,1) circle (2pt);
\draw (1.5,1) node [anchor=north] {$l$};
\draw (2,0) rectangle (3,.5);
\draw (2.5,.25) node {$i^{m}$};
\draw (2,2) rectangle (3,2.5);
\draw (2.5,2.25) node {$i^{n-k}$};
\draw (2.75,.5) -- (.75,2) [->][very thick];
\draw [shift={+(0,2)}](0,2) rectangle (1,2.5);
\draw [shift={+(0,2)}](.5,2.25) node {$i^{n}$};
\draw  [shift={+(0,2)}](.75,.5) -- (2.75,2) [->][very thick];
\draw  [shift={+(0,2)}](.75,2) arc (180:360:.75cm and .5cm)[->][very thick];
\draw  [shift={+(0,2)}](1.5,1.5) node [anchor=south] {$k-l$};
\filldraw  [shift={+(0,2)}][blue](1.5,1.5) circle (2pt);
\draw [shift={+(0,2)}] (2,2) rectangle (3,2.5);
\draw  [shift={+(0,2)}](2.5,2.25) node {$i^{m}$};
\draw  [shift={+(0,2)}](2.25,.5) -- (.25,2) [<-][very thick];

\draw  (3.5,2.25) node {$=$};

\draw (4,0) rectangle (5,.5);
\draw (4.5,.25) node {$i^n$};
\draw (4,4) rectangle (5,4.5);
\draw (4.5,4.25) node {$i^{n}$};
\draw (6,0) rectangle (7,.5);
\draw (6.5,.25) node {$i^m$};
\draw (6,4) rectangle (7,4.5);
\draw (6.5,4.25) node {$i^{m}$};
\draw (4.5,.5) -- (4.5,4) [<-][very thick];
\draw (6.5,.5) -- (6.5,4) [->][very thick];

\end{tikzpicture}
\end{equation}
where $ c_{m,n}^{k,l} = k! \binom{m}{k} \binom{n}{k} \binom{k}{l}.$

To see where this relation comes from consider the $k=0$ term in the left hand side (it has no cups, caps, or dots in it).  Expanding the idempotents  
$\begin{tikzpicture}[>=stealth]
\draw (0,0) rectangle (1,.5);
\draw (.5,.25) node {$i^n$};
\end{tikzpicture}
$
and 
$\begin{tikzpicture}[>=stealth]
\draw (0,0) rectangle (1,.5);
\draw (.5,.25) node {$i^m$};
\end{tikzpicture}
$
explicitly as a sum of permutations we rewrite the $k=0$ term as a linear combination of diagrams with many crossings. Then crossings involving two upward pointing strands or two downward pointing strands can be absorbed into the idempotents at the top or bottom of the diagram (notice that this is a consequence of using the idempotent $e_{triv}$, which satisfies $\sigma e_{triv} = e_{triv} \sigma = e_{triv}$ for any permutation $\sigma$). Thus we write the $k=0$ term as a diagram with no idempotents in the middle:
 $$
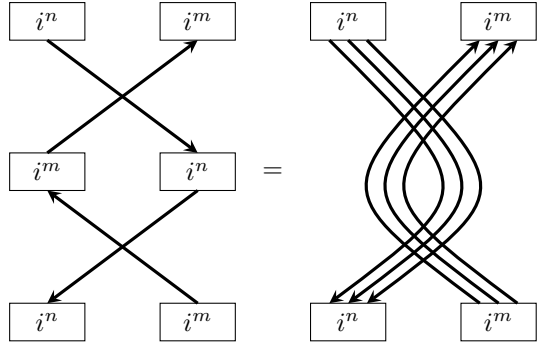
\begin{tikzpicture}[>=stealth]
\draw (0,0) rectangle (1,.5);
\draw (.5,.25) node {$i^n$};
\draw (0,2) rectangle (1,2.5);
\draw (.5,2.25) node {$i^{m}$};
\draw (.5,.5) -- (2.5,2) [<-][very thick];
\draw (2,0) rectangle (3,.5);
\draw (2.5,.25) node {$i^{m}$};
\draw (2,2) rectangle (3,2.5);
\draw (2.5,2.25) node {$i^{n}$};
\draw (2.5,.5) -- (.5,2) [->][very thick];

\draw  [shift={+(0,2)}](.5,.5) -- (2.5,2) [->][very thick];
\draw  [shift={+(0,2)}](0,2) rectangle (1,2.5);
\draw  [shift={+(0,2)}](.5,2.25) node {$i^{n}$};
\draw  [shift={+(0,2)}](2,2) rectangle (3,2.5);
\draw  [shift={+(0,2)}](2.5,2.25) node {$i^{m}$};
\draw  [shift={+(0,2)}](2.5,.5) -- (.5,2) [<-][very thick];

\draw  (3.5,2.25) node {$=$};
\draw (4,0) rectangle (5,.5);
\draw (4.5,.25) node {$i^n$};
\draw (4,4) rectangle (5,4.5);
\draw (4.5,4.25) node {$i^{n}$};
\draw (6,0) rectangle (7,.5);
\draw (6.5,.25) node {$i^m$};
\draw (6,4) rectangle (7,4.5);
\draw (6.5,4.25) node {$i^{m}$};

\draw (6.5,.5) .. controls (4.5,2) .. (6.5,4)[->][very thick] ;
\draw (4.5,.5) .. controls (6.5,2) .. (4.5,4)[<-] [very thick];
\draw (6.25,.5) .. controls (4.25,2) .. (6.25,4)[->][very thick] ;
\draw (4.25,.5) .. controls (6.25,2) .. (4.25,4)[<-] [very thick];
\draw (6.75,.5) .. controls (4.75,2) .. (6.75,4)[->][very thick] ;
\draw (4.75,.5) .. controls (6.75,2) .. (4.75,4)[<-] [very thick];

\end{tikzpicture}
$$ 
In both sides of the above picture there are $n$ strands emanating from an idempotent 
$\begin{tikzpicture}[>=stealth]
\draw (0,0) rectangle (1,.5);
\draw (.5,.25) node {$i^n$};
\end{tikzpicture}
$
and in the right hand side we have drawn more than one of these strands to emphasize the fact that there are many crossings. 

More generally we use the relation
$$
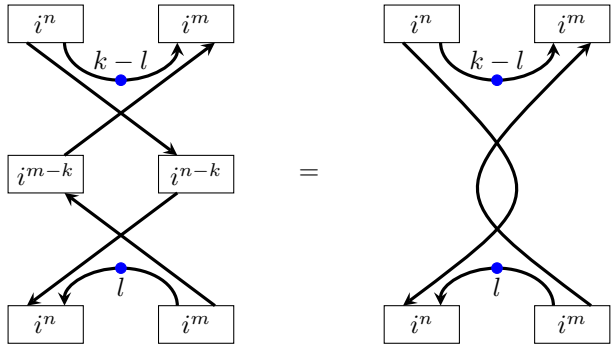
\begin{tikzpicture}[>=stealth]
\draw (0,0) rectangle (1,.5);
\draw (.5,.25) node {$i^n$};
\draw (0,2) rectangle (1,2.5);
\draw (.5,2.25) node {$i^{m-k}$};
\draw (.25,.5) -- (2.25,2) [<-][very thick];
\draw (2.25,.5) arc (0:180:.75cm and .5cm)[->][very thick];
\filldraw [blue](1.5,1) circle (2pt);
\draw (1.5,1) node [anchor=north] {$l$};
\draw (2,0) rectangle (3,.5);
\draw (2.5,.25) node {$i^{m}$};
\draw (2,2) rectangle (3,2.5);
\draw (2.5,2.25) node {$i^{n-k}$};
\draw (2.75,.5) -- (.75,2) [->][very thick];
\draw [shift={+(0,2)}](0,2) rectangle (1,2.5);
\draw [shift={+(0,2)}](.5,2.25) node {$i^{n}$};
\draw  [shift={+(0,2)}](.75,.5) -- (2.75,2) [->][very thick];
\draw  [shift={+(0,2)}](.75,2) arc (180:360:.75cm and .5cm)[->][very thick];
\draw  [shift={+(0,2)}](1.5,1.5) node [anchor=south] {$k-l$};
\filldraw  [shift={+(0,2)}][blue](1.5,1.5) circle (2pt);
\draw [shift={+(0,2)}] (2,2) rectangle (3,2.5);
\draw  [shift={+(0,2)}](2.5,2.25) node {$i^{m}$};
\draw  [shift={+(0,2)}](2.25,.5) -- (.25,2) [<-][very thick];

\draw  (4.0,2.25) node {$=$};
\draw  [shift={+(1,0)}](4,0) rectangle (5,.5);
\draw  [shift={+(1,0)}](4.5,.25) node {$i^n$};
\draw  [shift={+(1,0)}](4,4) rectangle (5,4.5);
\draw  [shift={+(1,0)}](4.5,4.25) node {$i^{n}$};
\draw  [shift={+(1,0)}](6,0) rectangle (7,.5);
\draw  [shift={+(1,0)}](6.5,.25) node {$i^m$};
\draw  [shift={+(1,0)}](6,4) rectangle (7,4.5);
\draw  [shift={+(1,0)}](6.5,4.25) node {$i^{m}$};

\draw  [shift={+(1,0)}](6.25,.5) arc (0:180:.75cm and .5cm)[->][very thick];
\filldraw [shift={+(1,0)}] [blue](5.5,1) circle (2pt);
\draw  [shift={+(1,0)}] (5.5,1) node [anchor=north] {$l$};

\draw  [shift={+(1,2)}](4.75,2) arc (180:360:.75cm and .5cm)[->][very thick];
\draw  [shift={+(1,2)}](5.5,1.5) node [anchor=south] {$k-l$};
\filldraw  [shift={+(1,2)}][blue](5.5,1.5) circle (2pt);

\draw  [shift={+(1,0)}](6.75,.5) .. controls (4.75,2) .. (6.75,4)[->][very thick] ;
\draw  [shift={+(1,0)}](4.25,.5) .. controls (6.25,2) .. (4.25,4)[<-] [very thick];

\end{tikzpicture}
$$
to write the entire left hand side of Equation (\ref{eq:main}) as a linear combination of diagrams without up-up or down-down crossings.

Now we start pulling apart the inner double crossings in each term one at a time using the relation (\ref{eq:rel4'}). At every stage we absorb newly created upward (or downward) pointing crossings in to the top (or bottom) idempotents. The result is a linear combination of diagrams of the form

$$
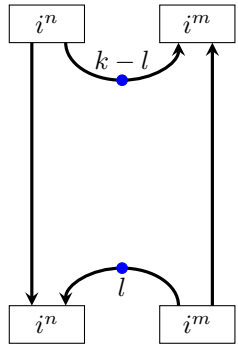
\begin{tikzpicture}[>=stealth]

\draw (0,0) rectangle (1,.5);
\draw (.5,.25) node {$i^n$};
\draw (2.5,.25) node {$i^m$};
\draw (2.25,.5) arc (0:180:.75cm and .5cm)[->][very thick];
\filldraw [blue](1.5,1) circle (2pt);
\draw (1.5,1) node [anchor=north] {$l$};
\draw (2,0) rectangle (3,.5);

\draw [shift={+(0,2)}](0,2) rectangle (1,2.5);

\draw  [shift={+(0,2)}](.75,2) arc (180:360:.75cm and .5cm)[->][very thick];
\draw  [shift={+(0,2)}](1.5,1.5) node [anchor=south] {$k-l$};
\filldraw  [shift={+(0,2)}][blue](1.5,1.5) circle (2pt);
\draw [shift={+(0,2)}] (2,2) rectangle (3,2.5);
\draw  [shift={+(0,2)}](2.5,2.25) node {$i^{m}$};
\draw  [shift={+(0,2)}](.5,2.25) node {$i^{n}$};
\draw (2.7,.5) --(2.7,4)[->][very thick] ;
\draw (.3,.5) --(.3,4)[<-] [very thick];

\end{tikzpicture}
$$

Keeping track of the constants carefully one arrives at Equation (\ref{eq:main}). The particular values of the constants $c_{m,n}^{k,l}$ are not really important apart from the fact that they are non-zero.  The key observation which follows from Equation (\ref{eq:main}) is that the identity 2-morphism of $Q_i^{(n)}P_i^{(m)}$ factors through $P_i^{(m-k)}Q_i^{(n-k)} \otimes H^\star(\P^k)$ as a composition of 2-morphisms 
$$ \id: Q_i^{(n)}P_i^{(m)} \xrightarrow{f} P_i^{(m-k)}Q_i^{(n-k)} \otimes H^\star(\P^k) \xrightarrow{g} Q_i^{(n)}P_i^{(m)} $$
where
$$
\begin{tikzpicture}[>=stealth]
\draw  (-2.5,.25) node {$f= $};
\draw  (-1.5,.25) node {$\sum_{0 \le l \le k}$};
\draw (0,-1) rectangle (1,-.5);
\draw (.5,-.75) node {$i^n$};
\draw (0,1) rectangle (1,1.5);
\draw (.5,1.25) node {$i^{m-k}$};
\draw (.25,-.5) -- (2.25,1) [<-][very thick];
\draw (2.25,-.5) arc (0:180:.75cm and .5cm)[->][very thick];
\filldraw [blue](1.5,0) circle (2pt);
\draw (1.5,0) node [anchor=north] {$l$};
\draw (2,-1) rectangle (3,-.5);
\draw (2.5,-.75) node {$i^{m}$};
\draw (2,1) rectangle (3,1.5);
\draw (2.5,1.25) node {$i^{n-k}$};
\draw (2.75,-.5) -- (.75,1) [->][very thick];
\end{tikzpicture}
$$
and 
$$
\begin{tikzpicture}[>=stealth]
\draw  (-3.0,3) node {$g=$};
\draw  (-1.5,3) node {$\sum_{0 \le l \le k} c_{m,n}^{k,l} $};
\draw (2,2) rectangle (3,2.5);
\draw (2.5,2.25) node {$i^{n-k}$};
\draw (0,2) rectangle (1,2.5);
\draw (.5,2.25) node {$i^{m-k}$};
\draw [shift={+(0,2)}](0,2) rectangle (1,2.5);
\draw [shift={+(0,2)}](.5,2.25) node {$i^{n}$};
\draw  [shift={+(0,2)}](.75,.5) -- (2.75,2) [->][very thick];
\draw  [shift={+(0,2)}](.75,2) arc (180:360:.75cm and .5cm)[->][very thick];
\draw  [shift={+(0,2)}](1.5,1.5) node [anchor=south] {$k-l$};
\filldraw  [shift={+(0,2)}][blue](1.5,1.5) circle (2pt);
\draw [shift={+(0,2)}] (2,2) rectangle (3,2.5);
\draw  [shift={+(0,2)}](2.5,2.25) node {$i^{m}$};
\draw  [shift={+(0,2)}](2.25,.5) -- (.25,2) [<-][very thick];
\end{tikzpicture}
$$

In addition, a similar graphical manipulation shows that
\begin{equation}\label{eq:main2}
\begin{tikzpicture}[>=stealth]
\draw (0,0) rectangle (1,.5);
\draw (.5,.25) node {$i^{m-k'}$};
\draw (2,0) rectangle (3,.5);
\draw (2.5,.25) node {$i^{n-k'}$};

\draw [shift={+(0,2)}](0,0) rectangle (1,.5);
\draw [shift={+(0,2)}](.5,.25) node {$i^n$};
\draw [shift={+(0,2)}](0,2) rectangle (1,2.5);
\draw [shift={+(0,2)}](.5,2.25) node {$i^{m-k}$};
\draw [shift={+(0,2)}](.25,.5) -- (2.25,2) [<-][very thick];
\draw [shift={+(0,2)}](2.25,.5) arc (0:180:.75cm and .5cm)[->][very thick];
\filldraw [shift={+(0,2)}][blue](1.5,1) circle (2pt);
\draw [shift={+(0,2)}](1.5,1) node [anchor=north] {$l$};
\draw [shift={+(0,2)}](2,0) rectangle (3,.5);
\draw [shift={+(0,2)}](2.5,.25) node {$i^{m}$};
\draw [shift={+(0,2)}](2,2) rectangle (3,2.5);
\draw [shift={+(0,2)}](2.5,2.25) node {$i^{n-k}$};
\draw [shift={+(0,2)}](2.75,.5) -- (.75,2) [->][very thick];

\draw (.75,.5) -- (2.75,2) [->][very thick];
\draw  (.75,2) arc (180:360:.75cm and .5cm)[->][very thick];
\draw  (1.5,1.5) node [anchor=south] {$k'-l'$};
\filldraw  [blue](1.5,1.5) circle (2pt);
\draw  (2.25,.5) -- (.25,2) [<-][very thick];
\draw  (3.5,2.25) node {$=$};

\draw  (5,2.25) node {$\frac{\delta_{k,k'}\delta_{l,l'}}{c_{k,l}^{m,n}}$};
\draw (6,0) rectangle (7,.5);
\draw (6.5,.25) node {$i^{m-k}$};
\draw (6,4) rectangle (7,4.5);
\draw (6.5,4.25) node {$i^{m-k}$};
\draw (8,0) rectangle (9,.5);
\draw (8.5,.25) node {$i^{n-k}$};
\draw (8,4) rectangle (9,4.5);
\draw (8.5,4.25) node {$i^{n-k}$};
\draw (6.5,.5) -- (6.5,4) [->][very thick];
\draw (8.5,.5) -- (8.5,4) [<-][very thick];

\end{tikzpicture}
\end{equation}
The fact that the left hand side is zero when $l \neq l'$ follows from the fact that when the middle idempotents are expanded as a linear combination of permutations, each term contains either a left twist curl (which is zero) or a strand with two $ii$ dots on it (which is also zero).

It follows from Equation (\ref{eq:main2}) that the identity 2-morphism of $\oplus_{k \ge 0} P_i^{(m-k)} Q_i^{(n-k)} \otimes H^\star(\P^k)$ factors as
$$\id : \oplus_{k \ge 0} P_i^{(m-k)}Q_i^{(n-k)} \otimes H^\star(\P^k) \xrightarrow{g} Q_i^{(n)} P_i^{(m)} \xrightarrow{f} \oplus_{k \ge 0} P_i^{(m-k)}Q_i^{(n-k)} \otimes H^\star(\P^k),
$$
while the identity 2-morphism of $Q_i^{(n)} P_i^{(m)}$ factors as $f\circ g$.
Thus $f$ and $g$ are inverse isomorphisms, as desired.

The proof that $Q_j^{(n)} P_i^{(m)} \cong P_i^{(m)} Q_j^{(n)} \oplus P_i^{(m-1)} Q_j^{(n-1)}$ when $\la i, j \ra = -1$ is similar to the proof above. In particular, the isomorphism $f'$ from the left hand side to the right hand side is given by
$$
\begin{tikzpicture}[>=stealth]
\draw  (-1,.25) node {$f'= $};
\draw (0,-1) rectangle (1,-.5);
\draw (.5,-.75) node {$i^n$};
\draw (0,1) rectangle (1,1.5);
\draw (.5,1.25) node {$j^{m}$};
\draw (.25,-.5) -- (2.25,1) [<-][very thick];
\draw (2,-1) rectangle (3,-.5);
\draw (2.5,-.75) node {$j^{m}$};
\draw (2,1) rectangle (3,1.5);
\draw (2.5,1.25) node {$i^{n}$};
\draw (2.75,-.5) -- (.75,1) [->][very thick];

\draw  (3.5,.25) node {$+$};

\draw [shift={+(4,0)}] (0,-1) rectangle (1,-.5);
\draw [shift={+(4,0)}] (.5,-.75) node {$i^n$};
\draw [shift={+(4,0)}] (0,1) rectangle (1,1.5);
\draw [shift={+(4,0)}] (.5,1.25) node {$j^{m-1}$};
\draw [shift={+(4,0)}] (.25,-.5) -- (2.25,1) [<-][very thick];
\draw [shift={+(4,0)}] (2.25,-.5) arc (0:180:.75cm and .5cm)[->][very thick];
\filldraw [shift={+(4,0)}] [blue](1.5,0) circle (2pt);
\draw [shift={+(4,0)}] (2,-1) rectangle (3,-.5);
\draw [shift={+(4,0)}] (2.5,-.75) node {$j^{m}$};
\draw [shift={+(4,0)}] (2,1) rectangle (3,1.5);
\draw[shift={+(4,0)}]  (2.5,1.25) node {$i^{n-1}$};
\draw [shift={+(4,0)}] (2.75,-.5) -- (.75,1) [->][very thick];
\end{tikzpicture}.
$$
Note that there are no terms of the form 
$$
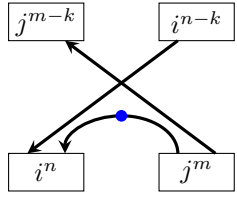
\begin{tikzpicture}[>=stealth]
\draw (0,-1) rectangle (1,-.5);
\draw (.5,-.75) node {$i^n$};
\draw (0,1) rectangle (1,1.5);
\draw (.5,1.25) node {$j^{m-k}$};
\draw (.25,-.5) -- (2.25,1) [<-][very thick];
\draw (2.25,-.5) arc (0:180:.75cm and .5cm)[->][very thick];
\filldraw [blue](1.5,0) circle (2pt);
\draw (2,-1) rectangle (3,-.5);
\draw (2.5,-.75) node {$j^{m}$};
\draw (2,1) rectangle (3,1.5);
\draw (2.5,1.25) node {$i^{n-k}$};
\draw (2.75,-.5) -- (.75,1) [->][very thick];
\end{tikzpicture}
$$
for $k > 1$ (this is in contrast to the isomorphism in the corresponding relation for $i=j$).  This difference between the $i=j$ and $\la i,j \ra = -1$ case comes from the fact that $ij$ dots anticommute.  This anticommutation of dots implies that the above diagram is equal to zero when $k>1$: grow a crossing between two adjacent strands from the $j^m$ symmetrizer, move the two degree one dots on those strands past each other (picking up a sign), and then absorb the crossing back into the bottom symmetrizer:
\begin{equation}\label{eq:zero}
\begin{tikzpicture}[>=stealth]
\draw (0,-1) rectangle (1,-.5);
\draw (.5,-.75) node {$i^n$};
\draw (0,1) rectangle (1,1.5);
\draw (.5,1.25) node {$j^{m}$};
\draw (.25,-.5) -- (.25,1) [->][very thick];
\draw (.75,-.5) -- (.75,1) [->][very thick];
\filldraw [blue](.25,0) circle (2pt);
\filldraw [blue](.75,.5) circle (2pt);

\draw (1.5,.33) node {$=$};

\draw [shift={+(2,0)}] (0,-1) rectangle (1,-.5);
\draw [shift={+(2,0)}] (.5,-.75) node {$i^n$};
\draw [shift={+(2,0)}] (0,1) rectangle (1,1.5);
\draw [shift={+(2,0)}] (.5,1.25) node {$j^{m}$};
\draw [shift={+(2,0)}] (.25,-.5) -- (.75,1) [->][very thick];
\draw [shift={+(2,0)}] (.75,-.5) -- (.25,1) [->][very thick];
\filldraw [shift={+(2,0)}] [blue](.35,-.2) circle (2pt);
\filldraw [shift={+(2,0)}] [blue](.6,0) circle (2pt);

\draw (3.5,.33) node {$=$};

\draw [shift={+(4,0)}] (0,-1) rectangle (1,-.5);
\draw [shift={+(4,0)}] (.5,-.75) node {$i^n$};
\draw [shift={+(4,0)}] (0,1) rectangle (1,1.5);
\draw [shift={+(4,0)}] (.5,1.25) node {$j^{m}$};
\draw [shift={+(4,0)}] (.25,-.5) -- (.75,1) [->][very thick];
\draw [shift={+(4,0)}] (.75,-.5) -- (.25,1) [->][very thick];
\filldraw [shift={+(4,0)}] [blue](.35,.7) circle (2pt);
\filldraw [shift={+(4,0)}] [blue](.6,.5) circle (2pt);

\draw (5.5,.33) node {$=$};

\draw  [shift={+(6,0)}] (0,-1) rectangle (1,-.5);
\draw  [shift={+(6,0)}] (.5,-.75) node {$i^n$};
\draw   [shift={+(6,0)}] (0,1) rectangle (1,1.5);
\draw  [shift={+(6,0)}] (.5,1.25) node {$j^{m}$};
\draw  [shift={+(6,0)}] (.25,-.5) -- (.25,1) [->][very thick];
\draw  [shift={+(6,0)}] (.75,-.5) -- (.75,1) [->][very thick];
\filldraw [shift={+(6,0)}]  [blue](.75,0) circle (2pt);
\filldraw [shift={+(6,0)}]  [blue](.25,.5) circle (2pt);

\draw (7.5,.33) node {$= -$};

\draw  [shift={+(8,0)}] (0,-1) rectangle (1,-.5);
\draw  [shift={+(8,0)}] (.5,-.75) node {$i^n$};
\draw   [shift={+(8,0)}] (0,1) rectangle (1,1.5);
\draw  [shift={+(8,0)}] (.5,1.25) node {$j^{m}$};
\draw  [shift={+(8,0)}] (.25,-.5) -- (.25,1) [->][very thick];
\draw  [shift={+(8,0)}] (.75,-.5) -- (.75,1) [->][very thick];
\filldraw [shift={+(8,0)}]  [blue](.25,0) circle (2pt);
\filldraw [shift={+(8,0)}]  [blue](.75,.5) circle (2pt);

\end{tikzpicture}
\end{equation}
Thus one cannot have more than one dot connecting the idempotents $i^n$ and $j^m$.  
$\square$

\subsection{Further relations among 1-morphisms}

\begin{prop}
For $i,j \in I_\G$ we have isomorphisms in $\H^\G$:
\begin{enumerate}
\item $P_i^{(m)}$ and $P_j^{(1^n)}$ commute while $Q_i^{(m)}$ and $Q_j^{(1^n)}$ also commute
\item $Q_i^{(1^n)} P_i^{(m)} \cong P_i^{(m)} Q_i^{(1^n)} \oplus P_i^{(m-1)} Q_i^{(1^{n-1})} \otimes_\k H^\star(\mathbb{P}^1) \oplus P_i^{(m-2)} Q_i^{(1^{n-2})}$.
\item $Q_j^{(1^n)} P_i^{(m)} \cong \oplus_{k \ge 0} P_i^{(m-k)} Q_j^{(1^{n-k})}$ if $\la i,j \ra = -1$.
\item $Q_j^{(1^n)} P_i^{(m)} \cong P_i^{(m)} Q_j^{(1^n)}$ if $\la i,j \ra = 0$.
\end{enumerate}
\end{prop}

The graphical proofs of this statement are analogous to and no more difficult than the proofs of Section \ref{sec:proofprop2}; we have omitted the details in the interest of space.  We do point out, however, the fact that there are many summands on the right hand side of the relations
$$
	Q_j^{(1^n)} P_i^{(m)} \cong \oplus_{k \ge 0} P_i^{(m-k)} Q_j^{(1^{n-k})}, \ \ \ \la i,j \ra = -1,
$$
in contrast to the decomposition of $Q_j^{(n)} P_i^{(m)}$ (which has only two summands).  This is because a trivial idempotent $e_{triv,i}$ and a sign idempotent labeled $e_{sign,j}$ can have more than one dot between them, thus there are non-trivial 2-morphisms from $P_i^{(m-k)} Q_j^{(1^{n-k})}$ to $Q_j^{(1^n)} P_i^{(m)}$.  In particular, the diagrammatic computation analogous to that of Equation \ref{eq:zero} does show anything when one of the idempotents is a sign idempotent and the other is a trivial idempotent; since $s_i e_{sign} = e_{sign} s_i = -e_{sign}$, absorbing a crossing into the sign idempotent introduces a $-1$.

\subsection{The functor $\Psi$}

We define a covariant autoequivalence $\Psi': {\H'}^\G \longrightarrow {\H'}^\G$ as follows. $\Psi'$ is the identity on objects and on 1-morphisms and also is the identity on cups, caps, and dots. On the other hand, $\Psi'$ acts as multiplication by $-1$ on any crossing between two adjacent strands. This map induces a covariant autoequivalence $\Psi$ on $\H^\G$ whose square is the identity.

Since a crossing of two upward pointing strands is multiplied by $-1$, it follows that $\Psi$ takes the idempotent 2-morphism $e_{triv}$ to the idempotent 2-morphism $e_{sign}$ (and vice versa). Thus
$$ \Psi(P_i^{(n)}) = P_i^{(1^n)}, \ \ \Psi(P_i^{(1^n)}) = P_i^{(n)}, \ \ \Psi(Q_i^{(n)}) = Q_i^{(1^n)}, \ \ \Psi(Q_i^{(1^n)}) = Q_i^{(n)}.
$$
It follows from the existence of the automorphism $\Psi$ that decompositions between products of $P_i^{(1^n)}$s and $Q_i^{(1^n)}$s have the same form as decompositions between products of
$P_i^{(n)}$s and $Q_i^{(n)}$s.  The autoequivalence $\Psi$ descends in the Grothendieck group to the automorphism $\psi$ of Section \ref{sec:psi}.

\section{Proof of Theorem \ref{thm:main1}}\label{sec:K-theory}

In this section we study the sequence of maps 
$$ \h_\G \xrightarrow{\pi} K_0(\H^\G) \xrightarrow{K_0(\eta)} K_0(\H_\G) \rightarrow \End(\oplus_{n\geq0} K_0(\mathcal{C}_n^\G)) \simeq \End(\mathcal{F}) $$
where $\mathcal{F}$ is isomorphic to the Fock space representation of $\h_\G$. We will show that both $\pi$ and $K_0(\eta)$ are isomorphisms,  thus proving Proposition \ref{prop:Kgroups} and Theorem \ref{thm:main1}.  In particular, this shows that both $\H_\G$ and $\H^\G$ categorify $\h_\G$. 

We begin with some preliminary results. 

\begin{lemma}\label{lem:1} 
There are no negative degree endomorphisms of $\prod_i P_i^{m_i} \prod_i Q_i^{n_i}$ while the algebra of degree zero endomorphisms is isomorphic to $\otimes_i \k[S_{m_i}] \otimes_i \k[S_{n_i}]$. 
\end{lemma}
\begin{proof}
Any negative degree endomorphism must contain a left-oriented cap or right oriented cup (since these are the only generating 2-morphisms of negative degree).  Such a cup or cap is either part of a left-twist curl (which makes the entire picture equal to zero) or part of a counter-clockwise circle. However, by the defining relations, the only non-zero counter-clockwise circles are those of non-negative degree (i.e. those carrying dots of degree two). Thus any negative degree endomorphism must be zero. 

Note that this argument would not apply to endomorphisms of a 1-morphism which contains a $P$ to the right of a $Q$.  For example, a left-oriented cap followed by a right oriented cap is a non-zero degree $-2$ endomorphism of $Q_iP_i$.

By the same argument any degree zero diagram describing an endomorphism of $\prod_i P_i^{m_i} \prod_i Q_i^{n_i}$ cannot have any right-twist curls or dots thus must be isotopic to a sum of braid-like diagrams (i.e. diagrams with no local minima or maxima and no dots) connecting a $P_i$ to a $P_i$ or a $Q_i$ to $Q_i$. The result now follows.  
\end{proof}

\begin{prop}\label{prop:indec} Any 1-morphism in $\H^\G$ decomposes uniquely into a finite direct sum of indecomposables. The indecomposable 1-morphisms in $\H^\G$ are of the form $\prod_i P_i^{\lambda_i} \prod_i Q_i^{\mu_i}$ for some partitions $\{\lambda_i, \mu_i\}$. 
\end{prop}
\begin{proof}
Lemma \ref{lem:1} implies that the space of endomorphisms (of a fixed degree) of any 1-morphism is finite dimensional. This means that $\H^\G$ satisfies the Krull-Schmidt property and hence any 1-morphism decomposes uniquely into a finite direct sum of indecomposables. 

Now any 1-morphism in $\H^\G$ is a direct summand of a sequence of $P_i$'s and $Q_i$'s. Using relations from Proposition \ref{prop:basicrel} (which were shown to hold for $\H^\G$ in Section \ref{sec:proofprop2}) it follows that any 1-morphism is a direct sum of (shifts of) 1-morphisms of the form $\prod_i P_i^{\lambda_i} \prod_i Q_i^{\mu_i}$. It remains to show that such endomorphisms are indecomposable. 

But this follows from Lemma \ref{lem:1} since $\prod_i P_i^{\lambda_i} \prod_i Q_i^{\mu_i}$ is a direct summand of $\prod_i P_i^{|\lambda_i|} \prod_i Q_i^{|\mu_i|}$ corresponding to the minimal idempotent $\prod_i e_{\lambda_i,1} \prod_i e_{\mu_i,1} \in \otimes_i \k[S_{|\lambda_i|}] \otimes_i \k[S_{|\mu_i|}]$. 
\end{proof}

\begin{Remark}\label{rem:1} 
Arguing as above it is straightforward to check that the 1-morphisms $P^\lambda Q^\mu$ for all $I_\G$-colored partitions $\lambda$ and $\mu$ form a complete set of indecomposable 1-morphisms in $\H_\G$.  The classes of these indecomposable 1-morphisms in the Grothendieck group thus form a ``canonical basis'' of the Heisenberg algebra. This basis is analogous in many respects to the canonical bases of Kashiwara and Lusztig in the representation theory of quantum groups.
\end{Remark}

\begin{lemma}\label{lem:decomp} 
In $\H^\G$ we have $P_i^\lambda \cdot P_i^\mu \cong \oplus_\nu (P_i^\nu)^{\oplus c_{\lambda,\mu}^\nu}$ where $c_{\lambda,\mu}^\nu$ is the Littlewood-Richardson coefficient. In particular, we have 
$$P_i^{(n)} P_i^{(m)} \cong \oplus_{k=0}^{\mbox{min}(n,m)} P_i^{(n+m-k,k)} \text{ and } 
P_i^{(m)} P_j^{(1^n)} \cong P_i^{(m,1^n)} \oplus P_i^{(m+1,1^{n-1})}.$$
\end{lemma}
\begin{proof}
The statement follows from the representation theory of symmetric groups. Let $V_\lambda$ and $V_\mu$ be irreducible representations of $S_n$ and $S_m$, respectively. Then $V_\lambda \boxtimes V_\mu$ is an irreducible representation of $S_n\times S_m$.  This representation, when induced to $S_{n+m}$ decomposes as a direct sum of irreducible representations $V_\nu$, and the multiplicity of $V_\nu$ in this decomposition is the Littlewood-Richardson coefficient $c_{\lambda,\mu}^\nu$.  

This fact about representations of symmetric groups can be rephrased completely in terms of idempotents. If $e_\lambda \in\k[S_n]$, $e_\mu \in\k[S_m]$ and $e_{\nu} \in \k[S_{n+m}]$ are idempotents such that $V_\lambda \cong \k[S_n] e_\lambda$, $V_\mu \cong \k[S_m]e_\mu$, and $V_\nu \cong \k[S_{m+n}]e_\nu$, then when we have 
$$ \k[S_{n+m}] e_\lambda e_\mu \cong \oplus_s \k[S_{n+m}]e_s $$
for some minimal idempotents $\{e_s\}$ in $\k[S_{n+m}]$. The number of such $e_s$ such that $\k[S_{n+m}]e_s \cong \k[S_{n+m}] e_\nu$ is the Littlewood-Richardson coefficient $c_{\lambda,\mu}^\nu$. Under the embeddings
$$ \k[S_n]\otimes_\k\k[S_m] \subset \k[S_{n+m}] \hookrightarrow \End_{\H^\G}(P_i^{n+m})$$ 
this decomposition of idempotents implies the decomposition in the Lemma.
\end{proof}

\begin{cor}\label{cor:generate}
The classes $\{[P_i^{(m)}],[Q_j^{(n)}]: i,j \in I_\G, \ n,m \ge 0 \}$ generate $K_0(\H^\G)$ as an algebra.
\end{cor}
\begin{proof}
By Proposition \ref{prop:indec} above $K_0(\H^\G)$ is spanned by elements of the form $\prod [P_i^{\lambda_i}] \prod_i [Q_i^{\mu_i}]$. Thus it suffices to show that any $[P_i^\lambda]$ is generated by $\{[P_i^{(m)}]: m \ge 0\}$. This follows by induction using Lemma \ref{lem:decomp}. For instance, 
$$P_i^{2} = P_i^{(2)} \oplus P_i^{(1^2)}$$
so that $P_i^{(1^2)}$ is generated by $P_i$ and $P_i^{(2)}$. 
\end{proof}
\begin{Remark}
The expression of $[P_i^{\lambda}]$ in terms of $[P_i^{(m)}]$'s is given by the Giambelli identity expressing the Schur polynomial in terms of products of elementary symmetric functions.  Explicitly, 
$$ [P_i^{\lambda}] = \mbox{det}_{kl} [P_i^{(\lambda'_k + k-l)}] $$
where $\lambda' = (\lambda'_1\geq \lambda'_2\geq \dots)$ is the transpose of the partition $\lambda$.
\end{Remark}

In order to prove Proposition \ref{prop:Kgroups} and Theorem \ref{thm:main1}, we now consider the sequence 
$$\h_\G \xrightarrow{\pi} K_0(\H^\G) \xrightarrow{K_0(\eta)} K_0(\H_\G) \rightarrow \End(\oplus_{n\geq0} K_0(\mathcal{C}_n^\G)) \simeq \End(\mathcal{F}) $$
and show that both $\pi$ and $K_0(\eta)$ are isomorphisms.

First we check that $\pi$ is injective. This follows since the Fock space representation of $\h_\G$ is faithful, which in turn follows easily from its description as differential operators acting on a polynomial algebra, \cite{FJW1, FJW2}.  Thus the composition of all maps in the above sequence is injective, which implies that $\pi$ is injective.

Next we check that $\pi$ is surjective. Note that by Lemma \ref{cor:generate} $K_0(\H^\G)$ is generated as an algebra by the classes $[P_i^{(m)}]$ and $[Q_j^{(n)}]$. Since $\pi(p_i^{(m)}) = [P_i^{(m)}]$ and $\pi(q_j^{(n)}) = [Q_j^{(n)}]$ these generators are in the image of $\pi$. Thus $\pi$ is surjective.  

Since $\pi$ is surjective and the total composition injective it follows that $K_0(\eta)$ is injective. Finally it is easy to see $K_0(\eta)$ is surjective using Remark \ref{rem:1}. 
\begin{flushright}
$\square$
\end{flushright}

\section{Relation to Hilbert schemes}\label{sec:geometry}

We will now discuss how $\H^\G$ acts on the derived categories of coherent sheaves on Hilbert schemes of points of the corresponding ALE space $X_\G$.

We take $\k = \C$. If $\G \subset SL_2(\C)$ is finite then by the McKay correspondence the derived category of finite dimensional left $\Sym^* V^\vee \rtimes \G$-modules is isomorphic to the derived category $DCoh(X_\G)$ of coherent sheaves on the resolution $X_\G := \widehat{\C^2/\G}$ of the quotient $\C^2/\G$. To simplify notation we will denote the Hilbert scheme $\Hilb^n(X_\G)$ by $X_\G^{[n]}$.

Now by the results of Bridgeland, King and Reid \cite{BKR} the derived category $\catC_n^\G$ of $A_n^\G = (\Sym^* V^\vee \rtimes \G)^{\otimes n} \rtimes S_n$ modules is isomorphic to the derived category of coherent sheaves on the $S_n$-equivariant Hilbert scheme $\Hilb^{S_n}(X_\G^n)$. This Hilbert scheme parametrizes $S_n$-equivariant subschemes of $X_\G^n$ whose global sections are isomorphic to the regular representation of $S_n$. By work of Haiman \cite{H}, $\Hilb^{S_n}(X_\G^n)$ is isomorphic to the usual Hilbert scheme of points $X_\G^{[n]}$.

Hence Theorem \ref{thm:main2} gives a categorical Heisenberg action of $\H^\G$ on $\bigoplus_n DCoh(X_\G^{[n]})$ (if $\G = \C^\times$ then we should think of $\catC_n^\G$ as the derived category of coherent sheaves on the Hilbert scheme of the quotient stack $[\C^2/\C^\times]$). 

We now explain how this action of $\H^\G$ on $\bigoplus_n DCoh(X_\G^{[n]})$ looks like geometrically. This description is to some degree conjectural because we only sketch why the algebraically defined functors $\fP_i$ and $\fQ_i$ from Section \ref{sec:action} induce the geometric functors defined in this section. Checking the relations in the 2-category $\H^\G$ directly from the geometry (without using the algebraic description provided by BKR) is much more difficult. Two such computations are illustrated in subsections \ref{sec:temp1} and \ref{sec:temp2} below. These computations are analogous to those from \cite{CKL1} and \cite{CK} in the setting of categorical $\sl_n$ actions. Because of the technical difficulties in such geometric setups we chose to take a more algebraic approach in this paper. 

\subsection{The functors}

Inside $X_\G^{[n]} \times X_\G^{[n+1]}$ we have the natural subvariety consisting of ideal sheaves $(J_0, J_1)$ where $J_1 \subset J_0$. This subvariety also maps to $X_\G$ by taking the quotient $J_0/J_1$ and subsequently induces a functor $DCoh(X_\G^{[n]} \times X_\G) \rightarrow DCoh(X_\G^{[n+1]})$. Under the correspondence
$$DCoh(X_\G^{[n]}) \longleftrightarrow D(A_n^\G \dmod)$$
this functor is induced by the $(A_{n+1}^\G, A_n^\G \times A_1^\G)$-bimodule $A_{n+1}^\G$ (where we use the usual inclusion of $A_n^\G \otimes A_1^\G$ into $A_{n+1}^\G$).

Now, the fibre of $X_\G \rightarrow \C^2/\G$ over $0 \in \C^2/\G$ is a tree of $\P^1$'s. More precisely, there is one $E_i \cong \P^1$ for each $i \in I_\G$ except for the node which corresponds to the trivial $\G$-representation (we label this node $0$ and the corresponding minimal idempotent $e_0$). Moreover, $E_i$ and $E_j$ intersect at a point if $i,j \in I_\G$ are connected by an edge and do not intersect if $i,j$ are not connected.

Under the McKay correspondence the $A_1^\G$-module $\k[\G]e_i$ corresponds to the sheaf $\O_{E_i}(-1)$ while $\k[\G]e_0$ corresponds to the structure sheaf $\O_{\cup_i E_i}[-1]$ shifted into cohomological degree one. We will use the notation $E_0 := \cup_i E_i$.

Now the functor $\fP_i: \catC_n^\G \rightarrow \catC_{n+1}^\G$ in $\H^\G$ is induced by $A_{n+1}^\G \otimes_{A_n^\G \otimes A_1^\G} (A_n^\G \otimes \k[\G]e_i)$. It should follow that the corresponding functor $\fP_i: X_\G^{[n]} \rightarrow X_\G^{[n+1]}$ is induced by the kernels
$$\O_{P_i(n,n+1)} \otimes \pi^* \O_{\P^1}(-1) \text{ if } i \ne 0 \text{ and } \O_{P_0(n,n+1)}[-1] \text{ if } i = 0.$$
Here
$$P_i(n,n+1) := \{(J_0, J_1): \supp(J_0/J_1) \in E_i\} \subset X_\G^{[n]} \times X_\G^{[n+1]}$$
and $\pi: P_i(n,n+1) \rightarrow E_i$ is the natural map which remembers the support point $\supp(J_0/J_1)$. The functors $\fQ_i: X_\G^{[n+1]} \rightarrow X_\G^{[n]}$ are given by the adjoint of $\fP_i$ (up to a shift).

The difference between the left and right adjoints of $\fP_i$ is tensoring by the line bundle
\begin{equation}\label{eq:temp}
\pi_1^* \omega_{X_\G^{[n]}} [\dim X_\G^{[n]}] \otimes \pi_2^* \omega_{X_\G^{[n+1]}}^\vee [-\dim X_\G^{[n+1]}].
\end{equation}
Since $X_\G$ is holomorphic symplectic it follows that $X_\G^{[n]}$ is also symplectic and hence has trivial canonical bundle. Thus (\ref{eq:temp}) is the trivial line bundle with a shift of $[-2]$ (in other words $\fP_i^L \cong \fP_i^R [-2]$).

\subsubsection{The symmetric group}\label{sec:temp1}
To see the action of the symmetric group $S_k \subset \End(\fP_i^k)$ we must first compute the convolution product of the kernels $\O_{P_i(n,n+1)} \otimes \pi^* \O_{\P^1}(-1)$. The line bundle $\pi^* \O_{\P^1}(-1)$ plays no role so, for simplicity, we will omit it. The convolution $\O_{P_i(n+k-1,n+k)} * \dots * \O_{P_i(n,n+1)} \in DCoh(X_\G^{[n]} \times X_\G^{[n+k]})$ is given by
\begin{equation}\label{eq:temp2}
\pi_{0,k*}(\pi_{k-1,k}^* \O_{P_i(n+k-1,n+k)} \otimes \dots \otimes \pi_{0,1}^* \O_{P_i(n,n+1)})
\end{equation}
where $\pi_{i,j}$ is the projection $X_\G^{[n+k]} \times \dots \times X_\G^{[n]} \rightarrow X_\G^{[n+i]} \times X_\G^{[n+j]}$. One can check that the intersections of $\pi^{-1}_{l,l+1} P_i(n+l,n+l+1)$ for $l=0, \dots, k-1$ are of the expected dimension so (\ref{eq:temp2}) is isomorphic to $\pi_{0,k*}(\O_{\cap_l \pi^{-1}_{l,l+1} P_i(n+l,n+l+1)})$. Now
$$\cap_l \pi^{-1}_{l,l+1} P_i(n+l,n+l+1) = \{J_0 \supset J_1 \supset \dots \supset J_k: \supp(J_l/J_{l+1}) \in E_i \} \subset X_\G^{[n]} \times \dots \times X_\G^{[n+k]}.$$
The map $\pi_{0,k}$ forgets $J_1, \dots, J_{k-1}$. Generically, $\{\supp(J_l/J_{l+1}): l=0, \dots, k-1\}$ are $k$ distinct points on $E_i$. Thus $\pi_{0,k}$ is generically a cover of degree $k!$ and there is a natural action of $S_k$ which permutes the points. This action should extend to all of $\cap_l \pi^{-1}_{l,l+1} P_i(n+l,n+l+1)$ and induce an action of $S_k$ on the pushforward in (\ref{eq:temp2}).

\subsubsection{The commutator relation}\label{sec:temp2}
It is also possible to geometrically see the commutator relation $\fQ_i \fP_i \cong \fP_i \fQ_i \oplus \id[1] \oplus \id [-1]$. To do this let us compute the convolution
$$\O_{P_i(n,n+1)} * \O_{P_i(n,n+1)} \in DCoh(X_\G^{[n]} \times X_\G^{[n]})$$
which is the kernel corresponding to the composition $\fQ_i \fP_i$ (again we ignore line bundles for convenience). This is equal to 
$$\pi_{13}(\pi_{23}^* \O_{P_i(n,n+1)} \otimes \pi_{12}^* \O_{P_i(n,n+1)}) \cong \pi_{13*}(\O_{\pi_{23}^{-1}(P_i(n,n+1))} \otimes \O_{\pi_{12}^{-1}(P_i(n,n+1))}).$$
Now
$$\pi_{23}^{-1}P_i(n,n+1) \cap \pi_{12}^{-1}P_i(n,n+1) = \{J_0 \supset J_1 \subset J_0': \supp(J_0/J_1, J'_0/J_1) \in E_i\} \subset X_\G^{[n]} \times X_\G^{[n+1]} \times X_\G^{[n]}$$
has two components. One of them, denoted $\Delta_0$, is where $J_0 = J_0'$ while the other, denoted $\Delta_1$, is the closure of the locus where $J_0 \ne J_0'$. It is not hard to see that $\dim(\Delta_1) = 2n$ while $\dim(\Delta_0) = 2n+1$ which means that the intersection is of the expected dimension along $\Delta_1$ but not $\Delta_0$. The pushforward via $\pi_{13}$ is generically one-to-one along $\Delta_1$ and is a $\P^1$ fibration along $\Delta_0$. 

Notice that $\pi_{13}(\Delta_0)$ is the diagonal $\Delta \subset X_\G^{[n]} \times X_\G^{[n]}$. Now one should show that $\Delta_0$ contributes $\O_\Delta[1] \oplus \O_\Delta[-1]$ via the pushforward. The computation here is a more involved version of the calculation $\Ext^k_{X_\G}(\O_{E_i}, \O_{E_i}) \cong \begin{cases} \C \text{ if } k=0,2 \\ 0 \text{ otherwise } \end{cases}$ occuring on the surface $X_\G$ (which uses that the normal bundle of $E_i \cong \P^1 \subset X_\G$ is $\O_{\P^1}(-2)$). 

On the other hand, $\Delta_1$ contributes just $\O_{\pi_{13}(\Delta_1)}$. Finally, one can check via a direct convolution computation (which is simpler than the one above) that $\O_{\pi_{13}(\Delta_1)}$ is actually equal to the kernel which induces $\fP_i \fQ_i$. Since $\O_\Delta$ induces the identity functor the commutation relation follows. 

\subsection{Nakajima's action on cohomology}

The fundamental class of the variety $P_i(n,n+1)$ induces one of Nakajima's Heisenberg operators on the cohomology of Hilbert schemes. The other Nakajima operators are given by fundamental classes of the varieties
$$ P_i(n,n+k) = \{(J_0, J_1): \supp(J_0/J_1) \text{ is supported at a single point of } E_i\} \subset   X_\G^{[n]} \times X_\G^{[n+k]}.$$
Nakajima's theorem \cite{N1} is that fundamental classes of the $P_i(n,n+k)$ (and their Poincar\'e duals) satisfy the defining relations of the ``standard'' Heisenberg generators $a_i(k)$ from section \ref{sec:hei}. However, when $k>1$ the varieties $P_i(n,n+k)$ are singular, a fact which perhaps accounts for some of the difficulty of lifting the Nakajima action on cohomology to K-theory. In our categorical approach, the Heisenberg generators which are categorified are not the standard generators $a_i(k)$ but rather the generators $p_i^{(k)}$ which are categorified by the direct summands $\fP_i^{(k)}$ of $\fP_i^k$. In other words, $\fP_i^{(k)}$ is the image of $e_{triv} \in S_k \subset \End(\fP_i^k)$ where $e_{triv}$ is the idempotent inside $\C[S_k]$ corresponding to the trivial partition. 

We do not know a direct, geometric description of the kernel which induces $\fP_i^{(k)}$. The only geometry we can see is that involving $\fP_i^k$ and the action of $S_k$ on it. One can probably guess the support of the kernel for $\fP_i^{(k)}$ but it is going to be more complicated than just $P_i(n,n+k)$. 

The difficulty in explicitly identifying $\fP_i^{(k)}$ is probably related to the fact that it seems very hard to categorify the standard generators $a_i(k)$. By this we mean two things. First, we do not know if $\O_{P_i(n,n+k)}$ induce an action of the standard generators $a_i(k)$ on derived categories of coherent sheaves (we suspect that they do not). Secondly, we do not know categorical constructions of the expressions in section \ref{sec:h_Gdef} relating the generators $p_i^{(k)}$ and $a_i(k)$. Ideally, one would like some complex of functors involving $\fP_i^{(k)}$ which give categorify $a_i(k)$. 

\begin{Remark} There is a graded version of this whole story, given by the equivalence
$$D(A_n^\G \dgmod) \longleftrightarrow DCoh^{\C^\times}(X_\G^{[n]})$$
where the right side is the derived category of $\C^\times$ equivariant sheaves. Here the action of $\C^\times$ on $X_\G^{[n]}$ is induced from the diagonal action of $\C^\times$ on $\C^2$.
\end{Remark}

\section{An abelian 2-representation of $\H_\G$}\label{sec:koszul}

We now describe another 2-representation of $\H_\G$ which is closely related to the 2-representation of section \ref{sec:action} on $\oplus_n D(A_n^\G \dgmod)$. More precisely we replace $A^\G_n$ with its Koszul dual 
$$B_n^\G := (B^\G \otimes \dots \otimes B^\G) \rtimes S_n.$$ 
(The algebra $B_n^\G$ can also be written in smash product notation as $({B^\G})^{\otimes n}\# \k[S_n]$.)  The algebras $B_n^\G$ are superalgebras, and it turns out that the Heisenberg category $\H_\G$ acts on the module category $\oplus_n B_n^\G \dgmod$. Moreover, this action is non-derived and is therefore, in some sense, simpler than the action from section \ref{sec:action}. However, we have emphasized the action in \ref{sec:action} because of its closer relationship to geometry.

After recalling some basic definitions of superalgebras and modules over them, we will define all the generating 1 and 2-morphisms.  The verification of relations, which is straightforward, will be left to the reader.

\subsection{Superalgebras and supermodules}
Let $B$ be a superalgebra.  A left module over $B$ is a supervector space $W$ which is a left $B$ module in the usual sense.  A right supermodule is a supervector space which is a right $B$ module in the usual (non-super) sense, after ignoring the $\Z_2$ gradings on $B$ and on $W$.

Given two superalgebras $B_1$ and $B_2$, we form the tensor product superalgebra $B_1\otimes B_2$ with multiplication
$$
	(a\otimes b)(a'\otimes b') = (-1)^{|b||a'|}(aa'\otimes bb').
$$
There is an isomorphism of superalgebras
$$
	B_1\otimes B_2 \cong B_2\otimes B_1, \ \ a\otimes b \mapsto (-1)^{|a||b|}b\otimes a.
$$

If $V$ and $W$ are left supermodules for $B_1$ and $B_2$ respectively, then $W_1\otimes W_2$ is a left supermodule for the superalgebra $B_1\otimes B_2$.  The action of $B_1\otimes B_2$ on 
$W_1\otimes W_2$ is
$$
	a\otimes b \cdot (v\otimes w) = (-1)^{|b||v|} (a\cdot v)\otimes (b\cdot w).
$$

A morphism of left $B$ supermodules $f: W_1\rightarrow W_2$ is a map of vector spaces such that
$$
	f(a\cdot v) = (-1)^{|f||a|} f(v), \ \ v\in W_1, \ a\in B.
$$
Here $|f|$ is the degree of $f$ considered as a map of super vector spaces.
In contrast, a morphism of right $B$ supermodules $g: W_1\rightarrow W_2$ is a map of vector spaces such that
$$
	g(v\cdot a) = g(v)\cdot a, \ \ v\in W_1, \ a\in B.
$$
Note there is no sign here.

If $f:W_1\rightarrow W_2$ is a morphism of left $B_1$ supermodules and $g:W'_1\rightarrow W'_2$ is a morphism of left $B_2$ supermodules, then we write $f\otimes g$ as the induced morphism of left $B_1\otimes B_2$ modules 
$$ (f \otimes g): W_1\otimes W'_1\longrightarrow W_2\otimes W'_2. $$
Explicitly, we have
\begin{equation}\label{eq:tensor}
(f \otimes g) (v\otimes w) = (-1)^{|g||v|} f(v)\otimes g(w).
\end{equation}

A $(B_1,B_2)$ superbimodule $W$ is a left $B_1$ supermodule which is a right $B_2$ supermodule for the same grading, such that
$$
	a\cdot (v\cdot b) = (a\cdot v)\cdot b.
$$
A morphism of superbimodules is a map which is a morphism of left and right supermodules.
For the remainder of this section, all objects will be assumed super though we will often omit writing the prefix ``super" from discussions of algebras, modules and bimodules.

\subsection{The algebras $B_n^\G$}
Recall the algebra
$$B_n^\G := (B^\G \otimes \dots \otimes B^\G) \rtimes S_n,$$
which is the Koszul dual of $A_n^\G$.  $B_n^\G$ contains both $\k[S_n]$ and ${B^\G}^{\otimes n}$ as subalgebras, and the action of $S_n$ on ${B^\G}^{\otimes n}$ is by superpermutations: if $s_k \in S_n$ is the simple transposition $(k,k+1)$, then
$$ s_k \cdot (b_1\otimes \dots \otimes b_k \otimes b_{k+1}\otimes \dots \otimes b_n) = (-1)^{|b_k||b_{k+1}|} b_1\otimes \dots \otimes b_{k+1}\otimes b_k \otimes \dots \otimes b_n.
$$
\subsection{Functors $\fP$ and $\fQ$}

Consider the natural inclusion $B_n^\G \hookrightarrow B_{n+1}^\G$. We let $P^\G(n)$ and $(n)Q^\G$ denote $B_{n+1}^\G \{1\}$ and $B_{n+1}^\G$ thought of as a $(B_{n+1}^\G, B_n^\G)$ bimodule and $(B_n^\G,B_{n+1}^\G)$ bimodule respectively. 

Now $\k[S_{n+1}]$ is a flat $\k[S_n]$ module where the module structure is defined via the natural inclusion $\k[S_n] \rightarrow \k[S_{n+1}]$. Subsequently $B_{n+1}^\G$ is a flat $B_n^\G$ module. It follows that we can define functors 
$$\fP(n): B_n^\G \dgmod \longrightarrow B_{n+1}^\G \dgmod \text{ and } (n)\fQ: B_{n+1}^\G\dgmod \longrightarrow B_n^\G \dgmod$$
between abelian categories (instead of derived categories) using
$$\fP(n)(\cdot) := P^\G(n) \otimes_{B_n^\G} (\cdot) \text{ and } (n)\fQ(\cdot) := (n)Q^\G \otimes_{B_{n+1}^\G} (\cdot).$$
As before we usually omit the $(n)$ to simplify notation.

\subsection{Natural Transformations}

\subsubsection{Definition of $X$}

For any $n \geq 1$ there is an inclusion $B_1^\G \hookrightarrow B_{n}^\G$ given by sending an element of $B_1^\G$ to the last factor of the product
$$ a \mapsto (1 \otimes 1 \otimes \dots \otimes 1 \otimes a) \in B_n^\G. $$
The image of this inclusion supercommutes with the subalgebra $B_{n-1}^\G \subset B_n^\G$. We henceforth consider $B_1^\G$ as a subalgebra of $B_n^\G$ via this embedding.  For homogeneous $b \in B_1^\G$ we define 
$$X(b) : \ \ B^\G_n \longrightarrow B^\G_n \ \ w \mapsto  (-1)^{|w||b|}w b$$
using right multiplication by $b \in B_1^\G \subset B_{n}^\G$. 

The map $X(b)$ is a map of $(B_n^\G,B_{n-1}^\G)$ bimodules. We thus get a well-defined natural transformation
$$X(b) : \fP \longrightarrow \fP \{ \deg b \}.$$

\subsubsection{Definition of $T$}

Notice that $\fP\fP: B_n^\G \dgmod \rightarrow B_{n+2}^\G \dgmod$ is induced by the $(B_{n+2}^\G, B_n^\G)$ bimodule 
$$B_{n+2}^\G \{1\} \otimes_{B_{n+1}^\G} B_{n+1}^\G \{1\} \cong B_{n+2}^\G \{2\}.$$
So to define a degree zero map $T: \fP \fP \rightarrow \fP \fP$ we need a map $B_{n+2}^\G \rightarrow B_{n+2}^\G$ of $(B_{n+2}^\G, B_n^\G)$ bimodules. 
For this we use $w \mapsto w t_{n+1}$
where $t_{n+1}$ is the simple transposition $(n+1,n+2) \in S_{n+2}$. Since $t_{n+1}$ commutes with the subalgebra $B_n^\G \subset B_{n+2}^\G$ this gives a well defined map of $(B_{n+2}^\G, B_n^\G)$ bimodules. 

\subsubsection{Definition of $\adj: \fQ \fP \rightarrow \id\{-1\}$}

The identity functor on $B_n^\G \dgmod$ is given by tensoring with the $(B_n^\G, B_n^\G)$ bimodule $B_n^\G$. On the other hand, $\fQ \fP$ is induced by the bimodule
$$B^\G_{n+1} \otimes_{B_{n+1}^\G} B^\G_{n+1} \{1\}.$$  
As a $(B_n^\G,B_n^\G)$ bimodule, this is generated by the subalgebra $B_1^\G$ and the element $s_n$ (so any $(B_n^\G,B_n^\G)$ bimodule map is determined by where it sends $B_1^\G$ and $s_n$). We define a $(B_n^\G, B_n^\G)$ bimodule map
$$B^\G_{n+1} \otimes_{B_{n+1}^\G} B^\G_{n+1} \rightarrow B_n^\G \{-2\}$$
by $s_n \mapsto 0$, $1 \mapsto 0$, $ v_1 \mapsto 0$, $v_2 \mapsto 0$, and $\omega = v_1 \wedge v_2 \mapsto 1$ where $\{ 1, v_1, v_2, v_1 \wedge v_2 \}$ is the standard basis of $B_1^\G$.  This induces a map $\fQ \fP \rightarrow \id \{-1\}$.

\subsubsection{Definition of $\adj: \fP \fQ \rightarrow \id \{1\}$}
The algebra multiplication 
$$ B^\G_{n+1} \otimes_{B_n^\G} B^\G_{n+1} \rightarrow B^\G_{n+1} \ \ a \otimes b \mapsto ab $$
is a map of $(B_{n+1}^\G, B_{n+1}^\G)$ bimodules and induces the map $\adj: \fP \fQ \longrightarrow \id \{1\}.$

\subsubsection{Definition of $\adj: \id \rightarrow \fP \fQ \{1\}$}

The functor $\fP \fQ \{1\}$ is induced by the $(B_{n+1}^\G, B_{n+1}^\G)$ bimodule $B_{n+1}^\G \otimes_{B_n^\G} B_{n+1}^\G \{2\}.$ Now any $(B_{n+1}^\G, B_{n+1}^\G)$ bimodule map from $B_{n+1}^\G$ to this bimodule is determined by the image of $1 \in B_{n+1}^\G$. Moreover, an element $b \in B^\G_{n+1} \otimes_{B_n^\G} B^\G_{n+1}$ can be in the image of $1$ under such a bimodule map if and only if $b$ is a Casimir element of $B^\G_{n+1} \otimes_{B_n^\G} B^\G_{n+1}$ (i.e. if and only if $yb = (-1)^{|b||y|}by$ for all $y \in B_{n+1}^\G$).

We define the bimodule map $B^\G_{n+1} \longrightarrow B^\G_{n+1} \otimes_{B_n^\G} B^\G_{n+1} \{2\}$ by $1 \mapsto b_0 $ where 
\begin{eqnarray*}
b_0 &=& \sum_{b \in \mathcal{B}} \sum_{i=1}^{n+1} s_i \dots s_n b^\vee\otimes b s_n\dots s_i.
\end{eqnarray*}
where $\mathcal{B}$ is our basis of $B_1^\G \subset B_{n+1}^\G$ and (by convention) $s_i \dots s_n = s_n \dots s_i = 1$ when $i=n+1$ in the summation above. We leave it to the reader to check that $b_0$ is a Casimir element. This map induces $\adj: \id \rightarrow \fP \fQ \{1\}$.  

\subsubsection{Definition of $\adj: \id \rightarrow \fQ \fP \{-1\}$}

The functor $\fQ \fP \{-1\}$ is induced by the $(B_n^\G, B_n^\G)$ bimodule 
$$ B_{n+1}^\G \otimes_{B_{n+1}^\G} B_{n+1}^\G \cong B_{n+1}^\G.$$ 
The natural inclusion of $B_n^\G$ as a subalgebra of $B_{n+1}^\G$ is a map of $(B_n^\G, B_n^\G)$ bimodules and defines $\adj: \id \longrightarrow \fQ \fP \{-1\}$.

\subsection{The abelian 2-representation}

The following is the analogue of Theorem \ref{thm:main2}.

\begin{theorem}\label{thm:Koszul}
The functors $\fP$ and $\fQ$ together with natural transformations $X,T$, and $\adj$ defined above give a categorical action of $\H_\G$ on the abelian category $\oplus_n B_n^\G \dgmod$.
\end{theorem}

The proof of the above theorem is analogous to the proof of Theorem \ref{thm:main2} but easier since none of the functors or natural transformations are derived (in other words, there are no complexes involved and nowhere does one need to take resolutions). 

\subsection{Relation to wreath products $\k[\G^n \rtimes S_n]$}\label{sec:wreath}

The algebra $\k[\G^n \rtimes S_n]$ embeds as a subalgebra (concentrated in degree zero) of $B_n^\G$.  This embedding gives rise to induction and restriction functors
$$\Ind_n : \k[\G^n \rtimes S_n] \dgmod \rightarrow B_n^\G \dgmod$$
$$\Res_n: B_n^\G \dgmod \rightarrow  \k[\G^n \rtimes S_n] \dgmod$$
between categories of finite dimensional (over $\k$) graded (super) modules. These functors are (up to shift) left and right adjoint to each other. 

Denote by $K_0(\k[\G^n \rtimes S_n])$ and $K_0(B_n^\G)$ the complexified Grothendieck groups of these categories and by $K_P(\k[\G^n \rtimes S_n]])$ and $K_P(B_n^\G)$ the subgroups generated by projective modules. Of course, since the characteristic of $\k$ is zero all $\k[\G^n \rtimes S_n]$ modules are projective and $K_0(\k[\G^n \rtimes S_n]) = K_P(\k[\G^n \rtimes S_n])$.  

At the same time $\k[\G^n \rtimes S_n]$ is also a quotient of $B_n^\G$ via the surjection which kills all elements of positive degree. This surjection induces an extension by zero functor
$$ \mbox{ex}_n:  \k[\G^n \rtimes S_n] \dgmod \longrightarrow B_n^\G \dgmod$$
whose left adjoint is the functor 
$$ \mbox{ss}_n:  B_n^\G \dgmod \longrightarrow \k[\G^n \rtimes S_n] \dgmod$$
which takes a $B_n^\G$ module to its maximal semisimple quotient considered naturally as a $\k[\G^n \rtimes S_n]$ module. These functors have the following properties: 
\begin{itemize}
\item the functors $\Res_n$ and $\mbox{ex}_n$ induce isomorphisms $ K_0(B_n^\G) \cong K_0(\k[\G^n \rtimes S_n])$
\item the functors $\Ind_n$ and $\mbox{ss}_n$ induce isomorphisms $K_P(B_n^\G) \cong K_P(\k[\G^n \rtimes S_n])$
\end{itemize}
The functor $\mbox{ss}_n$ does not have finite homological dimension so we restrict to the projective Grothendieck group in the second isomorphism. Therefore, as a corollary to Theorem \ref{thm:Koszul} we obtain the following:

\begin{cor}\label{cor:wreath}
The Heisenberg algebra $\h_\G$ acts on $\oplus_n K_0(\k[\G^n \rtimes S_n])$ and on $\oplus_n K_P(\k[\G^n \rtimes S_n])$.  
\end{cor}

This recovers a theorem of Frenkel-Jing-Wang from \cite{FJW1}. However, this action does not appear to come from a categorical action of $\H_\G$ on the categories $\oplus_n \k[\G^n \rtimes S_n] \dgmod$ (in part because the 2-morphisms in $\H_\G$ have non-trivial gradings). Indeed the functors constructed in \cite{FJW1} give rise to a \emph{non-deformed} Heisenberg algebra action (given by setting $t=1$ in the definition of $\h_\G$) on $\oplus_n K_0(\k[\G^n \rtimes S_n])$. 

In \cite{FJW2} Frenkel-Jing-Wang obtain quantum Heisenberg actions by replacing $\G$ with $\G \times \k^\times$. This corresponds to considering graded versus ungraded $A_n^\G$-modules in this section as well as Section \ref{sec:action}. 

\section{Other remarks and open problems}\label{sec:open}

We would like to end by making some general remarks and discussing what we consider to be some interesting unanswered questions. 

\subsection{Relationship to Khovanov's Heisenberg categorification and work of Shan-Vasserot}\label{sec:khovanov}

Our 2-category $\H_\G$ takes as its input data the embedded finite subgroup $\G \subset SL_2(\C)$, and the definition of $\H_\G$ can be generalized to an arbitrary embedded finite group $\G \subset SL_n(\C)$. The associated Heisenberg 2-category will then act naturally on $\oplus_n D(A_n^\G \dmod)$ where $A_n^\G$ is defined as above except that now $V$ is the standard representation of $SL_n(\C)$. In fact, such a construction already makes sense when $\G$ is trivial and $V=0$. In this case the definition of the associated Heisenberg category is due to Khovanov \cite{K}. 

In Khovanov's construction, like ours, the 1-morphisms are generated by two elements $P$ and $Q$ which are biadjoint up to shift. The difference appears in the structure of 2-morphisms. Here the r\^oles of $\G$ and $V$ become apparent. In particular, our categories inherit a non-trivial $\Z$ grading from the natural grading on the exterior algebra $\Lambda^*(V)$. This grading makes the Grothendieck group of $\H_\G$ into a $\Z[t,t^{-1}]$ module which categorifies a quantum deformation of the Heisenberg algebra.

It would be interesting to formulate a deeper relationship between our categorification and his. One possible source for such a relationship comes from the representation theory of Hecke algebras at roots of unity. In particular, in \cite{LS} the second author, together with Alistair Savage, defined a $q$-deformation of Khovanov's Heisenberg categorification, recovering Khovanov's definition at $q=1$. At roots of unity, the Grothendieck group of the categorification of \cite{LS} changes. Moreover, if $q = e^{\frac{2\pi i}{k}}$, the associated Heisenberg category acts on $\bigoplus_{n \ge 0} H_n(q) \dmod$ where $H_n$ is the Hecke algebra of the symmetric group $S_n$. The category $\bigoplus_{n \ge 0} H_n(q) \dmod$ was used earlier by Ariki to categorify the basic representation of the affine Lie algebra $\widehat{\sl_k}$ \cite{A}. From this point of view, it is tempting to conjecturally identify Khovanov's Heisenberg categorification (after deforming as in \cite{LS} and setting $q = e^{\frac{2 \pi i}{k}}$) with a categorification of the \emph{principal} Heisenberg subalgebra of the affine Lie algebra $\widehat{\sl_k}$.

On the other hand, in \cite{CL} we describe a precise relationship between $\H^\G$ where $\G = \Z/k\Z$ is cyclic and the quantum affine algebra $\widehat{\sl_k}$. Namely, $\H^\G$ categorifies the \emph{homogeneous} Heisenberg subalgebra of $\widehat{\sl_k}$. Thus, at least in type A, a concrete relationship between our categorification and Khovanov's might come from the 2-representation theory of the affine Lie algebra $\widehat{\sl_k}$.  

More recently, categorical Heisenberg actions have also appeared in the theory of cyclotomic rational double affine Hecke algebras (rDAHA). In \cite{ShV}, Shan and Vasserot describe certain induction and restriction functors between these rDAHAs which are also defined using a cyclic subgroup $\G \subset SL_2(\C)$. In fact, these rDAHAs are deformations of our algebras $A_n^\G$ so the existence of a Heisenberg action should not be entirely surprising. 

On the other hand, as far as we understand, Shan and Vasserot do not compute the categorical commutator relations between their $P$ and $Q$ (they only know the commutator relation at the level of Grothendieck groups). Moreover, the type of Heisenberg they seem to get is a higher level version of Khovanov's. This means that they have generators $P$ and $Q$ (but no $P_i$ and $Q_i$ for $i \in I_\G$) with $[P,Q]$ equal to several copies of the identity. 

After passing to the Grothendieck group their Heisenberg corresponds to the diagonal Heisenberg inside $\widehat{\gl_k}$ whereas in this paper we consider the homogeneous Heisenberg inside $\widehat{\sl_k}$. These two combine to give the homogeneous Heisenberg inside $\widehat{\gl_k}$. So it is likely that the construction in \cite{ShV} is complementary to our work. 

There is a large amount of work still to be done in this area and we are happy to bring to the attention of the reader the problem of relating all the various categorical appearances of the Heisenberg algebra. 

\subsection{Relation to Kac-Moody 2-categories $\mathcal{U}(\g)$}

When $\g$ is the affine Lie algebra associated to $\G$ we expect a direct relationship between the 2-categories $\H^\G$ and $\mathcal{U}(\g)$. This expectation is motivated by the existence of two presentations of the affine Lie algebra $U(\g)$, one as an extension of the loop algebra of the corresponding finite dimensional Lie algebra and one as a Kac-Moody Lie algebra. 

These two presentations should have two distinct categorifications. The 2-categories $\mathcal{U}(\g)$ from \cite{KL3,R} categorify the Kac-Moody presentation. A second ``loop categorification'' should come from combining the 2-category associated with the finite dimensional Lie algebra with our Heisenberg 2-category $\H^\G$. One would like an equivalence between these categorifications which lifts the isomorphism between the two different presentations. We will study such a relationship in \cite{CL}. 

\subsection{Degenerate affine Hecke algebras}

In $H^\G$ we will also use hollow dots as shorthand for drawing right-twist curls. Subsequently we obtain the following relations.

$$
\begin{tikzpicture}[>=stealth]
\draw (0,0) -- (0,2)[->][very thick];
\draw [red] (0,1.33) circle (4pt);
\draw (0,.44) node [anchor=east] [black] {$i$};
\draw (0,.88) node [anchor=east] [black] {$j$};
\filldraw [blue] (0,.66) circle (2pt);
\draw (.5,1) node {=};
\draw (1,0) -- (1,2)[->][very thick];
\draw [red] (1,.66) circle (4pt);
\filldraw [blue] (1,1.33) circle (2pt);
\draw (1,1.11) node [anchor=west] [black] {$i$};
\draw (1,1.55) node [anchor=west] [black] {$j$};

\draw [shift={+(5,0)}](0,0) -- (0,2)[->][very thick];
\draw [shift={+(5,0)}][red] (0,1.33) circle (4pt);
\draw [shift={+(5,0)}](0,.44) node [anchor=east] [black] {$i$};
\draw [shift={+(5,0)}](0,.88) node [anchor=east] [black] {$i$};
\filldraw [shift={+(5,0)}][blue] (0,.66) circle (2pt);
\draw [shift={+(5,0)}](.5,1) node {=};
\draw [shift={+(5,0)}](1,0) -- (1,2)[->][very thick];
\draw [shift={+(5,0)}] [red] (1,.66) circle (4pt);
\filldraw [shift={+(5,0)}][blue] (1,1.33) circle (2pt);
\draw [shift={+(5,0)}](1,1.11) node [anchor=west] [black] {$i$};
\draw [shift={+(5,0)}](1,1.55) node [anchor=west] [black] {$i$};
\end{tikzpicture}
$$
where $i,j \in I_\G$ are joined by an edge. The hollow dots have an interesting ``affine Hecke" type relation with crossings:
$$
\begin{tikzpicture}[>=stealth]
\draw (6.1,-0.2) node [anchor=east] [black] {$i$};
\draw (5.4,-0.2) node [anchor=east] [black] {$i$};
\draw (4.4,-0.2) node [anchor=east] [black] {$i$};
\draw (3.7,-0.2) node [anchor=east] [black] {$i$};
\draw (2.7,-0.2) node [anchor=east] [black] {$i$};
\draw (1.7,-0.2) node [anchor=east] [black] {$i$};
\draw (0.2,-0.2) node [anchor=east] [black] {$i$};
\draw (1.2,-0.2) node [anchor=east] [black] {$i$};
\draw [->](0,0) -- (1,1) [very thick];
\draw [->](1,0) -- (0,1) [very thick];
\draw [red](.25,.25) circle (4pt);
\draw (1.25,.5) node{$=$};
\draw [->](1.5,0) -- (2.5,1) [very thick];
\draw [->](2.5,0) -- (1.5,1) [very thick];
\draw [red](2.25,.75) circle (4pt);
\draw (3,.5) node{$+$};
\draw [shift={+(3.5,0)}][->](0,0) -- (0,1) [very thick];
\draw [shift={+(3.5,0)}][->](.75,0) -- (.75,1) [very thick];
\filldraw [shift={+(3.5,0)}][blue](0,.5) circle (2pt);
\draw (4.7,.5) node{$+$};
\draw [shift={+(5.2,0)}][->](0,0) -- (0,1) [very thick];
\draw [shift={+(5.2,0)}][->](.75,0) -- (.75,1) [very thick];
\filldraw [shift={+(5.2,0)}][blue](.75,.5) circle (2pt);
\end{tikzpicture}
$$

$$
\begin{tikzpicture}[>=stealth]
\draw (6.1,-0.2) node [anchor=east] [black] {$i$};
\draw (5.4,-0.2) node [anchor=east] [black] {$i$};
\draw (4.4,-0.2) node [anchor=east] [black] {$i$};
\draw (3.7,-0.2) node [anchor=east] [black] {$i$};
\draw (2.7,-0.2) node [anchor=east] [black] {$i$};
\draw (1.7,-0.2) node [anchor=east] [black] {$i$};
\draw (0.2,-0.2) node [anchor=east] [black] {$i$};
\draw (1.2,-0.2) node [anchor=east] [black] {$i$};
\draw [->](0,0) -- (1,1) [very thick];
\draw [->](1,0) -- (0,1) [very thick];
\draw [red](.25,.75) circle (4pt);
\draw (1.25,.5) node{$=$};
\draw [->](1.5,0) -- (2.5,1) [very thick];
\draw [->](2.5,0) -- (1.5,1) [very thick];
\draw [red](2.25,.25) circle (4pt);
\draw (3,.5) node{$+$};
\draw [shift={+(3.5,0)}][->](0,0) -- (0,1) [very thick];
\draw [shift={+(3.5,0)}][->](.75,0) -- (.75,1) [very thick];
\filldraw [shift={+(3.5,0)}][blue](0,.5) circle (2pt);
\draw (4.7,.5) node{$+$};
\draw [shift={+(5.2,0)}][->](0,0) -- (0,1) [very thick];
\draw [shift={+(5.2,0)}][->](.75,0) -- (.75,1) [very thick];
\filldraw [shift={+(5.2,0)}][blue](.75,.5) circle (2pt);
\end{tikzpicture}
$$

$$
\begin{tikzpicture}[>=stealth]
\draw (4.8,-0.2) node [anchor=east] [black] {$j$};
\draw (4.1,-0.2) node [anchor=east] [black] {$i$};
\draw (2.7,-0.2) node [anchor=east] [black] {$j$};
\draw (1.7,-0.2) node [anchor=east] [black] {$i$};
\draw (0.2,-0.2) node [anchor=east] [black] {$i$};
\draw (1.2,-0.2) node [anchor=east] [black] {$j$};
\draw [->](0,0) -- (1,1) [very thick];
\draw [->](1,0) -- (0,1) [very thick];
\draw [red](.25,.25) circle (4pt);
\draw (1.25,.5) node{$=$};
\draw [->](1.5,0) -- (2.5,1) [very thick];
\draw [->](2.5,0) -- (1.5,1) [very thick];
\draw [red](2.25,.75) circle (4pt);
\draw (3,.5) node{$+ \hspace{.2cm} \epsilon_{ij}$};
\draw [shift={+(3.9,0)}][->](0,0) -- (0,1) [very thick];
\draw [shift={+(3.9,0)}][->](.75,0) -- (.75,1) [very thick];
\filldraw [shift={+(3.9,0)}][blue](0,.66) circle (2pt);
\filldraw [shift={+(3.9,0)}][blue](.75,.33) circle (2pt);
\end{tikzpicture}
$$

$$
\begin{tikzpicture}[>=stealth]
\draw (4.8,-0.2) node [anchor=east] [black] {$j$};
\draw (4.1,-0.2) node [anchor=east] [black] {$i$};
\draw (2.7,-0.2) node [anchor=east] [black] {$j$};
\draw (1.7,-0.2) node [anchor=east] [black] {$i$};
\draw (0.2,-0.2) node [anchor=east] [black] {$i$};
\draw (1.2,-0.2) node [anchor=east] [black] {$j$};
\draw [->](0,0) -- (1,1) [very thick];
\draw [->](1,0) -- (0,1) [very thick];
\draw [red](.25,.75) circle (4pt);
\draw (1.25,.5) node{$=$};
\draw [->](1.5,0) -- (2.5,1) [very thick];
\draw [->](2.5,0) -- (1.5,1) [very thick];
\draw [red](2.25,.25) circle (4pt);
\draw (3,.5) node{$+ \hspace{.2cm} \epsilon_{ij}$};
\draw [shift={+(3.9,0)}][->](0,0) -- (0,1) [very thick];
\draw [shift={+(3.9,0)}][->](.75,0) -- (.75,1) [very thick];
\filldraw [shift={+(3.9,0)}][blue](0,.33) circle (2pt);
\filldraw [shift={+(3.9,0)}][blue](.75,.66) circle (2pt);
\end{tikzpicture}
$$

These relations are similar to the relations in the degenerate affine Hecke algebra of the symmetric group and its wreath products, which suggests that the endomorphism algebra of $P_i^n$ (or more generally the endomorphism algebra of $P_{i_1}^{n_1}\dots P_{i_k}^{n_k}$) is an interesting algebraic object in its own right.  It is perhaps worthwhile to express at least one of these algebras in more standard algebraic notation. 

We thus define the graded $\k$-algebra $H^n_i$ as follows. The generators of $H_i^n$ are elements $y_k,z_k$ for $k=1,\dots,n$ and $t_l$ for $1 \le l \le n-1$ where
$$\deg(y_k) = \deg(z_k) = 2 \text{ and } \deg(t_l) = 0.$$
The relations are as follows. Firstly, the $\{t_l\}_{l=1}^{n-1}$ generate a subalgebra isomorphic to the group algebra $\k[S_n]$ of the symmetric group, with $t_l$ equal to the simple transposition $(l,l+1)$.  The 
$\{y_k\}_{k=1}^n$ generate a symmetric algebra $\Sym(y_1,\dots,y_n)$, and $\{z_k\}_{k=1}^n$ generate an exterior algebra $\Lambda(z_1,\dots,z_n)$.  The remaining relations are
\begin{eqnarray*}
y_kz_l = z_ly_k \text{ for all } k,l\\
t_ky_l = y_l t_k \text{ for } l\neq k,k+1,\\
t_kz_l = z_lt_k \text{ for } l \neq k,k+1,\\
t_kz_{k+1} =z_kt_k, \\
t_kz_k = z_{k+1}t_k,\\
t_ky_{k+1} = y_kt_k + z_k + z_{k+1},\\
y_{k+1}t_{k} = t_ky_k + z_k + z_{k+1}.
\end{eqnarray*} 

Note that there is a natural map of $\k$-algebras $H_i^n \longrightarrow \End_{\H^\G}(P_i^{n})$ taking $t_k$ to a crossing between the $k$ and $k+1$st strands, $y_k$ to a hollow dot on the $k$th strand, and $z_k$ to a solid dot on the $k$th strand.  

\begin{conj}\label{conj:injective}
The natural map of $\k$-algebras $H_i^n \longrightarrow \End_{\H^\G}(P_i^{n})$ is injective.
\end{conj}

\begin{Remark}\label{rem:injective}
In section \ref{sec:koszul} the functor $P^n$ is induced by the $(B_k^\G, B_{k+n}^\G)$ bimodule $B_{n+k}^\G \{n\}$. Then it is clear from the above definitions of $X(a)$ and $T$ that the map $b \in B_n^\G$ acts on $B_{n+k}^\G \{n\}$ by right multiplication. Subsequently the map $ B_n^\G \longrightarrow \End_{\H_\G}(P^n) $ is injective. When $\G$ is cyclic (i.e. when $\H_\G$ and $\H^\G$ agree), this means that the subalgebra of $H_i^n$ generated by the $t$'s and $z$'s injects into $\End_{\H^\G}(P_i^{n})$. 

The more complicated part of the conjecture above involves the hollow dots, which generate a commutative subalgebra of $\End_{\H^\G}(P^n)$.  Under the analogy between $H_i^n$ and the degenerate affine Hecke algebra, it is natural to consider the image of hollow dots in representations of $\H^\G$ as ``Jucys-Murphy elements''.  We refer the reader to \cite{K} for a related discussion where the algebra $B_n^\G$ is replaced by the group algebra $\k[S_n]$ of the symmetric group, in which case the algebra $H_i^n$ is replaced by the degenerate affine Hecke algebra.
\end{Remark}

The space of 2-morphisms $\End(\id)$ from the empty diagram to itself is spanned by closed diagrams. Moreover, since one can slide closed diagrams past vertical lines, one can assume that all closed components appear to the right of all components with boundary. On the other hand, any closed diagram acts naturally on any space of 2-morphisms in $\H^\G$, since closed diagrams can be added to the right of a 2-morphism to produce a new 2-morphism with the same source and target. A more ambitious version of Conjecture \ref{conj:injective} is that, up to the action of $\End(\id)$, all endomorphisms of $P_i^n$ come from the algebra $H_i^n$:

\begin{conj}\label{conj:iso}
The natural map $H_i^n\otimes_\k \End_{\H^\G}(\id) \longrightarrow \End_{\H^\G}(P_i^n)$ is an isomorphism. 
\end{conj}

\subsection{The structure of $\End_{\H^\G}(\id)$} 

Conjecture \ref{conj:iso} reduces the problem of understanding $\End_{\H^\G}(P_i^{n})$ to the problem of understanding $\End_{\H^\G}(\id)$. An element of $\End_{\H^\G}(\id)$ can be simplified to a linear combination of unions of disjoint, oriented circles containing solid and hollow dots. In fact, with a little manipulation one can also write any counterclockwise oriented circle as a linear combination of products of clockwise oriented circles. 

Notice that if you have a circle labeled by $i \in I$ with a solid dot then you can split the dot into two solid dots corresponding to maps $P_i \rightarrow P_j \la 1 \ra \rightarrow P_i \la 2 \ra$ where $i,j \in I$ are joined by an edge in the Dynkin diagram. Moving one of the dots around the other side and joining back the two dots gives a circle labeled $j$ with a solid dot. Thus it is clear that the label on a circle containing a solid dot is irrelevant.  

For $i\in I$ and $k\in \N$ let
\begin{itemize}
\item $c_k^i$ be a clockwise circle labeled $i$ containing exactly $k$ hollow dots, and 
\item $c_k$ denote a clockwise circle containing one solid dot and $k$ hollow dots (the labeling of the circle $c_k$ is not important by the observation above).  
\end{itemize}
Notice that $\deg(c_k^i) = 2k+2$ and $\deg(c_k) = 2k+4$. 

\begin{conj} $\End_{\H^\G}(\id)$ is isomorphic to the symmetric algebra freely generated by $\{c_k^i\}_{k \ge 0, i \in I}$ and $\{c_k\}_{k \ge 0}$.
\end{conj}

\subsection{Hochschild cohomology of Hilbert schemes}

Let us work over $\k=\C$. Our categorical action of $\H^\G$ on $\oplus_n DCoh(X_\G^{[n]})$ gives a degree preserving map of rings
\begin{equation}\label{eq:HHmap}
\End_{\H^\G}(\id_n) \rightarrow HH^*(X_\G^{[n]})
\end{equation}
where $\id_n$ is the identity functor of the region labeled by $n \in \Z$ and $HH^*$ denotes the Hochschild cohomology. Unfortunately, this map is neither injective nor surjective. For example, $HH^1(X_\G)$ is nonzero because $X_\G$ carries nontrivial vector fields but $\End_{\H^\G}(\id_n)$ is zero in odd degrees. 

On the other hand, one can consider ``cyclotomic quotients'' of $\End_{\H^\G}(\id_n)$ much like those of quiver Hecke algebras considered in \cite{KL1,KL2,KL3,R}. For example, one such quotient $\overline{\End}_{\H^\G}(\id_n)$ is defined as the quotient by the ideal generated by diagrams which containing a region labeled by an integer $n < 0$.

{\bf Example.} $\overline{\End}_{\H^\G}(\id_0) \cong \C$ essentially because any counterclockwise circle is zero or the indentity and any other diagram will contain a region labeled by some $n < 0$. 

{\bf Example.} Likewise we have 
$$\overline{\End}_{\H^\G}(\id_1) \cong 
\begin{cases} 
\C \text{ in degree } 0, 4 \\ 
\C^{|\G|} \text{ in degree } 2 \\
0 \text{ in all other degrees}
\end{cases}$$
where the degree two maps are spanned by clockwise circles labeled some $i \in I_\G$. The product of any two such circles is equal to a clockwise circle containing a solid dot (this spans the space of degree 4 maps). 

The map (\ref{eq:HHmap}) factors through $\overline{\End}_{\H^\G}(\id_n)$. One would like to know if the resulting morphism has any interesting properties. Likewise, one can ask this same question if you replace $DCoh(X_\G^{[n]})$ by the category $B_n^\G \dgmod$ from section \ref{sec:koszul}. 

\appendix

\section{The case $\G=\Z_2$}\label{sec:z2}

The case $\G=\Z_2$ differs slightly from other non-trivial $\G \subset SL_2(\C)$ and the main definitions in the body of the paper need to be modified slightly. We collect the appropriate definitions for $\G=\Z_2$ in this section. Having made these definitions all of the results and proofs in the body of the paper carry over to $\G=\Z_2$ without difficulty. 

The vertex set $I_{\Z_2}=\{0,1\}$ consists of the two distinct isomorphism classes of representations of $\Z_2$, the trivial representation and the sign representation.  The quantum Cartan matrix $C_{\Z_2}$ is the matrix with entries $\la i,j \ra$ indexed by $i,j \in I_{\Z_2}$ given by 
$$ \la i,j \ra =
\begin{cases}
        t + t^{-1} & \text{ if } i = j \\
        -2 & \text{ if } i \ne j.
\end{cases} $$
The Heisenberg algebra $\h_{\Z_2}$ is the unital infinite dimensional $\k[t,t^{-1}]$ algebra with generators $p_i^{(n)}, q_i^{(n)}$ for $i \in I_{\Z_2}$ and integers $n \ge 0$, with relations
$$p_i^{(n)} p_j^{(m)} = p_j^{(m)} p_i^{(n)} \text{ and } q_i^{(n)}q_j^{(m)} = q_j^{(m)} q_i^{(n)} \text{ for any } i,j \in I_{\Z_2}$$
$$q_i^{(n)} p_j^{(m)} =  
\begin{cases}
\sum_{k \ge 0} [k+1] p_i^{(m-k)} q_i^{(n-k)} \text{ if } i=j \\
p_j^{(m)} q_i^{(n)} + 2 p_j^{(m-1)} q_i^{(n-1)} + p_j^{(m-2)} q_i^{(n-2)} \text{ if } i \ne j.
\end{cases}$$

The definition of the 2-category $\H'_{\Z_2}$ is the same as before (and $\H_{\Z_2}$ is still defined as the Karoubi envelope of $\H'_{\Z_2}$).

On the other hand, ${\H'}^{\Z_2}$ is slightly different. We fix an orientation $\epsilon$ of the Dynkin diagram. Since there are two edges connecting the two vertices $i$ and $j$ we describe the orientation $\epsilon$ by two functions $\epsilon_{ij}(0)$ and $\epsilon_{ij}(1)$. More precisely, $\epsilon_{ij}(0) = \pm 1$ depending on whether the first edge is oriented from $i$ to $j$ or from $j$ to $i$. Similarly, $\epsilon_{ij}(1)$ describes the orientation of the second edge.

As in the case $\G \neq \Z_2$ the objects of ${\H'}^{\Z_2}$ are indexed by the integers $\Z$ while the 1-morphisms are generated by $P_i$ and $Q_i$ ($i \in I_{\Z_2}$). 

The space of 2-morphisms is the $\k$-algebra generated by suitable planar diagrams modulo local relations. As before, the diagrams consist of oriented compact one-manifolds immersed into the plane strip $\R \times [0,1]$ (possibly with crossings) modulo isotopies fixing the boundary. 

Each such strand is labeled by some $i \in I_{\Z_2}$. It can carry a solid dot as before which represents a map $X: P_i \rightarrow P_i \la 2 \ra$. These $ii$ dots move freely on strands and through crossings just as in the case $\G \neq \Z_2$.

The main difference is that the strands now also carry a hollow rectangle (see below). More precisely, corresponding to each of the two edges in the Dynkin diagram connecting vertices $i$ and $j$ there is a 2-morphism $X: P_i \rightarrow P_j \la 1 \ra$ (one denoted by the solid dot and the other by the hollow rectangle).
$$
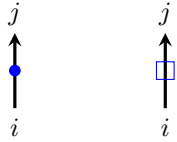
\begin{tikzpicture}[>=stealth]
\draw (0,0) -- (0,1) [->][very thick];
\draw (0,-.25) node {$i$};
\filldraw [blue] (0,.5) circle (2pt);
\draw (0,1.25) node {$j$};
\draw (2,0) -- (2,1) [->][very thick];
\draw (2,-.25) node {$i$};
\path (2,.5) node [shape=rectangle,blue, draw]{};
\draw (2,1.25) node {$j$};
\end{tikzpicture}
$$
These dots and rectangles are also allowed to move freely along the one-manifold along cups and caps and through crossings.  Dots or  rectangles on distinct strands anticommute when they move past each other relative to the vertical. The relations which govern compositions are as follows: 
\begin{equation}
\begin{tikzpicture}[>=stealth]
\draw (0,0) -- (0,2) [->][very thick];
\draw (0,-.25) node {$i$};
\filldraw [blue] (0,.66) circle (2pt);
\draw (-.25,1) node {$j$};
\filldraw [blue] (0,1.33) circle (2pt);
\draw (0,2.25) node {$i$};
\draw (.5,1) node {$=$};
\draw (1.33,1) node {$\epsilon_{ij}(0)$} ;
\draw (2,0) -- (2,2) [->][very thick];
\draw (2,-.25) node {$i$};
\filldraw [blue] (2,1) circle (2pt);
\draw (2,2.25) node {$i$};

\draw (5,0) -- (5,2) [->][very thick];
\draw (5,-.25) node {$i$};
\path (5,.66) node [shape=rectangle,blue, draw]{};
\draw (4.75,1) node {$j$};
\path (5,1.33) node [shape=rectangle,blue, draw]{};
\draw (5,2.25) node {$i$};
\draw (5.5,1) node {$=$};
\draw (6.33,1) node {$\epsilon_{ij}(1)$} ;
\draw (7,0) -- (7,2) [->][very thick];
\draw (7,-.25) node {$i$};
\filldraw [blue] (7,1) circle (2pt);
\draw (7,2.25) node {$i$};

\draw (10,0) -- (10,2) [->][very thick];
\draw (10,-.25) node {$i$};
\filldraw [blue] (10,.66) circle (2pt);
\draw (9.75,1) node {$j$};
\path (10,1.33) node [shape=rectangle,blue, draw]{};
\draw (10,2.25) node {$i$};
\draw (11,1) node {$=$};
\draw (12,0) -- (12,2) [->][very thick];
\draw (12,-.25) node {$i$};
\filldraw [blue] (12,1.33) circle (2pt);
\draw (11.75,1) node {$j$};
\path (12,.66) node [shape=rectangle,blue, draw]{};
\draw (12,2.25) node {$i$};
\draw (12.5,1) node {$=$};
\draw (13,1) node {$0.$};
\end{tikzpicture}
\end{equation}
Also, the composition of an $ii$ dot with any other dot or rectangle is zero. 

Next, we have the same relations (\ref{eq:rel1'}) and (\ref{eq:rel2'}) as well as (\ref{eq:rel4'}). The only more interesting relation is (\ref{eq:rel3'}) which is replaced by 
\begin{equation}\label{eq:rel3''}
\begin{tikzpicture}[>=stealth]
\draw (0,0) .. controls (1,1) .. (0,2)[<-][very thick];
\draw (1,0) .. controls (0,1) .. (1,2)[->] [very thick];
\draw (0,-.25) node {$i$};
\draw (1,-.25) node {$j$};
\draw (1.5,1) node {=};
\draw (2,0) --(2,2)[<-][very thick];
\draw (3,0) -- (3,2)[->][very thick];
\draw (2,-.25) node {$i$};
\draw (3,-.25) node {$j$};
\draw (3.75,1) node {$-\ \epsilon_{ij}(0)$};

\draw (4,1.75) arc (180:360:.5) [very thick];
\draw (4,2) -- (4,1.75) [very thick];
\draw (5,2) -- (5,1.75) [very thick][<-];
\draw (5,.25) arc (0:180:.5) [very thick];
\filldraw [blue] (4.5,1.25) circle (2pt);
\filldraw [blue] (4.5,0.75) circle (2pt);
\draw (5,0) -- (5,.25) [very thick];
\draw (4,0) -- (4,.25) [very thick][<-];
\draw (4,-.25) node {$i$};
\draw (5,-.25) node {$j$};
\draw (4,2.25) node {$i$};
\draw (5,2.25) node {$j$};
\draw (5.75,1) node {$-\ \epsilon_{ij}(1)$};

\draw (6,1.75) arc (180:360:.5) [very thick];
\draw (6,2) -- (6,1.75) [very thick];
\draw (7,2) -- (7,1.75) [very thick][<-];
\draw (7,.25) arc (0:180:.5) [very thick];
\path (6.5,1.25) node [shape=rectangle,blue, draw]{};
\path (6.5,.75) node [shape=rectangle,blue, draw]{};
%\filldraw [blue] (7,1.66) circle (2pt);
%\filldraw [blue] (7,0.33) circle (2pt);
\draw (7,0) -- (7,.25) [very thick];
\draw (6,0) -- (6,.25) [very thick][<-];
\draw (6,-.25) node {$i$};
\draw (7,-.25) node {$j$};
\draw (6,2.25) node {$i$};
\draw (7,2.25) node {$j$};
\end{tikzpicture}
\end{equation}

Finally, we assign a $\Z$ grading on the space of planar diagrams like before. In particular, we define the degree of both an $ij$ dot and an $ij$ rectangle to be one when $i \ne j$ and the degree of an $ii$ dot to be two. Equipped with these assignments all the graphical relations are graded, providing a $\Z$ grading on ${\H'}^{\Z_2}$. We then define $\H^{\Z_2}$ to be the Karoubi envelope of ${\H'}^{\Z_2}$.

\end{document}